\documentclass[12pt, a4paper]{amsart}
\usepackage[utf8]{inputenc}
\usepackage[english]{babel}
\usepackage[foot]{amsaddr}

\usepackage[margin=2.5cm]{geometry}
\usepackage{enumerate}
\usepackage{float}
\usepackage[final]{graphicx}
\usepackage{epstopdf} 

\usepackage{tabularx}
\newcolumntype{C}[1]{>{\centering\arraybackslash}m{#1}}

\usepackage{caption}
\usepackage{subcaption}
\usepackage{enumitem}
\usepackage{comment}
\usepackage{soul}

\usepackage{tikz-cd}
\usepackage{amsmath}
\usepackage{amsfonts}
\usepackage{amssymb}
\usepackage{amsthm}
\usepackage{hyperref}
\usepackage{cleveref}

\DeclareCaptionSubType[alph]{figure}
\usepackage[T1]{fontenc}
\usepackage{mathrsfs}
\usepackage{mathtools}

\usepackage{tikz}
\usetikzlibrary{decorations.markings}
\usetikzlibrary{decorations.pathreplacing}
\usetikzlibrary{arrows,shapes,positioning,patterns}
\usetikzlibrary{knots}
\tikzstyle{none}=[inner sep=0pt]
\pgfdeclarelayer{edgelayer}
\pgfdeclarelayer{nodelayer}
\pgfsetlayers{edgelayer,nodelayer,main}
\usepackage{url} 
\usepackage[all]{xy}

\newcommand*{\Scale}[2][4]{\scalebox{#1}{$#2$}}%

\newcommand{\bN}{\mathbb N}
\newcommand{\bK}{\mathbb K}

\newcommand{\bS}{\mathbb S}

\newcommand{\R}{\mathcal R}

\newcommand{\bZ}{\mathbb Z}

\newcommand{\Hmu}{\mathrm{H}_{\mu}}
\newcommand{\Cmu}{C_{\mu}}

\newcommand{\qdim}{\mathrm{qdim}}

\renewcommand{\H}{\mathrm{H}}

\newcommand{\sign}{\rm sgn}

\newcommand{\Cone}{{\mathfrak{Cone}~}}

\newcommand{\lgt}{\mathrm{lenght}}

\title{Combinatorial and Topological Aspects of Path Posets, and Multipath Cohomology}
\author{Luigi Caputi}
\address{LC: Institute of Mathematics, University of Aberdeen, AB24 3UE, UK
}
\email{luigi.caputi@abdn.ac.uk}
\address{ORCID id: https://orcid.org/0000-0002-6853-7651}
\author{Carlo Collari}
\address{CC: Dipartimento di Matematica, Universit\`a di Pisa, 56127 Pisa, IT
}
\email{carlo.collari.math@gmail.com}
\address{ORCID id: https://orcid.org/0000-0003-0034-0702}
\author{Sabino Di Trani}
\address{SDT (Corresponding Author): Dipartimento di Matematica , Universit\`a di Trento,  IT
}
\email{sabino.ditrani@unitn.it}
\address{ORCID id: https://orcid.org/0000-0002-6651-558X}

\theoremstyle{plain}
\newtheorem{thm}{Theorem}[section]
\newtheorem{prop}[thm]{Proposition}
\newtheorem{lem}[thm]{Lemma}
\newtheorem{cor}[thm]{Corollary}
\theoremstyle{remark}
\newtheorem{rem}[thm]{Remark}
\newtheorem*{construction}{Construction of $X(\tG)$}
\newtheorem{example}[thm]{Example} 
\newtheorem{q}[thm]{Question}

\theoremstyle{definition}
\newtheorem{defn}[thm]{Definition}

\newtheorem{crit}{Criterion}
\newtheorem*{crit*}{Criterion~B}

\definecolor{aquamarine}{rgb}{0.5, 1.0, 0.83}
\definecolor{princetonorange}{rgb}{1.0, 0.56, 0.0}
\definecolor{caribbeangreen}{rgb}{0.0, 0.8, 0.6}
\definecolor{bunired}{rgb}{0.8, 0.0, 0.0}
\definecolor{cdgreen}{rgb}{0.0, 0.42, 0.24}
\definecolor{lavender(floral)}{rgb}{0.71, 0.49, 0.86}
\definecolor{bluedefrance}{rgb}{0.19, 0.55, 0.91}

\newcommand{\tG}{{\tt G} }
\newcommand{\tH}{{\tt H} }
\newcommand{\tI}{{\tt I}}
\newcommand{\tD}{{\tt D}}
\newcommand{\tT}{{\tt T}}
\newcommand{\tP}{{\tt P}}
\newcommand{\tL}{{\tt L}}
\newcommand{\tA}{{\tt A}}

\newcommand{\Hasse}{{\tt Hasse}}

\begin{document}
\maketitle

\begin{abstract}
Multipath cohomology is a cohomology theory for directed graphs, which is defined using the path poset.
The aim of this paper is to investigate combinatorial properties of path posets, and to provide computational tools for multipath cohomology. In particular, we develop acyclicity criteria, and provide computations of multipath cohomology groups of oriented linear graphs.  We further interpret the path poset as the face poset of a simplicial complex, and we investigate realisability problems.
\end{abstract}
\section{Introduction}

Cohomology theories of directed graphs (shortly, digraphs) have become extremely important tools, and  are, nowadays, of central interest for the mathematical and scientific community. 
This is mainly due to the emergence of new techniques in Topological Data Analysis, which hinge on (co)homological and homotopical methods.

In this paper we are concerned with a cohomology theory of digraphs called multipath cohomology~\cite{primo}, and denoted by $\mathrm{H}_{\mu}^*$. This is defined as the poset homology~\cite{chandler2019posets,primo} of the path poset (cf.~\cite{turner}).
More abstractly, multipath cohomology can be seen as a functor cohomology, or as a {cellular cohomology}~\cite{TurnerEverittCell} -- see~\cite[Section~6]{primo} for a comparison. 
%
The main advantage of using poset homology over functor/cellular cohomology is its amenability to computations.
In view of the discussion in~\cite[Section~6]{primo}, the computations provided here yield information about the functor and cellular cohomology groups of (a mild modification of) path posets.

The goal of this paper is to analyse  how the combinatorics of the path poset may affect multipath cohomology. As a by-product, we develop a number of techniques which can be used to explicitly compute~$\mathrm{H}_{\mu}^\ast$ with coefficients in a field. A first application is the following theorem (see Theorem~\ref{thm:linearcohomology});

\begin{thm}
\label{thm:complete comp line digr}
Let $\tL$ be an oriented linear graph and $\bK$ be a field. Then,  there exist integers $h$, $k_1,\dots,k_h$, which are combinatorially determined by~$\tL$, such that the multipath cohomology of~$\tL$ 
is trivial if $k_i\equiv 0 \mod 3$ and $1 \leq i< h$; otherwise multipath cohomology
decomposes as the tensor product of graded modules, as follows
\[
\quad \Hmu^{* + h - 1}(\tL;\bK)= \Hmu^*(\mathtt{A}_{3\lfloor k_1/3 \rfloor}; \bK)\otimes \dots \otimes \Hmu^*(\mathtt{A}_{3\lfloor k_{h-1}/3 \rfloor}; \bK)\otimes \Hmu^*(\mathtt{A}_{k_h} ; \bK) \ ,
\]
 where each factor 
\begin{equation*}
  \mathrm{H}_{\mu}^k(\tA_n; \bK) =
\begin{cases}
    \bK & \mbox{if }  n=3(k-1)+2 \mbox{ or } n=3k,\\
    	0 & \mbox{otherwise} \\
	\end{cases}
\end{equation*}
is the multipath cohomology group of an alternating graph $\tA_n$ --~cf.~Figure~\ref{fig:alternating}.
\end{thm}

\begin{figure}[h]
    \centering
    \begin{tikzpicture}[baseline=(current bounding box.center)]
		\tikzstyle{point}=[circle,thick,draw=black,fill=black,inner sep=0pt,minimum width=2pt,minimum height=2pt]
		\tikzstyle{arc}=[shorten >= 8pt,shorten <= 8pt,->, thick]
		
		\node[above] (v0) at (0,0) {$v_0$};
		\draw[fill] (0,0)  circle (.05);
		\node[above] (v1) at (1.5,0) {$v_1$};
		\draw[fill] (1.5,0)  circle (.05);
		\node[above] (v2) at (3,0) {$v_2$};
		\draw[fill] (3,0)  circle (.05);
		\node[above] (v4) at (4.5,0) {$v_3$};
		\draw[fill] (4.5,0)  circle (.05);
		\node[above] (v5) at (6,0) {$v_4$};
		\draw[fill] (6,0)  circle (.05);
		\node[above] (v6) at (7.5,0) {$v_5$};
		\draw[fill] (7.5,0)  circle (.05);
		 \node[above]  at (9,0) {$v_{n-1}$};
		 \node (v7) at (9,0) {};
		\draw[fill] (9,0)  circle (.05);
		
		\node   at (8.25,0) {$\dots$};
		 \node (v8) at (10.5,0) {};
		 \node[above]  at (10.5,0) {$v_n$};
		\draw[fill] (10.5,0)  circle (.05);
		
		\draw[thick, bunired, -latex] (0.15,0) -- (1.35,0);
		\draw[thick, bunired, -latex] (2.75,0) -- (1.65,0);
		\draw[thick, bunired, latex-] (4.35,0) -- (3.15,0);
		\draw[thick, bunired, latex-] (4.65,0) -- (5.85,0);
		\draw[thick, bunired, latex-] (7.35,0) -- (6.15,0);
	    \draw[thick, bunired, ] (v7) -- (v8);

	\end{tikzpicture}
	\caption{An alternating graph on $n+1$ vertices. The edge between $v_{n -1} $ and~$v_{n}$ can be oriented either way depending on the parity of $n$. }
    \label{fig:alternating}
\end{figure}
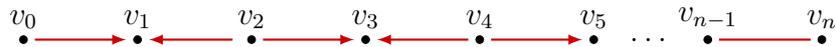

This result, together with the methods involved in its proof, proves that the multipath cohomology  of a digraph captures relevant combinatorial information. Among other implications, our computations reveal the non-triviality of multipath cohomology, in the sense of the following theorem (see Proposition~\ref{prop:comphom});

\begin{thm}
\label{thm:mulipath high degrees}
The multipath cohomology~$\Hmu^*(-;\bK)$ with coefficients in a field~$\bK$ can  be supported in arbitrarily high degree, and can be of arbitrarily high dimension.  
\end{thm}

Our interest in the combinatorial properties of the path poset is not limited to the purpose of understanding multipath cohomology groups, but it extends to its connections with the so-called {monotone} properties. 
A property of (di)graphs is called \emph{monotone} if it is preserved under deletion of edges -- e.g.~being acyclic, being a forest \emph{etc.} -- and the study of the homotopy type of simplicial complexes associated to monotone properties of (di)graphs is a central topic in combinatorial topology;
see, for instance, the classical papers~\cite{Vassiliev1993,bjorner_cdg, KozlovTrees,BBLSW}, as well as the more recent works \cite{zbMATH06797304,zbMATH07079126,PaoliniSalvetti}.
The study of a simplicial complex associated to the path poset fits into this framework.

In the last part of the paper we  describe the relationship between multipaths and simplicial complexes. We show that for a given digraph~$\tG$, there exists a simplicial complex $X(\tG)$ whose (reduced) simplicial cohomology is the multipath cohomology of $\tG$ -- see Theorem~\ref{thm:multipath is simplicial};

\begin{thm}
\label{thm:hs}
For each $n \geq 0$, the multipath cohomology group $\mathrm{H}_\mu^n(\tG;\bK)$ of $\tG$ is isomorphic to the ordinary homology~$\widetilde{\rm H}^{n-1}(X(\tG);\bK)$.
\end{thm}

In virtue of this theorem, we reinterpret some of our
 combinatorial results -- 
for instance, we reprove a  Mayer-Vietoris theorem for multipath cohomology. 
Then, we turn to the question of which simplicial complexes arise as multipath complexes. Among others, we observe that all wedges of spheres of the same dimension can be obtained in this way (Example~\ref{ex:spheres}~and~Proposition~\ref{prop:wedgsofspheres}).
It remains open the problem of whether all wedges of spheres, or more complicated spaces, can be realised as $X(\tG)$ -- cf.~Questions~\ref{q:realisability1}~and~\ref{q:realisability2}.

\subsection*{On the combinatorial techniques}

The construction of multipath cohomology with arbitrary coefficients (algebras and bimodules) is quite abstract. One first associates to a digraph $\tG$ its path poset. Then, using the fact that every poset can be seen as a category, certain abstract categorical constructions can be used to define the multipath cohomology of~$\tG$. In this process, the combinatorial information provided by the graph is somehow obscured.
However, when restricting to coefficients in the base ring, instead of a general algebra, computations can be carried over the path poset. It becomes therefore useful to develop combinatorial tools to study {path} posets.
We rely on the gluing construction $\nabla$ (cf.~Definition~\ref{nabla}) to provide the description of a path poset in terms of simpler path posets (cf.~Theorem~\ref{finedimondo}). 
The decomposition in terms of the gluing construction $\nabla$, together with a Mayer-Vietoris type result for multipath cohomology,  provides useful acyclicity criteria (Criteria~\ref{crit:FDM} and~\ref{crit:istmo}).
These criteria can be applied to compute explicitly the cohomology of a number of graphs, see~Table~\ref{tab: homology computatin}.

\begin{table}[h]
\centering
	{
		\setlength{\extrarowheight}{17.5pt}%
		\begin{tabular}{C{4cm}|C{1.5cm}|C{1.5cm}|C{2.3cm}|C{3.5cm}}
			\raisebox{.2cm}{Digraph \tG} & \raisebox{.2cm}{$\mathrm{H}_{\mu}^0(\tG;\bK)$} &  \raisebox{.2cm}{$\mathrm{H}_{\mu}^1(\tG;\bK)$} & \raisebox{.2cm}{$\mathrm{H}_{\mu}^2(\tG;\bK)$} & \raisebox{.2cm}{$\mathrm{H}_{\mu}^i(\tG;\bK)$, $i>2$}\\
			\hline\hline

			\raisebox{-.15em}{%
				\begin{tikzpicture}[scale=0.6][baseline=(current bounding box.center)]
					\tikzstyle{point}=[circle,thick,draw=black,fill=black,inner sep=0pt,minimum width=2pt,minimum height=2pt]
					\tikzstyle{arc}=[shorten >= 8pt,shorten <= 8pt,->, thick]
					\node[above] (v0) at (0,0) {};
					\draw[fill] (0,0)  circle (.05);
					\node[above] (v1) at (1.5,0) {};
					\draw[fill] (1.5,0)  circle (.05);
					\node[] at (3,0) {\dots};
					\node[above] (v4) at (4.5,0) {};
					\draw[fill] (4.5,0)  circle (.05);
					\node[above] (v5) at (6,0) {};
					\draw[fill] (6,0)  circle (.05);
					\draw[thick, bunired, -latex] (0.15,0) -- (1.35,0);
					\draw[thick, bunired, -latex] (1.65,0) -- (2.5,0);
					\draw[thick, bunired, -latex] (3.4,0) -- (4.35,0);
					\draw[thick, bunired, -latex] (4.65,0) -- (5.85,0);
			\end{tikzpicture}} &  $0$ &  $0$ & $0$ & $0$ \\
			\raisebox{.25em}{%
				\begin{tikzpicture}[scale=0.5][baseline=(current bounding box.center)]
					\tikzstyle{point}=[circle,thick,draw=black,fill=black,inner sep=0pt,minimum width=2pt,minimum height=2pt]
					\tikzstyle{arc}=[shorten >= 8pt,shorten <= 8pt,->, thick]
					\node[above] (v0) at (0,1) {};
					\draw[fill] (0,1)  circle (.05);
					\node[above] (w0) at (4.5,1) {};
					\draw[fill] (4.5,1)  circle (.05);
					\node[above] (w1) at (6.0,1) {};
					\draw[fill] (6.0,1)  circle (.05);
					\node[above] (w2) at (7.5,1) {};
					\draw[fill] (7.5,1)  circle (.05);
					\node[above] (v1) at (1.5,1) {};
					\draw[fill] (1.5,1)  circle (.05);
					\node[above] (v2) at (3,1) {};
					\draw[fill] (3,1)  circle (.05);
					\draw[thick, bunired, -latex] (0.15,1) -- (1.35,1);
					\draw[thick, bunired, -latex] (2.85,1) -- (1.65,1);
					\draw[thick, bunired, -latex]  (3.15,1) -- (4.35,1) ;
					\draw[thick, bunired, -latex] (5.85,1) -- (4.65,1);
					\draw[thick, bunired, -latex]  (6.15,1) -- (7.35,1) ;
			\end{tikzpicture}} &  $0$ &  $0 $ & $\bK $ & $0$ \\
			\raisebox{-.55em}{%
				\begin{tikzpicture}[scale=0.35][baseline=(current bounding box.center)]
					\tikzstyle{point}=[circle,thick,draw=black,fill=black,inner sep=0pt,minimum width=2pt,minimum height=2pt]
					\tikzstyle{arc}=[shorten >= 8pt,shorten <= 8pt,->, thick]
					
					\node[above] (w0) at (0,-1) {};
					\draw[fill] (0,-1)  circle (.05);
					\node[above] (v0) at (0,1) {};
					\draw[fill] (0,1)  circle (.05);
					\node (e0) at (1.5,1) {$...$};
					\draw (1.5,1)  circle (.05);
					\node[above] (v1) at (1.5,0) {};
					\draw[fill] (1.5,0)  circle (.05);
					\node[above] (v2) at (3,1) {};
					\draw[fill] (3,1)  circle (.05);
					\node[above] (v3) at (3,-1) {};
					\draw[fill] (3,-1)  circle (.05);
					
					\draw[thick, bunired, -latex] (0.15,0.85) -- (1.35,0);
					\draw[thick, bunired, -latex] (0.15,-0.85) -- (1.35,0);
					\draw[thick, bunired, -latex] (2.85,0.95) -- (1.65,0.05);
					\draw[thick, bunired, -latex] (2.85,-0.95) -- (1.65,-0.05);
			\end{tikzpicture}} &  $0$ &  $\bK^{n-1}$ & $0$ & $0$ \\
	        
			\raisebox{-.55em}{%
				\begin{tikzpicture}[scale=0.35][baseline=(current bounding box.center)]
					\tikzstyle{point}=[circle,thick,draw=black,fill=black,inner sep=0pt,minimum width=2pt,minimum height=2pt]
					\tikzstyle{arc}=[shorten >= 8pt,shorten <= 8pt,->, thick]
					
					\node[above] (e0) at (0,0) {};
					\draw[fill] (0,0)  circle (.05);
					\node[above] (w0) at (0,-1) {};
					\draw[fill] (0,-1)  circle (.05);
					\node[above] (v0) at (0,1) {};
					\draw[fill] (0,1)  circle (.05);
					\node[above] (v1) at (1.5,0) {};
					\draw[fill] (1.5,0)  circle (.05);
					\node[above] (v2) at (3,1) {};
					\draw[fill] (3,1)  circle (.05);
					\node[above] (v3) at (3,-1) {};
					\draw[fill] (3,-1)  circle (.05);
					
					\node[left] at (-.65,0) {$n$};
					\node at (0,0.255) {$\vdots$};
					\node at (3,0.255) {$\vdots$};
					\node[right] at (3.55,0) {$m$};
					
					\draw[thick, bunired, -latex] (0.15,0) -- (1.0,0);

					\draw[thick, bunired, -latex] (0.15,0.85) -- (1.35,0);
					\draw[thick, bunired, -latex] (0.15,-0.85) -- (1.35,0);
					\draw[thick, bunired, latex-] (2.85,0.95) -- (1.65,0.05);
					\draw[thick, bunired, latex-] (2.85,-0.95) -- (1.65,-0.05);
			\end{tikzpicture}} &  $0$ &  $0$ & $\bK^{(n-1)(m-1)}$ & $0$ \\

%
					
			\raisebox{-1.5em}{%
					\begin{tikzpicture}[scale=0.3][baseline=(current bounding box.center)]
		\tikzstyle{point}=[circle,thick,draw=black,fill=black,inner sep=0pt,minimum width=2pt,minimum height=2pt]
		\tikzstyle{arc}=[shorten >= 8pt,shorten <= 8pt,->, thick]

		\node  (v0) at (0,0) {};
		\node[below] at (0,0) {};
		\draw[fill] (0,0)  circle (.05);
		\node  (v1) at (0,3) { };
		\node[above]  at (0,3) {};
		\draw[fill] (0,3)  circle (.05);
		\node  (v2) at (3,3) { };
		\node[above] at (3,3) {};
		\draw[fill] (3,3)  circle (.05);
		\node  (v3) at (3,0) { };
		\node[below]  at (3,0) {};
		\draw[fill] (3,0)  circle (.05);
		
		\draw[thick, bunired, -latex] (v0) -- (v1);
		\draw[thick, bunired, -latex] (v1) -- (v2);
		\draw[thick, bunired, -latex] (v2) -- (v3);
		\draw[thick, bunired, -latex] (v3) -- (v0);
		\draw[thick, bunired, -latex] (v1) -- (v3);
	\end{tikzpicture}} &  $0$ &  $0$ & $\bK$ & $\Hmu^{3}(\tG;\bK)\cong \bK$\\			
			\raisebox{-1.5em}{%
				\newdimen\R
\R=0.8cm
\begin{tikzpicture}
\draw[xshift=5.0\R, fill] (270:\R) circle(.03)  node[below] {};
\draw[xshift=5.0\R,fill] (225:\R) circle(.03)  node[below left]   {};
\draw[xshift=5.0\R,fill] (180:\R) circle(.03)  node[left] {};
\draw[xshift=5.0\R,fill] (135:\R) circle(.03)  node[above left] {};
\draw[xshift=5.0\R, fill] (90:\R) circle(.03)  node[above] {};
\draw[xshift=5.0\R,fill] (45:\R) circle(.03)  node[above right] {};
\draw[xshift=5.0\R,fill] (0:\R) circle(.03)  node[right] {};
\draw[xshift=5.0\R,fill] (315:\R) circle(.03)  node[below right] {};

\node[xshift=5.0\R] (v0) at (270:\R) { };
\node[xshift=5.0\R] (v1) at (225:\R) { };
\node[xshift=5.0\R] (v2) at (180:\R) { };
\node[xshift=5.0\R] (v3) at (135:\R) { };
\node[xshift=5.0\R] (v4) at (90:\R) { };
\node[xshift=5.0\R] (v5) at (45:\R) { };
\node[xshift=5.0\R] (v6) at (0:\R) { };
\node[xshift=5.0\R] (vn) at (315:\R) { };

\draw[thick, bunired, -latex] (v0)--(v1);
\draw[thick, bunired, -latex] (v1)--(v2);
\draw[thick, bunired, -latex] (v2)--(v3);
\draw[thick, bunired, -latex] (v3)--(v4);
\draw[thick, bunired, -latex] (v4)--(v5);
\draw[thick, bunired, -latex] (v5)--(v6);
\draw[thick, bunired, -latex] (vn)--(v0);

\draw[xshift=5.0\R, fill] (292.5:\R) node[below right] {};
\draw[xshift=5.0\R,fill] (247.5:\R) node[below left] {};
\draw[xshift=5.0\R,fill] (202.5:\R)   node[left] {};
\draw[xshift=5.0\R,fill] (157.5:\R)  node[above left] {};
\draw[xshift=5.0\R, fill] (112.5:\R)   node[above] {};
\draw[xshift=5.0\R,fill] (67.5:\R) node[above right] {};
\draw[xshift=5.0\R,fill] (22.5:\R) node[right] {};
\draw[xshift=4.95\R,fill] (337.5:\R)  node {$\cdot$} ;
\draw[xshift=4.95\R,fill] (333:\R)  node {$\cdot$} ;
\draw[xshift=4.95\R,fill] (342:\R)  node {$\cdot$} ;
\end{tikzpicture}} &  $0$ &  $0$ & $0 $ & $\Hmu^{n}(\tP_n;\bK)\cong \bK$ \\%
	\end{tabular}}
	\vspace{.5cm}
	\caption{Some digraphs and their respective multipath cohomologies.}
	\label{tab: homology computatin}
\end{table}

We also develop a deletion-contraction type result for multipath cohomology with coefficients in an algebra~$A$ (Theorem~\ref{thm:secremoveleaves}). This result allows us to give a recursive formula for the (graded) Euler characteristic of $\tA_n$, which shows how the complexity of multipath cohomology increases in case $A\neq R$, and to prove the analogue of \cite[Lemma~3.3]{Prz} -- see~Corollary~\ref{cor:homology linear graphs}. We conclude by noting that~Corollary~\ref{cor:homology linear graphs} can be proved, in case $A$ is commutative, by using the deletion-contraction long exact sequence in chromatic homology~\cite{HGRong}.

\subsection*{Conventions}

Typewriter font, e.g.~$\tG$, $\tH$, \emph{etc.}, are used to denote \emph{finite} graphs (both directed and unoriented). %
All base rings are assumed to be unital and commutative, and algebras are assumed to be associative. Unless otherwise stated, $R$ denotes a principal ideal domain, $\bK$ is a field, $A$ is a unital $R$-algebra, and all tensor products  $\otimes$ are assumed to be over the base ring~$R$. 
Given a cochain complex~$C^*$, we denote by $C^*[i]$ the shifted complex~$C^{*+i}$. General references for graph theory, algebra, and algebraic topology are \cite{West}, \cite{LangAlg}, and \cite{Kozlov}, respectively.

\subsection*{Acknowledgements}
The authors extend their felicitations to {\tt I.C.}
LC acknowledges support from the \'{E}cole Polytechnique F\'{e}d\'{e}rale de Lausanne via a collaboration agreement with the University of Aberdeen.
SDT was partially supported by GNSAGA - INDAM group during the writing of this paper. During the writing of this paper, CC was a postdoc at NYUAD. LC and CC acknowledge partial support from the Heilbronn Small Grants Scheme. 
All the authors wish to warmly thank the anonymous referee for their comments, which helped to improve the quality of the paper.

\section{Basic Notions}\label{sec:notions}

In this section,  we review some basic notions on (directed) graphs and posets, and recall the construction of multipath cohomology.  Then, we specialise the general construction to the case of multipath cohomology with  coefficients in a field~$\bK$ -- \emph{en lieu} of a general algebra~$A$. 

\subsection{Digraphs and posets}
Recall that a \emph{directed graph}~$\tG$, often shortened to \emph{digraph},  is a pair of finite sets $(V(\tG),E(\tG))$, called \emph{vertices} and \emph{edges}, where $E(\tG) \subseteq \{ V(\tG)\times V(\tG) \setminus \{ (v,v)\ |\ v\in V(\tG) \}$. %
Unless otherwise stated, we will refer to digraphs, simply, as graphs. 
When dealing with (un)directed graphs, i.e.~graphs for which the edges are not oriented, the adjective ``\emph{(un)oriented}'' will be explicitly stated. The forthcoming definitions for digraphs apply verbatim to unoriented graphs by discarding the orientation of the edges. 
Note that two vertices $v$ and $w$ in a digraph {can} share at most two edges: $(v,w)$ and $(w,v)$. 

By definition, an edge of a digraph is an ordered set of two distinct vertices, say $e =(v,w)$. The vertex $v$ is called  the \emph{source} of $e$, while the vertex $w$ is called the \emph{target} of $e$. The source and target of an edge $e$ will be denoted by $s(e)$ and $t(e)$, respectively. If a vertex $v$ is {either} a source or a target of an edge $e$, we will say that \emph{$e$ is incident to $v$}. Furthermore, we say that $v \in V(\tG)$ is a \emph{sink} (resp.~a \emph{source}) if for every $e\in E(\tG)$ incident to $v$ we have $v = t(e)$ (resp.~$v = s(e)$). 
Finally, a digraph $\tG$ with $n$ edges is a \emph{sink} (resp. a \emph{source}) \emph{over $n$ vertices} if it has a unique sink (resp.~source), and every edge in $\tG$ is incident to it.

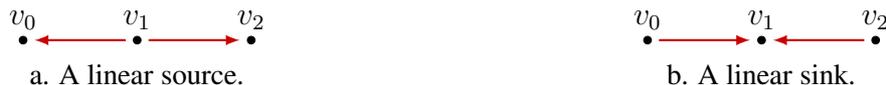
\begin{figure}[h]
	\centering
	\begin{subfigure}[b]{0.4\textwidth}
		\centering
		\begin{tikzpicture}[baseline=(current bounding box.center)]
		\tikzstyle{point}=[circle,thick,draw=black,fill=black,inner sep=0pt,minimum width=2pt,minimum height=2pt]
		\tikzstyle{arc}=[shorten >= 8pt,shorten <= 8pt,->, thick]
		
		\node[above] (v0) at (0,1) {$v_0$};
		\draw[fill] (0,1)  circle (.05);
		\node[above] (v1) at (1.5,1) {$v_1$};
		\draw[fill] (1.5,1)  circle (.05);
		\node[above] (v2) at (3,1) {$v_{2}$};
		\draw[fill] (3,1)  circle (.05);

		\draw[thick, bunired, -latex] (1.35,1) -- (0.15,1);
		\draw[thick, bunired, -latex] (1.65,1) -- (2.85,1);
	\end{tikzpicture}
\caption{A linear source. }\label{fig:source}
\end{subfigure}
\hspace{0.1\textwidth}
	\begin{subfigure}[b]{0.4\textwidth}
		\centering
		\begin{tikzpicture}[baseline=(current bounding box.center)]
		\tikzstyle{point}=[circle,thick,draw=black,fill=black,inner sep=0pt,minimum width=2pt,minimum height=2pt]
		\tikzstyle{arc}=[shorten >= 8pt,shorten <= 8pt,->, thick]
		
		\node[above] (v0) at (0,1) {$v_0$};
		\draw[fill] (0,1)  circle (.05);
		\node[above] (v1) at (1.5,1) {$v_1$};
		\draw[fill] (1.5,1)  circle (.05);
		\node[above] (v2) at (3,1) {$v_{2}$};
		\draw[fill] (3,1)  circle (.05);

		\draw[thick, bunired, -latex] (0.15,1) -- (1.35,1);
		\draw[thick, bunired, -latex] (2.85,1) -- (1.65,1);
	\end{tikzpicture}
\caption{A linear sink. }\label{fig:sink}
\end{subfigure}
	\caption{Linear source and sink.}\label{fig:nnstep}
\end{figure}

A \emph{morphism of digraphs} from  $\tG_1$ to  $\tG_2$ is a function $\phi\colon V(\tG_1)\to V(\tG_2)$ such that:
\[ e = (v,w) \in E(\tG_1)\ \Longrightarrow\ \phi (e) \coloneqq (\phi(v),\phi(w)) \in E(\tG_2)\ .\]
 A morphism of digraphs is called \emph{regular} if it is injective as a function.

A \emph{sub-graph} \tH of a graph \tG is a graph such that $V(\tH)\subseteq V(\tG)$ and $E(\tH)\subseteq E(\tG)$, and in such case we  write $\tH \leq \tG$.
If $\tH \leq \tG$ and $\tH\neq \tG$ we say that \tH is a \emph{proper sub-graph} of $\tG$, and we write~$\tH< \tG$.

\begin{defn}
 If $\tH\leq \tG$ and $V(\tH) = V(\tG)$ we  say that \tH is a \emph{spanning sub-graph} of $\tG$.
\end{defn}

Given a {proper} spanning sub-graph $\tH < \tG$, we can find an edge $e\in E(\tG)\setminus E(\tH)$. 
 The spanning sub-graph of \tG {obtained from \tH by adding an edge $e$}  is simply denoted {by} $\tH\cup e$.
~\newline

We now review some basic notions about partially ordered sets.
A \emph{partially ordered set}, or simply \emph{poset}, is a pair $(S,\triangleleft)$ consisting of a set  $S$ and a partial order $\triangleleft$ on $S$. A morphism of posets $f\colon (S,\triangleleft )\to (S',\triangleleft')$ is a strictly monotone map of sets. 

\begin{example}
The \emph{standard Boolean poset} $\mathbb{B}(n)$ (of size $2^{n}$) is the poset~$(\wp(\{ 0, ... , n -1\}),\subset)$, where $\wp$ denotes the power set -- i.e.~the set of all subsets.
A poset is called a \emph{Boolean poset}, if it is isomorphic to the standard Boolean poset $\mathbb{B}(n)$, for some $n$. 
\end{example}
\begin{example}
Let $\tG$ be a digraph with $n$ edges. The poset~$(SSG(\tG),<)$ of \emph{spanning subgraphs} of $\tG$ is given by all the spanning subgraphs of $\tG$ with order relation given by the property of being a subgraph. The associated covering relation~$\prec$ can be described as follows:
\[
\tH \prec \tH' \Longleftrightarrow \exists \ e\in E(\tH')\setminus E(\tH) : \tH'= \tH\cup e \ .
\]
Then, $(SSG(\tG),<)$ is a Boolean poset isomorphic to $\mathbb{B}(n)$  -- see also \cite[Example~2.14]{primo}.
\end{example}

Given a partial order $\triangleleft$ on a set~$S$, there is an associated \emph{covering relation}  $\widetilde{\triangleleft}$, given by $x\  \widetilde{\triangleleft}\ y$ if, and only if, $x \triangleleft y$ and there is no $z$ such that $x\triangleleft z$, $z\triangleleft y$.
In order to visually represent posets associated to digraphs, we use covering relations and the associated Hasse graphs.
Recall that the \emph{Hasse graph}~$\Hasse (S,\triangleleft)$ of a poset $(S,\triangleleft)$ is the graph whose vertices are the elements of $S$ and such that $(x,y)$ is an edge if, and only if, $x\,\widetilde{\triangleleft}\, y$.
Each morphism of digraphs $\phi\colon\Hasse(S,\triangleleft)\to \Hasse(S',\triangleleft') $ induces a morphism of posets $\phi\colon(S,\triangleleft)\to (S',\triangleleft')$. We remark that not all morphisms of posets arise this way.
Recall also that, given a poset $(S,\triangleleft)$,  a  \emph{sub-poset} is a subset $S'\subseteq S$ with the order relation $\triangleleft_{| S'\times S'}$ induced by~$\triangleleft$.

\begin{defn}\label{def:hereditary}
 A sub-poset $(S',\triangleleft_{| S'\times S'})$ is called \emph{downward closed} with respect to $(S,\triangleleft)$, if whenever $h\triangleleft h'$ and $h'\in S'$, then $h\in S'$.
\end{defn}

Essential to the construction of multipath cohomology are the following properties:

\begin{defn}\label{def:faith and squre}
Let $(S,\triangleleft)$ be a poset and $(S',\triangleleft_{| S'\times S'})$ be a sub-poset of $(S,\triangleleft)$.
\begin{enumerate}
	\item\label{def:faith_label1} We say that $(S,\triangleleft)$ is \emph{squared} if for each triple $x,y,z\in S'$ such that $z$ covers $y$ and $y$ covers $x$, there is a unique $y'\neq y$ such that $z$ covers $y'$ and $y'$ covers $x$. Such elements $x$, $y$, $y'$, and $z$ will be called a \emph{square} in $S$.
\item We say that $(S',\triangleleft_{| S'\times S'})$ is \emph{faithful} if the covering relation in $S'$ induced by $\triangleleft_{| S'\times S'}$ is the restriction of the covering relation in $S$ induced by $\triangleleft$;
\end{enumerate}
\end{defn}

Note that Boolean posets are squared and that the property of being squared or faithful is preserved under intersections \cite[Proposition 2.21]{primo}.
Furthermore, downward  closed sub-posets are faithful, and each downward  closed sub-poset of a squared poset is also squared.

If $(S,\triangleleft_S)$ and $(S',\triangleleft_{S'})$ are posets, then  their \emph{product poset} is the  set $S \times S'$ with the relation 
\begin{equation}\label{eq:product posets}
    (x_1,x_2)\triangleleft_{S\times S'}(y_1,y_2) \Longleftrightarrow (x_1,x_2)=(y_1,y_2) \mbox{ or } x_1 \triangleleft_S y_1  \mbox{ and } x_2 \triangleleft_{S'} y_2. 
\end{equation}
{The definition of product poset is essential to introduce the cone of a poset. Let $P$ be a poset.
\begin{defn}\label{SqCone}  
The \emph{cone of $P$}, denoted by $\Cone P$, is the product poset 
$P \times \mathbb{B}(1)$. 
\end{defn}
}
As $\mathbb{B}(1)$ is (isomorphic to) the poset on the set $\{0,1\}$ with the relation $0<1$, the cone $\Cone P$ can also be seen as $P \times \{0,1\}$.
The covering relation in $\Cone P$ can be explicitly described; an element $(a,i)$ is covered by $(b,j)$ if and only if $i=j$ and $a \triangleleft b$ or $i<j$ and $a=b$.

\subsection{Path posets}
We introduce one of the main tools in the definition of multipath cohomology of directed graphs, the \emph{path poset} associated to a directed graph~$\tG$.
By a \emph{simple path} in \tG  we mean a sequence of edges $e_1,...,e_n$ of \tG such that~$s(e_{i+1})=t(e_i)$ for $i=1,\dots,n-1$, and no vertex is encountered twice, i.e.~if $s(e_i) = s(e_j)$ or $t(e_i) = t(e_j)$, then $i=j$, and is not a cycle, i.e.~$s(e_1)\neq t(e_n)$. 
A connected component of \tG is a sub-graph $\tH$ of \tG whose geometric realisation (as CW-complex) $|\tH|$ is connected.
Following \cite{turner}, a \emph{multipath} of $\tG$ is a spanning sub-graph such that each connected component is either a vertex or its edges admit an ordering such that it is a simple path.
The set of multipaths of $\tG$ has a natural partially ordered structure:

\begin{defn}\label{def:pathposet}
	The \emph{path poset} of $\tG$ is the poset $(P(\tG),<)$ associated to $\tG$, that is, the set of multipaths of \tG ordered by the relation of ``being a sub-graph''. 
\end{defn}

Observe that the order relation makes sense, because each multipath is a subgraph of $\tG$.

\begin{rem}\label{rem: PG squared and faithful}
The path poset~$P(\tG)$ is a downward closed subposet of a Boolean poset -- the Boolean poset of all spanning subgraphs. Therefore, it is faithful and squared. 
\end{rem}

When the partial order on $P(\tG)$ is not specified, we will always implicitly assume it to be the order relation~$<$.  Moreover, with abuse of notation, we will also  write $P(\tG)$ instead of  $(P(\tG),<)$.
We now provide some examples of path posets -- cf.~\cite[Section~2]{primo}.

\begin{example}\label{ex:Pn}
Consider the coherently oriented linear graph $\tI_n$ with $n$ edges -- Figure~\ref{fig:nstep}. 
Then, $(P(\tI_n), < )$ is isomorphic to the Boolean poset~$\mathbb{B}(n)$.
Let $\tP_n$ be the coherently oriented polygonal graph with $n+1$ edges -- Figure~\ref{fig:poly}. Then, $(P(\tP_n) ,<)$, is isomorphic to the Boolean poset~$\mathbb{B}(n)$ minus its maximum.
\begin{figure}[h]
	\begin{tikzpicture}[baseline=(current bounding box.center)]
		\tikzstyle{point}=[circle,thick,draw=black,fill=black,inner sep=0pt,minimum width=2pt,minimum height=2pt]
		\tikzstyle{arc}=[shorten >= 8pt,shorten <= 8pt,->, thick]
		
		\node[above] (v0) at (0,0) {$v_0$};
		\draw[fill] (0,0)  circle (.05);
		\node[above] (v1) at (1.5,0) {$v_1$};
		\draw[fill] (1.5,0)  circle (.05);
		\node[] at (3,0) {\dots};
		\node[above] (v4) at (4.5,0) {$v_{n-1}$};
		\draw[fill] (4.5,0)  circle (.05);
		\node[above] (v5) at (6,0) {$v_{n}$};
		\draw[fill] (6,0)  circle (.05);
		
		\draw[thick, bunired, -latex] (0.15,0) -- (1.35,0);
		\draw[thick, bunired, -latex] (1.65,0) -- (2.5,0);
		\draw[thick, bunired, -latex] (3.4,0) -- (4.35,0);
		\draw[thick, bunired, -latex] (4.65,0) -- (5.85,0);
	\end{tikzpicture}
	\caption{The coherently oriented linear graph $\tI_n$.}
	\label{fig:nstep}
\end{figure}
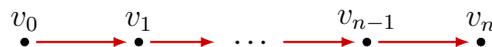
%
\begin{figure}[h]
\newdimen\R
\R=2.0cm
\begin{tikzpicture}
\draw[xshift=5.0\R, fill] (270:\R) circle(.05)  node[below] {$v_0$};
\draw[xshift=5.0\R,fill] (225:\R) circle(.05)  node[below left]   {$v_1$};
\draw[xshift=5.0\R,fill] (180:\R) circle(.05)  node[left] {$v_2$};
\draw[xshift=5.0\R,fill] (135:\R) circle(.05)  node[above left] {$v_3$};
\draw[xshift=5.0\R, fill] (90:\R) circle(.05)  node[above] {$v_4$};
\draw[xshift=5.0\R,fill] (45:\R) circle(.05)  node[above right] {$v_5$};
\draw[xshift=5.0\R,fill] (0:\R) circle(.05)  node[right] {$v_6$};
\draw[xshift=5.0\R,fill] (315:\R) circle(.05)  node[below right] {$v_{n}$};

\node[xshift=5.0\R] (v0) at (270:\R) { };
\node[xshift=5.0\R] (v1) at (225:\R) { };
\node[xshift=5.0\R] (v2) at (180:\R) { };
\node[xshift=5.0\R] (v3) at (135:\R) { };
\node[xshift=5.0\R] (v4) at (90:\R) { };
\node[xshift=5.0\R] (v5) at (45:\R) { };
\node[xshift=5.0\R] (v6) at (0:\R) { };
\node[xshift=5.0\R] (vn) at (315:\R) { };

\draw[thick, bunired, -latex] (v0)--(v1);
\draw[thick, bunired, -latex] (v1)--(v2);
\draw[thick, bunired, -latex] (v2)--(v3);
\draw[thick, bunired, -latex] (v3)--(v4);
\draw[thick, bunired, -latex] (v4)--(v5);
\draw[thick, bunired, -latex] (v5)--(v6);
\draw[thick, bunired, -latex] (vn)--(v0);

\draw[xshift=5.0\R, fill] (292.5:\R) node[below right] {$e_{n}$};
\draw[xshift=5.0\R,fill] (247.5:\R) node[below left] {$e_0$};
\draw[xshift=5.0\R,fill] (202.5:\R)   node[left] {$e_1$};
\draw[xshift=5.0\R,fill] (157.5:\R)  node[above left] {$e_2$};
\draw[xshift=5.0\R, fill] (112.5:\R)   node[above] {$e_3$};
\draw[xshift=5.0\R,fill] (67.5:\R) node[above right] {$e_4$};
\draw[xshift=5.0\R,fill] (22.5:\R) node[right] {$e_5$};
\draw[xshift=4.95\R,fill] (337.5:\R)  node {$\cdot$} ;
\draw[xshift=4.95\R,fill] (333:\R)  node {$\cdot$} ;
\draw[xshift=4.95\R,fill] (342:\R)  node {$\cdot$} ;
\end{tikzpicture}
\caption{The coherently oriented polygonal graph $\tP_n$.}
\label{fig:poly}
\end{figure}
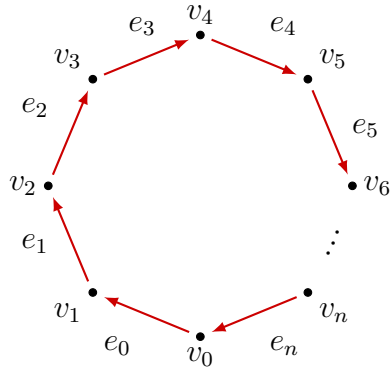
\end{example}

\subsection{Multipath cohomology}\label{sec:digraph_hom}

Given a special type of poset coherently assigned to each digraph, and a choice of a sign assignment on it  (see Definition~\ref{def:sign_ass}), one can define a cohomology theory for directed graphs -- see~\cite{primo}. 
Let $\bZ_2$ be the cyclic group on two elements.

\begin{defn}\label{def:sign_ass}
	A \emph{sign assignment} on a poset $(S,\triangleleft)$ is an assignment of elements $\epsilon_{x,y}\in \bZ_2$ to each pair of elements $x,y\in S$ with $x\ \widetilde{\triangleleft}\ y$, such that the equation
	\begin{equation}\label{eq:signassign}
		\epsilon_{x,y} + \epsilon_{y,z} \equiv \epsilon_{x,y'} + \epsilon_{y',z} + 1 \mod 2
	\end{equation}
	holds for each square  $x\ \widetilde{\triangleleft}\ y, y'\ \widetilde{\triangleleft}\ z$.
\end{defn}


\emph{A priori}, the existence of a sign assignment on a given poset is not clear. For a cohomological sufficient condition for a sign assignment on a given poset see~\cite[Section~3.2]{primo}. For the poset of spanning sub-graphs -- or, better, for any Boolean poset -- and their sub-posets, a sign assignment can be easily described -- see \cite{Khovanov}. 
More generally, one may ask when two sign assignments on a given poset are isomorphic. We refrain here from giving the definition of isomorphism of sign assignments -- cf.~\cite[Definition~3.13]{primo} -- but in cases of interest to us, all sign assignments are isomorphic -- cf.~\cite[Theorem~3.16 \& Corollary~3.17]{primo}.


Recall that the \emph{length} of a graph $\tG$, denoted by $\lgt(\tG)$,  is the number of edges in $\tG$.

\begin{defn}\label{def:ell}
Let $P$ be a finite poset with a minimum $m$. The \emph{level} $\ell(x)$ of $x\in P$ is the  minimal length among all simple paths joining $x$ to the minimum in $\Hasse(P)$.
\end{defn}

If $P= P(\tG)$ 
is the path poset of a graph,
then the level and the length coincide. 
More generally, if $P$ is a faithful sub-poset of $SS(\tG)$,  the notion of level of an element $\tH \in P$ can be extended as follows:
\[ \ell (\tH) = \# E(\tH) + \# V(\tH) - \min \{ \# E(\tH') + \# V(\tH') \mid \tH'\in P \} \ . \]

For the rest of the section $R$ denotes a commutative ring with identity, and $A$ an associative unital $R$-algebra.
An \emph{ordered digraph} is a digraph with a fixed ordering of the vertices.

Let $\tG$ be an ordered graph and let $v_0\in V(\tG)$ be the  minimum with respect to the given ordering.
Given a multipath $\tH<\tG$, to each connected component of $\tH$ we associate a copy of~$A$. Then we take the ordered tensor product. More concretely, if $c_0<\dots<c_k$ is the set of ordered connected components of $\tH$, we define: 
\begin{equation}\label{eq:fun_obj}
	\mathcal{F}_{A}(\tH)\coloneqq {A_{c_{0}}}  \otimes_R \cdots 
	\otimes_R  A_{c_k}\ ,\end{equation}
where  all the modules are labelled  by the respective component. 

Assume $\tH' = \tH\cup e$. We define the source~$s(e,\tH)$ (resp.~target~$t(e,\tH)$) of $e$ in $\tH$ as the index of the connected component of $\tH$ containing the source (resp.~target) of $e$.
Denote by $c_0$,...,$c_{k}$ the ordered components of $\tH$,  denote by  $c'_0$,...,$c'_{k-1}$ the ordered components of $\tH'$, and assume that the addition of $e$ merges $c_i$ and $c_j$. Then, for each $h=0,...,k-1$, there is a natural identification
\begin{equation}\label{eq:identification_components}
	c'_h = \begin{cases} c_h & \text{if } 0\leq h<i \text{ or } i< h < j;\\ c_i \cup e \cup c_j &\text{if }  h =i; \\ c_{h+1} &\text{if }  j\leq h<k.\end{cases}
\end{equation}
for some $0\leq i < j \leq k$. Using this  identification,
 we  define
$ \mu_{\tH\prec \tH'}\colon \mathcal{F}_{A}(\tH) \longrightarrow \mathcal{F}_{A}(\tH')$
 as 
\[ \mu_{\tH\prec \tH'}(a_0 \otimes \cdots \otimes a_k) =  
a_{0} \otimes \cdots \otimes a_{s(e,\tH)-1} \otimes a_{s(e,\tH)}\cdot a_{t(e,\tH)} \otimes a_{s(e,\tH)+1} \otimes \cdots \otimes \widehat{a_{t(e,\tH)}} \otimes \cdots \otimes a_{k-1} \otimes a_{k} 
	\]
where $ \widehat{a_{t(e,\tH)}} $ indicates that $a_{t(e,\tH)}$ is missing. 
Let $\epsilon$ be a sign assignment on $P(\tG)$.
We can now define the cochain groups
\[ C^n_{\mu}(\tG;A) \coloneqq \bigoplus_{\tiny\begin{matrix}
		\tH\in P(\tG)\\
		\ell(\tH) = n
\end{matrix}}  {\mathcal{F}_A(\tH)}
\]
together with the differential
\[d^{n}=d^{n}_{\mu} \coloneqq \sum_{\tiny\begin{matrix}
		\tH\in P(\tG)\\
		\ell(\tH) = n
\end{matrix}} \sum_{\tiny\begin{matrix}
		\tH'\in P(\tG)\\
		\tH \prec\tH'
\end{matrix}} (-1)^{ \epsilon(\tH,\tH')}\mu(\tH\prec \tH') \ .\]
 It has been proved that
 $(C_\mu^*(\tG;A),d^*)$  is a cochain complex \cite[Theorem~4.10]{primo}. Furthermore, the path poset $P(\tG)$ is squared and faithful by Remark~\ref{rem: PG squared and faithful}, hence the homology groups of $(C_\mu^*(\tG;A),d^*)$ do not depend on the sign assignment~$\epsilon$ used in the definition of $(C_\mu^*(\tG;A),d^*)$ and on the choice of the ordering on $V(\tG)$ \cite[Corollary~3.18 \& Proposition~4.11]{primo}. We are ready to give the definition of the multipath cohomology of a directed graph:
 
 \begin{defn}\label{def:multipathhom}
 The \emph{multipath cohomology}  $\mathrm{H}_{\mu}^*(\tG; A)$ of a digraph~$\tG$ with with coefficients in an algebra $A$ is the homology  of the cochain complex $(C^{*}_{\mu}(\tG; A),d^*)$. 
\end{defn}

Observe that, when $A$ is the ring~$R$, the tensor products in Equation~\eqref{eq:fun_obj} simply give $\mathcal{F}_{R}(\tH)= {R_{c_{0}}}  \otimes_R \cdots 
\otimes_R  R_{c_k}\cong R$, for each  multipath~$\tH<\tG${. The isomorphism between the tensor powers of $R$ and $R$ itself is given by multiplication} $\mu$, therefore the differential can be written as
\[d^{n} =\sum_{\tiny\begin{matrix}
		\tH\in P(\tG)\\
		\ell(\tH) = n
\end{matrix}} \sum_{\tiny\begin{matrix}
		\tH'\in P(\tG)\\
		\tH \prec\tH'
\end{matrix}} (-1)^{ \epsilon(\tH,\tH')} \mathrm{Id}_{\tH\prec \tH'} \ .\]
Identifying $\mathcal{F}_{R}(\tH)$ with a copy of $R$ gives us a a set of linearly independent generators $\{ b_{\tH} \}_{\tH}$ (as free $R$-module) for $C_\mu^*(\tG;R)$ indexed by multipaths. 

If $\phi\colon \tG_1 \to \tG_2$ is a regular morphism of digraphs, then it induces (functorially) a morphism of posets $P\phi\colon P({\tG_1}) \rightarrow P({\tG_2})$, as it sends multipaths of $\tG_1$ to multipaths of $\tG_2$. We obtain a (controvariant) morphism of cochain complexes 
$\phi^*\colon C_\mu^*(\tG_2;A) \to C_\mu^*(\tG_1;A)$
where we fixed a sign assignment on $P({\tG_2})$ and we considered the sign on $P({\tG_1})$ by restriction.

We conclude the section with the computation of the multipath cohomology of the coherently oriented linear graph;

\begin{example}\label{ex:In}
Consider the coherently oriented linear graph $\tI_n$ of length~$n$, illustrated in Figure~\ref{fig:nstep}.
Then, 
its multipath cohomology~$\Hmu^*(\tI_n; \bK)$ is trivial \cite[Example~4.20]{primo}.
\end{example}

\section{A Combinatorial Description of Path Posets}
\label{sec:comb_path_poset}

The aim of this section is to give a combinatorial description of the path poset associated to a directed graph $\tG$; we show that the path poset can be constructed by gluing together simpler path posets associated to (suitable)  subgraphs of $\tG$. 
In the follow up, we will deal with disconnected graphs. A  straightforward observation is that the multipath cohomology of disconnected graphs is the tensor product of the multipath cohomologies:

\begin{rem}\label{DisjointUnion}
 Let $\tG$ be the disjoint union of connected digraphs $\tG_1, \dots, \tG_n$.  Then, the path poset~$P(\tG)$ is the product $ P({\tG_1}) \times \dots \times P({\tG_n})$  -- cf.~Equation~\eqref{eq:product posets}. Hence, the multipath cohomology of $\tG$ splits as the graded tensor product:
 \[\Hmu^*(\tG;\bK)=\Hmu^*(\tG_1;\bK) \otimes \dots \otimes \Hmu^*(\tG_n;\bK) \ . \]
In particular, if $\Hmu^*(\tG_i;\bK)=0$ for  some $i \in \{1, \dots , n\}$, then $\Hmu^*(\tG;\bK)=0$.
\end{rem}

We introduce a gluing operation for directed graphs.

\begin{defn}[Gluing]\label{nabla} 
Let $\tG, \tG_1, \tG_2$ be digraphs, and $\imath_1\colon\tG \to \tG_1$ and $\imath_2\colon\tG \to \tG_2$ be  regular morphisms.
The \emph{gluing} of $\tG_1$ and $\tG_2$ along $\tG$ is the digraph~$\nabla_{\tG} (\tG_1, \tG_2)$ defined as follows: 
\begin{enumerate}
    \item $V(\nabla_\tG (\tG_1, \tG_2)) \coloneqq V(\tG_1) \sqcup V(\tG_2) /  \sim$, where $x \sim y$ if, and only if, either $x=y$ or $x\in \imath_1(\tG)$, $y\in \imath_2(\tG)$, and $\imath_1^{-1}(x)=\imath_2^{-1}(y)$;
    \item $([v],[w])\in E(\nabla_\tG (\tG_1, \tG_2))$ if, and only if, there exist $ v'\in [v]$, $ w'\in [w]$, and $i\in \{1,2\}$ such that $(v',w') \in E(\tG_i) $, where $[\cdot]$ denotes an equivalence class with respect to $\sim$.
\end{enumerate}
Roughly speaking, $\nabla_{\tG} (\tG_1, \tG_2)$ is the graph obtained from $\tG_1$ and $\tG_2$ by identifying the vertices and edges belonging to the image of $\tG$.
\end{defn}

When clear from the context and for ease of notation, we denote the edge $([v],[w])$ in  the set~$E(\nabla_\tG (\tG_1, \tG_2))$ as $(v,w)$.
For a given graph $\tG$, the operation $\nabla_{\tG}(-,-)$  is commutative and associative up to isomorphism of digraphs.
Let $\mathbf{Digraph}$ be the category of digraphs and regular morphisms of digraphs. We can reinterpret the gluing as a categorical pushout:

\begin{rem}\label{rem:nabla is pushout}
The operation $\nabla_{\tG}(-,-)$ is the categorical push-out -- cf.~\cite[Section~III.3]{maclane:71} -- in the category~$\mathbf{Digraph}$. 
Since $\nabla_\tG (\tG_1, \tG_2)$ is an object of $\mathbf{Digraph}$, and since the inclusions of $\tG_1$ and $\tG_2$ in $\nabla_\tG (\tG_1, \tG_2)$ are regular morphisms of digraphs, we have a commutative square
\[\begin{tikzcd}
 \tG \arrow[r, "\imath_1"] \arrow[d, "\imath_2"'] & \tG_1 \arrow[d, "\jmath_1"]\\
\tG_2\arrow[r, "\jmath_2"] & \nabla_\tG (\tG_1, \tG_2) \end{tikzcd}\]
	in $\mathbf{Digraph}$. 
Note that~$V(\nabla_\tG (\tG_1, \tG_2))$ is  the push-out of $V(\tG_1)$ and $V(\tG_2)$ along $V(\tG)$ in the category~$\mathbf{Set}$ of sets.
Now, given another digraph $\tG'$ such that the square
	\[\begin{tikzcd}
 \tG \arrow[r, "\imath_1"] \arrow[d, "\imath_2"'] & \tG_1 \arrow[d, "\jmath_1' "]\\
\tG_2\arrow[r, "\jmath_2' "] & \tG'
	\end{tikzcd}\]
commutes, then we get a function $V(\nabla_\tG (\tG_1, \tG_2))\to V(\tG')$ since $V(\nabla_\tG (\tG_1, \tG_2))$ is a push-out in~$\mathbf{Set}$. Such a function extends to a map of digraphs by definition of morphism of digraphs, which is injective as it is composition of injective functions. 
\end{rem}

If $\tt{G}'$ is a digraph, and 
$\imath\colon \tG' \rightarrow \nabla_{\tG} (\tG_1, \tG_2)$ and $\jmath\colon\tG' \rightarrow \tG_3$ are regular morphisms, we define:
\[\nabla_{\tG',\tG} (\tG_1, \tG_2, \tG_3)\coloneqq \nabla_{\tG'}(\nabla_{\tG} (\tG_1, \tG_2),\tG_3) \ . \]
In general, if $\tG_1, \dots, \tG_n$ is a family of digraphs such that for every $k<n$ there exist a (regular) morphisms of digraphs $\imath_k\colon \tG \rightarrow \tG_k$ we will denote $\nabla_{\tG,\dots, \tG} (\tG_1, \dots , \tG_n)$ by  $\nabla_{\tG} (\tG_1, \dots  \tG_n)$.
Note that $\nabla_{\tG} (\tG_1, \dots  \tG_n)$ does not depend, up to isomorphism of digraphs, on the order of the digraphs $\tG_1, \dots  ,\tG_n$, whereas $\nabla_{\tG',\tG} (\tG_1, \tG_2, \tG_3)$ might.

\begin{defn}
The gluing of two posets $P_1$ and $P_2$ along a common subposet~$P$ is the poset, denoted by $\nabla_{P}(P_1,P_2)$, whose Hasse diagram is the gluing of $\Hasse(P_1)$ and $\Hasse(P_2)$ along~$\Hasse(P)$.
\end{defn}

 Observe that the gluing does not commute with the operation of taking path posets. To see it, let $\tG_1$ be the graph \raisebox{-.2em}{\begin{tikzpicture}
\node (a) at (0,0) {}  ;
\node[above] at (a) {$v_0$};
\draw[fill] (a)  circle (.05);

\node (b) at (1,0) {} ;
\node[above] at (b) {$v_1$};

\draw[fill] (b)  circle (.05);

\node (c) at (2,0) {};
\node[above] at (c) {$v_2$};
\draw[fill] (c)  circle (.05);

\draw[thick, -latex, bunired] (a) -- (b);

\draw[thick, -latex, bunired] (b) -- (c);

\end{tikzpicture}}
 and $\tG_2$ the graph
 \raisebox{-.2em}{\begin{tikzpicture}
\node (a) at (0,0) {} circle (.05);
\node[above] at (a) {$v_1$};
    \draw[fill] (a) circle (.05);

\node (b) at (1,0) {} circle (.05);
\node[above] at (b) {$v_2$};
    \draw[fill] (b) circle (.05);

\node (c) at (2,0) {} circle (.05);
\node[above] at (c) {$v_3$};
    \draw[fill] (c) circle (.05);

\draw[thick, -latex, bunired] (a) -- (b);

\draw[thick, -latex, bunired] (b) -- (c);
\end{tikzpicture}}
 and consider the gluing~$\tG$ of $\tG_1$ and $\tG_2$ over  \raisebox{-.2em}{\begin{tikzpicture}
\node (a) at (0,0) {} circle (.05);
\node[above] at (a) {$v_1$};
    \draw[fill] (a) circle (.05);

\node (b) at (1,0) {} circle (.05);
\node[above] at (b) {$v_2$};
    \draw[fill] (b) circle (.05);

\draw[thick, -latex, bunired] (a) -- (b);
\end{tikzpicture}}\hspace{-.1em}. Then, the poset~$P(\tG)$ is isomorphic to the Boolean poset~$\mathbb{B}(3)$, whose Hasse diagram is (the $1$-skeleton of) a $3$-dimensional cube. On the other hand, the Hasse diagram of $\nabla_{P(\tG)} (P(\tG_1), P(\tG_2))$ is the gluing of two copies of  $\Hasse(\mathbb{B}(2))$ along a copy of  $\Hasse (\mathbb{B}(1))$  -- that is two (empty) squares attached along an edge.

We now relate the path poset of a graph to the gluing of the path posets certain subgraphs. First, for a vertex $v \in V(\tG)$, consider the set~$E_v$ of edges $e_1, \dots e_n$ in $\tG$ incident to $v$, ordered so that $v = t(e_i)$, for $i = 1,\dots, k$, and $v = s(e_j)$, for $j = k+1 ,\dots, n$. Denote by $\tG_v^{k}$ the graph obtained by deleting the edges $e_{1}, \dots e_{k}$ from $\tG$, and set $\tG_v^{(h)} \coloneqq \tG_v^{k} \cup e_h$. 
 
\begin{thm}\label{finedimondo}
If the vertex~$v$ is a target for $k\geq 2$ edges, then 
   \[P(\tG)\cong \nabla_{P({\tG_{v}^{k}})}
   \left(P\left(\tG_v^{(1)}\right), \dots, P\left(\tG_v^{(k)}\right)\right) \ .\]
In other words, the path poset $P(\tG)$ is isomorphic to an iterated gluing of the path posets of the subgraphs~$\tG_v^{(1)}$, ... ,$\tG_v^{(k)}$ over the path poset of ${\tG_v^{k}}$.
\proof
We recall that, if $\tG'$ is a subgraph of a digraph $\tG$, then every multipath in $P(\tG')$ can be seen as a multipath in $P(\tG)$.
This means that  $P(\tG')$ can be seen as an  downward closed (and, in particular, faithful) subposet of $P(\tG)$.

To prove the theorem, we want to produce an isomorphism of posets
\[ P(\tG)\cong \nabla_{P\left({\tG_v^{k}}\right)}
   \left(P\left(\tG_v^{(1)}\right), \dots, P\left(\tG_v^{(k)}\right)\right).\]
First, we start by identifying the underlying sets and, then, we proceed with proving that the respective poset structures are isomorphic.

For the rest of the proof, denote by $T$ the set $\{ e \in E(\tG) \mid t(e) = v \}$.
Let $\tH$ be a multipath in~$P(\tG)$. We have two possible cases:
\begin{enumerate}[label = Case \arabic*:]
 \item $\tH$ does not contain any edge in $T$. Then, all simple paths in $\tH$ are simple paths in $\tG_v^{k}$, and thus $\tH \in P(\tG_v^{k})$.
 \item $\tH$ contains an edge $e_h \in T$. In this case $\tH$ cannot contain any other $e_j \in T$. Therefore, we have $\tH \in P(\tG_v^{(h)}) \setminus P(\tG_v^{k})$.
\end{enumerate}
 On the other hand, observe that any multipath $\tH \in \nabla_{P(\tG_v^{k})}\left(P(\tG_v^{(1)}), \dots, P(\tG_v^{(k)})\right)$ can be identified  with either an element of $P(\tG_v^{(j)}) \setminus P(\tG_v^{k})$, for some $j\in \{1,...,k\}$, or with an element of the poset $P(\tG_v^{k})$. Thus, we have a way to uniquely identify $\tH$ with an element of $P(\tG)$ and, consequently, the underlying sets of $P(\tG)$ and $\nabla_{P(\tG_v^{k})}\left(P(\tG_v^{(1)}), \dots, P(\tG_v^{(k)})\right)$. With abuse of notation, we denote this element again by $\tH$.
 
 Now, we want to prove that the general multipath $\tH$ covers the same elements both in $P(\tG)$ and $\nabla_{P(\tG_v^{k})}\left(P(\tG_v^{(1)}), \dots, P(\tG_v^{(k)})\right)$. This is obvious if $\tH$ is a multipath of $P(\tG_v^{k})$. Assume that~$\tH$ is in $P(\tG) \setminus P(\tG_v^{k})$. Then, there exists a unique $e_j \in T$ such that $e_j \in \tH$. A multipath covered by $\tH$ is then a multipath in $P(\tG_v^{(j)})$ and consequently the same covering relations hold in $\nabla_{P(\tG_v^{k})}\left(P(\tG_v^{(1)}), \dots, P(\tG_v^{(k)})\right)$. 
Finally, if $\tH \in \nabla_{P(\tG_v^{k})}\left(P(\tG_v^{(1)}), \dots, P(\tG_v^{(k)})\right)$, then all the elements covered by $\tH$ are contained in $P(\tG_v^{(j)})$ and it is possible to conclude the proof because~$P(\tG_v^{(j)})$ in an downward closed subposet of $P(\tG)$. 
 \endproof
\end{thm}



\section{Applications to Multipath Cohomology}\label{sec:applications}

In this section, we prove a Mayer-Vietoris-type theorem, and some aciclicity criteria for multipath cohomology with coefficients in a field. 

\subsection{The cohomology of the cone construction}

Recall first that a  poset $P=(S,\triangleleft)$ can be seen as a category $\mathbf{P}$. 
Given a functor $\mathcal{F}$ from $\mathbf{P}$ to the category of vector spaces, with some mild assumptions on $P$, there are well-defined cohomology groups~$\H^*(\mathbf{P};\mathcal{F})$ of $\mathbf{P}$ with coefficients in $\mathcal{F}$ -- cf.~\cite[Theorem~3.7]{primo} -- which, when $P=P(\tG)$ is the path poset of a digraph~$\tG$, gives the multipath cohomology. 

We  denote by $(C^*(P; \bK),\partial_P)$ the  cochain complex $(C^*_{\mathcal{F}_{\bK}}(P),\partial)$ associated to a poset~$P$ (and level~$\ell$) and to the functor $\mathcal{F}_{\bK}$ assigning a copy of $\bK$ to each object of $\mathbf{P}$ and the identity $\bK\to \bK$ to each arrow. Analogously, we denote by $\H^*(P;\bK)$ the cohomology groups of the cochain complex $(C^*(P; \bK),\partial_P)$. 
Recall that for a map $f\colon (A,\partial_A)\to (B,\partial_B)$ of cochain complexes, the mapping cone ${\rm Cone} f$ is the cochain complex defined in degree $n$ as  $({\rm Cone} f)^n\coloneqq A^{n+1}\oplus B^n$ and differential 
\[
\partial\coloneqq 
\begin{pmatrix}
\partial_{A}[1] & 0\\
f[1] & \partial_B
\end{pmatrix}
\]
where $\partial_A[1]$ and $f[1]$ represent the differential $\partial_A$ and the morphism $f$ shifted by one. 

\begin{thm}\label{HomCono} 
Let $\tG$ be a digraph. Then, we get an isomorphism of cochain complexes
\[C^*(\Cone P(\tG);\bK) \cong {\rm Cone}\left({\rm Id}_{C^*_\mu(\tG;\bK)}\right){[-1]} \]
where ${\rm Cone}\left({\rm Id}_{C^*_\mu(\tG;\bK)}\right)$ represents the mapping cone of the identity map on the cochain complex $C^*_\mu(\tG;\bK)$.
Consequently, we have $\H^*(\Cone P(\tG);\bK)= 0$.
\end{thm}
{To simplify the notation, we drop the reference to $\bK$ in the proof of the theorem.}
\proof
Recall that the cone of a poset $P$ is the product poset $P\times \mathbb{B}(1)$ -- cf.~Definition \ref{SqCone}. 
Consider the partition $P(\tG)\times \{ 0 \} \sqcup P(\tG)\times\{ 1\}  = \Cone P(\tG)$. 
Furthermore, we also have that~$ \ell_{\Cone P(\tG)}(\tH,i)= \ell_{P(\tG)\times \{ i \}}(\tH,i) + i = \ell_{P(\tG)}(\tH) + i$, for $i\in \{ 0,1\}$.
As a consequence, we have the isomorphism of graded $\bK$-vector spaces
\[C^*(\Cone P(\tG)) \cong C^*(P(\tG)\times \{0\})\oplus C^*(P(\tG)\times \{1\})[-1]\]
where $C^*(P(\tG))[-1]$ denotes the complex $C^*(P(\tG))$ shifted by one.
In turn, we have an isomorphism of posets $P(\tG)\cong P(\tG)\times\{ i\}$, for $i=0,1$, given by the identification $\tH \mapsto (\tH,i)$.
These identifications induce the isomorphism of graded $\bK$-vector spaces
\[C^{*}(P(\tG)\times \{0\})\oplus C^{*-1}(P(\tG)\times \{1\})\cong\Cmu^{*}(\tG)\oplus \Cmu^{*-1}(\tG) ={\rm Cone}\left({\rm Id}_{C^*_\mu(\tG;\bK)}\right)[-1] . \]

Now, we have to show that the above isomorphism commutes with the differentials.
The differential~$\partial$ of ${\rm Cone}\left({\rm Id}_{C^*_\mu(\tG;\bK)}\right)[-1]$ is defined as
\[
\partial\coloneqq 
\begin{pmatrix}
\partial_{\mu} & 0\\
\mathrm{Id}_{C^*_\mu(\tG;\bK)} & \partial_{\mu}[-1]
\end{pmatrix}\]
where $(\partial_{\mu}[{-1}])^n \coloneqq (-1)^n\partial_{\mu}^{{n-1}}$.
The differential of $C^*(\Cone P(\tG))  $, can be  explicitly written:
\[ \partial_{\Cone P(\tG)}(e_{(\tH, i)}) = (1-i)e_{(\tH,1)}+\sum_{\tH\prec \tH'} (-1)^{\epsilon(\tH,\tH') + i}e_{(\tH',i)} \]
where $\epsilon(\tH,\tH')$ is a sign assignment on $P(\tG)$ and $e_{(\tH, i)}$ is the generator of $C^*(P(\tG)\times \{i\})$ associated to the multipath $\tH$ in $\tG$. It is easy to check that $\epsilon'((\tH,i),(\tH',j)) \coloneqq \epsilon(\tH,\tH') + i$ is a sign assignment on $\Cone P(\tG)$. 

The isomorphism $C^*(\Cone P(\tG)) \cong {\rm Cone}\left({\rm Id}_{C^*_\mu(\tG;\bK)}\right)[-1]$ described above  commutes with these differentials, concluding the proof of the first part of the statement.
The vanishing result follows from the classical properties of the mapping cone of chain complexes \cite{Weibel}.
\endproof

We observe that the second part of Theorem~\ref{HomCono} can be alternatively proved using discrete Morse theory -- see~\cite[Chapter 11]{Kozlov} for an introduction;
if we consider the edges in $\Hasse(\Cone P(\tG)))$ with source in $ P(\tG)\times \{ 0\}$ and target in $P(\tG)\times \{ 1\}$, then these form an acyclic matching (\cite[Definition 11.1]{Kozlov}) whose edges are incident to all vertices of the graph~$\Hasse(\Cone P(\tG))$. It follows from the definitions and \cite[Theorem 11.24]{Kozlov} that the homology of $C^*(\Cone P(\tG);\bK)$ is trivial.

\subsection{A Mayer-Vietoris theorem}

The goal of this subsection is to prove a result which is the analogue, in the framework of multipath cohomology, of the classical Mayer-Vietoris theorem. In the classical statement, given a decomposition of a topological space as union of two subspaces, there is an induced long exact sequence of (co-)homology groups featuring also their intersections. In the setting of multipath cohomology, the \emph{r\^ole} played by unions of topological spaces is given by the gluing of posets. 
Recall that, for $\phi\colon \tG'\to\tG$, we have an induced morphism of posets $P\phi\colon P(\tG')\to P(\tG)$ -- see \cite[Remark~2.33]{primo}. Furthermore, 
\cite[Proposition~5.11]{primo}
%
%
%
%
gives us a map between the multipath cochain complex of a graph $\tG$, and the multipath cochain complex of a sub-graph $\tG_1$.

\begin{thm}\label{thm:MVposets}
Let $\tG, \tG_1, \tG_2$ be directed graphs, and $i_1\colon\tG \to \tG_1$ and $i_2\colon\tG \to \tG_2$ be  regular morphisms of digraphs. Then, we have a  short exact sequence of cochain complexes 
\begin{equation}\label{eq:MV}
0 \rightarrow C^*(\nabla_{P(\tG)} (P(\tG_1), \, P(\tG_2));\bK) \xrightarrow{I^*} \Cmu^*( \tG_1;\bK)\oplus \Cmu^*(\tG_2;\bK) \xrightarrow{J^*}  \Cmu^*( \tG; \bK)\rightarrow 0\end{equation}
inducing the long exact sequence
  \[\Scale[0.96]
  { \cdots \to  \Hmu^{i-1}(\tG;\bK) \to \H^{i}(\nabla_{P(\tG)} (P(\tG_1), \, P(\tG_2));\bK) \to \Hmu^{i}(\tG_1;\bK)  \oplus \Hmu^{i}( \tG_2;\bK)  \to \Hmu^{i}( \tG;\bK) \to \cdots} \]
 of  cohomology groups.
 \end{thm}

 \begin{proof} 
 We first observe that, as a consequence of the definition of gluing -- cf.~Definition~\ref{nabla} -- $P(\tG_1)$ and $P(\tG_2)$ are isomorphic to subposets of $\nabla_{P(\tG)} (P(\tG_1), \, P(\tG_2))$; call $\jmath_1$ and $\jmath_2$ these isomorphisms.
The inclusions of graphs $i_1\colon\tG \to \tG_1$ and $i_2\colon\tG \to \tG_2$ induce morphisms of posets $\imath_1\colon P(\tG) \rightarrow P(\tG_1) $ and $\imath_2\colon P(\tG) \rightarrow P(\tG_2) $. All the resulting morphisms fit into the following commutative square of posets
\[\begin{tikzcd}
P(\tG)\arrow[r, "\jmath_1"]\arrow[d, "\jmath_2"']& P(\tG_1) \arrow[d, "\imath_1"] \\
P(\tG_2)\arrow[r,  "\imath_2"' ]& \nabla_{P(\tG)} (P(\tG_1), \, P(\tG_2)) 
	\end{tikzcd} \]
and, induce a commutative diagram  of cochain complexes:
 \[\begin{tikzcd}
C^*(\nabla_{P(\tG)} (P(\tG_1), \, P(\tG_2));\bK)\arrow[r, "\imath_1^*"]\arrow[d, "\imath_2^*"']& \Cmu^*(\tG_1;\bK) \arrow[d, "\jmath_1^*"]  \\
\Cmu^*(\tG_2;\bK)\arrow[r,  "\jmath_2^*"' ]& \Cmu^*(\tG;\bK) 
	\end{tikzcd}\]
Now, for every $n\in \bN$, consider the maps 
\[ I^n \coloneqq \imath^n_1 \oplus \imath^n_2 \colon C^n(\nabla_{P(\tG)} (P(\tG_1), \, P({\tG_2}));\bK) \rightarrow  \Cmu^n(\tG_1;\bK) \oplus \Cmu^n(\tG_2;\bK) \ ,\]
and 
\[ J^n \coloneqq \jmath^n_1 - \jmath^n_2 \colon  \Cmu^n(\tG_1;\bK) \oplus \Cmu^n(\tG_2;\bK) \rightarrow \Cmu^n(\tG;\bK) \ .\]
We proceed with proving that the sequence of complexes in Equation~\eqref{eq:MV} is exact. 
The cochain complexes $C^*(\nabla_{P(\tG)} (P(\tG_1), \, P_{G_2});\bK)$, $\Cmu^n(\tG_1;\bK)$, and $\Cmu^n(\tG_2;\bK)$ have bases indexed by the elements of the corresponding poset (namely, $\nabla_{P(\tG)} (P(\tG_1), \, P_{G_2})$, $P(\tG_1)$, and $P(\tG_2)$, respectively).
We denote by $b_\tH$ the element of each of these bases corresponding the multipath~$\tH$.

With this notation, a generic element $x$ of $C^n(\nabla_{P(\tG)} (P(\tG_1), \, P_{G_2});\bK)$ is of the form 
\[ x = \sum_{\tH \in P(\tG)} \alpha_{\tH} b_{\imath_{1}\circ \jmath_1(\tH)} + \sum_{\tH' \in P(\tG_1)\setminus \jmath_1(P(\tG))} \beta_{\tH'} b_{\imath_{1}(\tH')} + \sum_{\tH'' \in P(\tG_2)\setminus \jmath_2(P(\tG))} \gamma_{\tH'} b_{\imath_{2}(\tH'')}  \]
with $\ell(\tH) =\ell(\tH') =n$. Note that $\imath_{1}\circ \jmath_1 = \imath_{2}\circ \jmath_2$, thus $ b_{\imath_{1}\circ \jmath_1(\tH)} = b_{\imath_{2}\circ \jmath_2(\tH)}$ for each $\tH \in P(\tG)$.
We are now ready to verify that $I^n$ is injective. 
With respect to the basis  above, we can write
 \[
  { I^n(x) =
  \left( \sum_{\tH \in P(\tG)} \alpha_{\tH} b_{\jmath_1(\tH)} + \sum_{\tH' \in P(\tG_1)\setminus \jmath_1(P(\tG))} \beta_{\tH'} b_{\tH'}, \sum_{\tH \in P(\tG) } \alpha_{\tH} b_{\jmath_2(\tH)} + \sum_{\tH'' \in P(\tG_2)\setminus \jmath_{2}(P(\tG)) } \gamma_{\tH''} b_{\tH''}\right) } \ .\]
 It follows that $I^n(x) = 0$ if and only if all the coefficients $\alpha_\tH$, $\beta_\tH$ and $\gamma_{\tH'}$, and thus $x$, are zero.
 
 It is left to show that $J^n$ is surjective and that $\mathrm{Ker}(I^n) = \mathrm{Im}(J^n)$. First, we write $J^n$ explicitly, as follows;
 \[ J^n\left( \sum_{\tH' \in P(\tG_1)} \alpha_{\tH'} b_{\tH'} , \sum_{\tH'' \in P(\tG_2)} \beta_{\tH''} b_{\tH''}\right) = \sum_{\tH \in P(\tG)} (\alpha_{\jmath_1(\tH)} - \beta_{\jmath_2(\tH)} )b_{\tH} \ . \]
Now, for $J^n(x)$ to be zero, we must have that $\alpha_{\jmath_1(\tH)} = \beta_{\jmath_2(\tH)}$, for all $\tH \in P(\tG)$, independently on the values of~$\alpha_{\tH'}$ and $\beta_{\tH''}$ for~$\tH'\notin \jmath_1(P(\tG))$, and $\tH'' \notin \jmath_2(P(\tG))$. It follows that the kernel $J^n$ is precisely the image of~$I^n$. Finally, $J^n$ is clearly surjective, concluding the proof.
 \end{proof}

The following observation is straightforward:

\begin{rem}\label{cohnabla}
Under the assumptions of Theorem~\ref{thm:MVposets}, if $\tG_1$ and $\tG_2$ have trivial multipath cohomology, then~$\H^{n}(\nabla_{P(\tG)} (P(\tG_1), \, P(\tG_2));\bK)$ is isomorphic to  $\Hmu^{n-1}(\tG;\bK)$ for all $n$.
\end{rem}

By iterated applications of the Mayer-Vietoris long exact sequence for the multipath cohomology, we obtain the corollary:

\begin{cor}\label{nullgluing}
Let $\tG_1, \dots, \tG_n$ be digraphs and suppose that for all $j\in \{1,\dots,n\}$ there exists a regular morphism $i_j\colon \tG \rightarrow \tG_j $. If all the cohomology groups $\Hmu^*(\tG;\bK)$ and $\Hmu^*(\tG_j;\bK)$ vanish for all $j$, then \[\H^*\left(\nabla_{P(\tG)} (P(\tG_1), \dots  P(\tG_n));\bK\right)=0 \ .\]
\end{cor}

{In the next subsection, we apply the results shown in this subsection to obtain vanishing criteria for multipath cohomology.}

\subsection{Acyclicity Criteria and Examples}
The aim of this subsection is to find sufficient conditions on a graph $\tG$, for $\Hmu^*(\tG;\bK)$ to be trivial. Using the same notation as in Theorem~\ref{finedimondo}, we obtain the first vanishing criterion;
\begin{crit}\label{crit:FDM} Assume that a digraph~$\tG$ satisfies the following conditions:
\begin{enumerate}
\item there exists a vertex $v$ that is the target (or source) of $k \geq 2$ edges;
\item the graphs $\tG_v^{(1)},\dots,\tG_v^{(k)}, \tG_v^{k} $ have trivial multipath cohomology.
\end{enumerate}
Then, $\Hmu^*(\tG;\bK)=0$.
\end{crit}
\begin{proof}
By Theorem~\ref{finedimondo}, we have the isomorphism of posets \[P(\tG)\cong \nabla_{P\left({\tG_v^{k}}\right)} \left(P\left(\tG_v^{(1)}\right), \dots, P\left(\tG_v^{(k)}\right)\right) \ .\] Then, the statement follows from Corollary~\ref{nullgluing}.
\end{proof}
{At a first glance Criterion~\ref{crit:FDM} looks quite technical. However, it can easily applied in practice. 
\begin{example}
Let $\tG$ be the graph shown in Figure~\ref{fig:H_n1n}. Recall from Example~\ref{ex:In} that a coherent linear graph has trivial cohomology. Then, by applying Criterion~\ref{crit:FDM}, we get that $\Hmu^*(\tG;\bK)=0$.
\begin{figure}[H]
    \centering
    \begin{tikzpicture}[thick]
        \node (a1) at (0,0.5) {};
        \node (b1) at (1.25,1) {};
        \node (b2) at (2.75,1) {};
        \node (c1) at (0,1.5) {};
        \node (c2) at (4,1) {};
        
        \node[below left]  at (0,0.5) {};
        \node[above] at (1.25,1) {$v$};
        \node[above] at (0,0.7) {$\vdots$};

        \node[above] at (2.75,1) {};

        \draw[fill] (a1) circle (0.03);
        \draw[fill] (b1) circle (0.03);
        \draw[fill] (c1) circle (0.03);
        \draw[fill] (b2) circle (0.03);
        \draw[fill] (c2) circle (0.03);

        \draw[bunired,-latex](b1) -- (b2);
        \draw[bunired,latex-] (b1) -- (a1);
        \draw[bunired,-latex] (c1) -- (b1);
        \draw[bunired,-latex] (b2) -- (c2);
        
    \end{tikzpicture}
    \caption{A graph with $n$ edges with target $v$ glued to a linear graph.}
    \label{fig:H_n1n}
\end{figure}
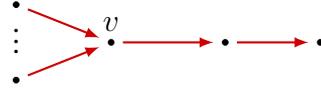
\end{example}}

Criterion~\ref{crit:FDM} says that we can infer the vanishing of multipath cohomology by looking at smaller pieces in the graph.
Our second criterion is based on the existence of suitably ``embedded'' subgraphs. To formalise this we need the notion of $\nu$-equivalence.

\begin{defn}\label{def:nu-equiv}
A  morphism of directed graphs $\phi\colon \tG \rightarrow \tG'$ is a \emph{$\nu$-equivalence} away from a (possibly empty) set of vertices $V \in V(\tG)$, if the valence of $v \in V(\tG)$ is the same as the valence of $\phi(v) \in V(\tG')$, for every $v \in V(\tG) \setminus V$.
\end{defn}

Note that a $\nu$-equivalence $\phi\colon \tG \rightarrow \tG'$ away from an empty set of vertices, is just the inclusion of $\tG$ as a connected component of $\tG'$.

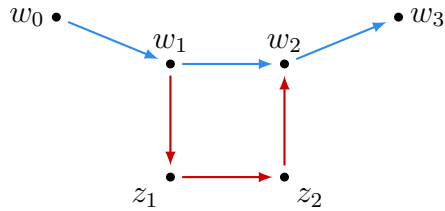
\begin{figure}[H]
    \centering
    \begin{tikzpicture}[thick,scale = 1.5]
        \node (a1) at (1,0) {};
        \node (a2) at (2,0) {};
        \node (b1) at (1,1) {};
        \node (b2) at (2,1) {};
        
        \node (c1) at (0,1.414) {};
        \node (c2) at (3,1.414) {};
        
        \node[below left]  at (1,0) {$z_1$};
        \node[below right] at (2,0) {$z_2$};
        \node[above] at (1,1) {$w_1$};
        \node[above] at (2,1) {$w_2$};
        
        \node[left] at (0,1.414) {$w_0$};
        \node[right] at (3,1.414) {$w_3$};
        
        \draw[fill] (a1) circle (0.03);
        \draw[fill] (b1) circle (0.03);
        \draw[fill] (c1) circle (0.03);
        \draw[fill] (a2) circle (0.03);
        \draw[fill] (b2) circle (0.03);
        \draw[fill] (c2) circle (0.03);

        \draw[bunired,-latex] (a1) -- (a2);
        \draw[bluedefrance,-latex] (b1) -- (b2);
        \draw[bunired,-latex] (a2) -- (b2);
        \draw[bunired,-latex] (b1) -- (a1);
        \draw[bluedefrance,-latex] (c1) -- (b1);
        \draw[bluedefrance,-latex] (b2) -- (c2);
        
    \end{tikzpicture}
    \caption{A graph $\tG$ and two copies of $\tI_3$ inside it. The morphism of digraphs~$\phi_{r}$, whose image is the red copy of $\tI_3$, is a $\nu$-equivalence away from $v_0$ and $v_3$, while morphism of digraphs $\phi_{b}$, whose image is the blue copy of $\tI_3$, is a $\nu$-equivalence away from $v_1$ and $v_2$.}
    \label{fig:nu_eq}
\end{figure}
\begin{example}
Let $\tG$ be the graph in Figure~\ref{fig:nu_eq}, and denote by $\tI_3$ the linear digraph illustrated in Figure~\ref{fig:nstep}, with vertices labelled as in the aforementioned figure.
Consider the morphisms of digraphs
$ \phi_{b},\phi_{r}\colon \tI_3 \to \tG, $
defined as follows: $\phi_{b}(v_i) = w_i$, for each $i\in\{0,1,2,3\}$, and
\[ \phi_{r}(v_i) = \begin{cases} w_1 & \text{if }i=0; \\ z_i & \text{if }i=1,2;\\ w_2 & \text{if }i=3.\end{cases} \]
Then, $\phi_b$ is a $\nu$-equivalence {away from} $v_1$ and $v_2$, while $\phi_b$ is so, {away from} $v_0$ and~$v_3$.
\end{example}
 In order to state our next criterion we need a new family of graphs $\tH_{n,m}$, illustrated in Figure~\ref{fig:H_nm}.
\begin{figure}[H]
    \centering
    \begin{tikzpicture}[thick]
        \node (a1) at (0,0.5) {};
        \node (a2) at (4,0.5) {};
        \node (b1) at (1.25,1) {};
        \node (b2) at (2.75,1) {};
        
        \node (c1) at (0,1.5) {};
        \node (c2) at (4,1.5) {};
        
        \node[below left]  at (0,0.5) {$w_n$};
        \node[below right] at (4,0.5) {$x_m$};
        \node[above] at (1.25,1) {$v_0$};
        \node[above] at (2.75,1) {$v_1$};
        \node[ ] at (0,1) {\raisebox{.5em}{$\vdots$}};
        \node[ ] at (4,1) {\raisebox{.5em}{$\vdots$}};
        
        \node[above left] at (0,1.5) {$w_1$};
        \node[above right] at (4,1.5) {$x_1$};
        
        \draw[fill] (a1) circle (0.03);
        \draw[fill] (b1) circle (0.03);
        \draw[fill] (c1) circle (0.03);
        \draw[fill] (a2) circle (0.03);
        \draw[fill] (b2) circle (0.03);
        \draw[fill] (c2) circle (0.03);

        \draw[bunired,-latex](b1) -- (b2);
        \draw[bunired,latex-] (a2) -- (b2);
        \draw[bunired,latex-] (b1) -- (a1);
        \draw[bunired,-latex] (c1) -- (b1);
        \draw[bunired,-latex] (b2) -- (c2);
        
    \end{tikzpicture}
    \caption{The graph $\tH_{n,m}$.}
    \label{fig:H_nm}
\end{figure}
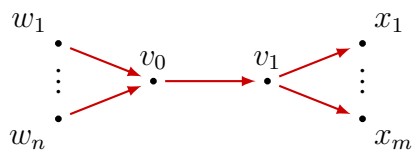
 \begin{crit}\label{crit:istmo}
 Let $\tG$ be a digraph and assume that there exists a {$\nu$-equivalence} $\phi\colon \tH_{n,m} \to \tG$ away from $w_1,...,w_n$ and $x_1, ...,x_m$.
If $(\phi(v_0),\phi(v_1))$ is not contained in any coherently oriented cycle\footnote{That is the image of a regular morphism $\tP_n \to \tG$, for some $n$.} of~$\tG$, then~$\Hmu^*(\tG;\bK)=0$.
\end{crit}
\proof
 Denote by  $e$  the edge $(\phi(v_0),\phi(v_1))$ of $\phi(\tH_{n,m})$  and set $\tG'\coloneqq \tG \setminus \{e\}$. 
We want to show that the poset~$P(\tG)$ is isomorphic to the cone~$\Cone P(\tG')$.
Consider the following subset of the path poset: $P(\tG)_0\coloneqq \{\tH \in P(\tG)\mid e \notin \H\}$ and  $P(\tG)_1\coloneqq\{\tH \in P(\tG)\mid e \in \H\}$ endowed with the poset structure induced by $P(\tG)$. 
We claim that $P(\tG)_i \cong P(\tG')\times \{ i \} \subset \Cone P(\tG')$, for $i=0,1$.
The isomorphism $P(\tG)_0 \cong P(\tG')$ is clear.
Now, to identify  $P(\tG)_1$ with $P(\tG')\times \{1\}$ we observe that:
\begin{enumerate}
    \item the edge~$e$ is coherently oriented with the other edges in $\phi(\tH_{n,m})$,
    \item no multipath contains two edges of the form $(\phi(x_i),\phi(v_1))$ nor contains two edges of the form~$(\phi(v_1),\phi(w_i))$,
    \item and $e$ is not contained in a coherently oriented cycle.
\end{enumerate}
It follows that $\tH \cup e$ is a multipath for every $\tH \in P(\tG)_0$, which implies that the map sending $\tH \in P(\tG)_0$  to  $\tH \cup e \in P(\tG)_1$ is well-defined. 
Note that this map is also a bijection, and that it preserves the inclusions, i.e.~if $\tH \prec \tH' \in P(\tG)_0$ then  $\tH \cup e \prec \tH'\cup e \in P(\tG)_1$. Consequently, we have a sequence of isomorphism of posets $P(\tG)_1 \cong P(\tG)_0 \cong P(\tG')\times \{ 0 \}  \cong P(\tG')\times \{ 1 \}$.
To complete the proof that $P(\tG) \cong \Cone P(\tG')$ we have to check that, under the above chain of identifications, a covering relation between two elements $\tH \in P(\tG)_0$ and $\tH' \in P(\tG)_1$ corresponds uniquely to a covering relation between the corresponding elements $(\tH,0) \in P(\tG')\times \{0\}$ and  $(\tH' \setminus \{e\},1)\in P(\tG')\times \{1\}$. 
This follows directly from  the description of the covering relation in $\Cone P(\tG')$.  As a consequence, the posets $P(\tG)$ and $\Cone P(\tG')$ are isomorphic. From Theorem~\ref{HomCono},  it follows  $\Hmu^*(\tG;\bK)=0$.
\endproof

\begin{rem}\label{rem:tail}
Criterion~\ref{crit:istmo} also holds if either $n =0$ or $m=0$.
In these cases, we say that the graph $\tG$ has \emph{a coherent tail}. With this terminology, we can restate the special case of Criterion~\ref{crit:istmo} when either $n=0$ or $m=0$ as follows; if $\tG$ has a coherent tail, then $\Hmu^*(\tG;\bK)=0$.
\end{rem}
We now provide some examples.
\begin{example}
 An \emph{arborescent graph} (or \emph{arborescence}) is a directed graph in which there is a vertex~$r$, called \emph{root}, and there is exactly one directed path from $r$ to any other vertex. If an arborescent graph~$\tT$ has a vertex at distance~$2$ from the root (i.e.~the unique path joining them has length $2$), then {up to orientation reversing of the edges} there is a $\nu$-equivalence $\tH_{0,m}\to \tT$ away from {$v_0$, for some $m>0$} . Applying Criterion~\ref{crit:istmo}, we get $\Hmu^*(\tT;\bK)=0$.
\end{example}
\begin{figure}[h]
	\begin{tikzpicture}[scale=0.4][baseline=(current bounding box.center)]
		\tikzstyle{point}=[circle,thick,draw=black,fill=black,inner sep=0pt,minimum width=2pt,minimum height=2pt]
		\tikzstyle{arc}=[shorten >= 8pt,shorten <= 8pt,->, thick]
		
		\node  (v0) at (0,0) {};
		\node[below] at (0,0) {$v_0$};
		\draw[fill] (0,0)  circle (.05);
		\node  (v1) at (0,3) { };
		\node[above]  at (0,3) {$v_1$};
		\draw[fill] (0,3)  circle (.05);
		\node  (v2) at (3,3) { };
		\node[above] at (3,3) {$v_{3}= v$};
		\draw[fill] (3,3)  circle (.05);
		\node  (v3) at (3,0) { };
		\node[below]  at (3,0) {$v_{2}$};
		\draw[fill] (3,0)  circle (.05);

		\node  (v5) at (6,3) { };
		\node[above] at (6,3) {$v_{5}$};
		\draw[fill] (6,3)  circle (.05);
		\node  (v4) at (6,0) { };
		\node[below]  at (6,0) {$v_{4}$};
		\draw[fill] (6,0)  circle (.05);
		
		\node  (v7) at (9,3) { };
		\node[above] at (9,3) {$v_{n-2}$};
		\draw[fill] (9,3)  circle (.05);
		\node  (v6) at (9,0) { };
		\node[below]  at (9,0) {$v_{n-3}$};
		\draw[fill] (9,0)  circle (.05);
		
		\node  (v9) at (12,3) { };
		\node[above] at (12,3) {$v_{n}$};
		\draw[fill] (12,3)  circle (.05);
		\node  (v8) at (12,0) { };
		\node[below]  at (12,0) {$v_{n-1}$};
		\draw[fill] (12,0)  circle (.05);
		
		\node  (p) at (7.5,1.8) { };
		\node[below]  at (7.5,1.8) {$\dots$};

		\draw[thick, bunired, -latex] (v1) -- (v0);
		\draw[thick, bunired, -latex] (v2) -- (v1);
		\draw[thick, bunired, -latex] (v2) -- (v3);
		\draw[thick, bunired, -latex] (v0) -- (v3);
		\draw[thick, bunired, -latex] (v5) -- (v4);
		\draw[thick, bunired, -latex] (v5) -- (v2);
		\draw[thick, bunired, -latex] (v3) -- (v4);
		
		\draw[thick, bunired, -latex] (v7) -- (v6);

		\draw[thick, bunired, -latex] (v9) -- (v7);
		\draw[thick, bunired, -latex] (v9) -- (v8);
		\draw[thick, bunired, -latex] (v6) -- (v8);

	\end{tikzpicture}
	\caption{The graph $\mathtt{O}_n$.}
	\label{fig:manysquares}
\end{figure}
\begin{example}
 Let $\mathtt{O}_n$ be the graph in Figure~\ref{fig:manysquares}, and set $v= v_3$, $e_1 = (v_3,v_2)$, and $e_2 = (v_3,v_1)$. Using the same notation of Criterion~\ref{crit:FDM}, we have that $\tG^2_v$ and $\tG_v^{(1)}$ have a coherent tail, while there is a $\nu$-equivalence $\tI_3 (\cong\tH_{1,1}) \to \tG^{(2)}_{v}$  whose image is the sub-graph with edges $(v_3,v_1),(v_1,v_0),(v_0, v_2)$. Thus, by Criterion \ref{crit:istmo}, $\tG^2_v$, $\tG_v^{(1)}$, and $\tG_v^{(2)}$ have trivial cohomology. 
 By Criterion~\ref{crit:FDM}, it follows that $\Hmu^{*}({\tt O}_n;\bK) =0$.
 \end{example}
 Another class of graphs, important to us, is given by the dandelion graphs:
\begin{defn}\label{dandelion}
Let $\tD_{n,m}$ the graph on $(n+m+1)$ vertices,  and $(m+n)$ edges defined as follows:
\begin{enumerate}
    \item $V(\tD_{n,m}) = \{ v_{0}, w_{1} ,...., w_{n}, x_{1}, ..., x_{m} \}$;
    \item $E(\tD_{n,m}) = \{ (w_i,v_0), (v_{0},x_j) \mid i=1,...,n; j=1,...,m \}$.
\end{enumerate}
The digraph $\tD_{n,m}$ is called a \emph{dandelion graph}.
\end{defn}
\begin{figure}[h]
	\begin{tikzpicture}[scale=0.5][baseline=(current bounding box.center)]
		\tikzstyle{point}=[circle,thick,draw=black,fill=black,inner sep=0pt,minimum width=2pt,minimum height=2pt]
		\tikzstyle{arc}=[shorten >= 8pt,shorten <= 8pt,->, thick]
		
		\node (v0) at (0,0) {};
		\node[above] at (0,0.1) {$v_0$};
		\draw[fill] (v0)  circle (.05);
		\node (w1) at (-2,1.5) {};
		\node[above] at (-2,1.5) {$w_1$};
		\draw[fill] (-2,1.5)  circle (.05);
		\node  (w2) at (-2,0) { };
		\node[left]  at (-2,0) {$w_{2}$};
		\draw[fill] (-2,0)  circle (.05);
		\node  (w3) at (-2,-1.5) { };
		\node[below]  at (-2,-1.5) {$w_{3}$};
		\draw[fill] (-2,-1.5)  circle (.05);
		
		\node  (x1) at (2,1.3) { };
		\node[right]  at (2,1.3) {$x_1$};
		\draw[fill] (2,1.3)  circle (.05);
		
		\node  (x2) at (2,-1.3) { };
		\node[right]  at (2,-1.3) {$x_2$};
		\draw[fill] (2,-1.3)  circle (.05);

		\draw[thick, bunired, -latex] (w1) -- (v0);
		\draw[thick, bunired, -latex] (w2) -- (v0);
		\draw[thick, bunired, -latex] (w3) -- (v0);
		\draw[thick, bunired, -latex] (v0) -- (x1);
		\draw[thick, bunired, -latex] (v0) -- (x2);
	\end{tikzpicture}
	\caption{The  graph $\tD_{3,2}$.}
	\label{fig:nmgraph}
\end{figure}
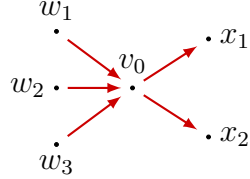
In other words we have a single $(m+n)$-valent vertex $v_0$, all remaining vertices are univalent, there are $n$ edges with target $v_0$, and there are $m$ edges with source $v_0$ -- cf.~Figure~\ref{fig:nmgraph}. 

\begin{rem}\label{rem:dandelion reversal}
If we reverse the orientation of all the edges in $\tD_{k,n-k}$ we obtain $\tD_{n-k,k}$. Then, we have an isomorphism  $P(\tD_{k,n-k}) \cong P(\tD_{n-k,k})$; hence, $\Hmu^n(\tD_{k,n-k};\bK) \cong \Hmu^n(\tD_{n-k,k};\bK)$ for all $n$. 
\end{rem}

The dandelion digraph~$\tD_{n,0}$ is a source with $n$ edges, and the dandelion digraph~$\tD_{0,n}$ is a sink with $n$ edges. 
The dandelion digraph $\tD_{1,1}$ is the 2-step graph $\tI_2$. 
\begin{rem}\label{rem:0dandelion}
If $\tG$  is a source or a sink with $n\geq2$ edges, then $\dim \Hmu^1(\tG;\bK)=n-1$  and $\dim \Hmu^i(\tG;\bK)=0$ for $i\neq 1$; in fact the path poset of a sink (or a source) with $n$ edges is given by a single multipath of length $0$, and $n$ multipaths of length $1$. It follows that, in the case at hand, the multipath chain complex is very simple:
\[0 \rightarrow C_{\mu}^0({\tt D}_{0,n};\bK) \cong  \bK \overset{d^0}{\longrightarrow} C_{\mu}^1({\tt D}_{0,n};\bK) \cong \bK^n \rightarrow 0.\]
Furthermore, the map $d^0$ is injective (since $d^0$ is the map $x\mapsto (\pm x, ...., \pm x)$ for an appropriate choice of signs), giving trivial {co}homology in degree $0$ and a {co}homology group of dimension~$n-1$ in degree $1$.
\end{rem}
An immediate consequence of Remark~\ref{rem:0dandelion} is that we can have multipath cohomology groups of arbitrary dimension (as $\bK$-vector space).  
\begin{prop}\label{prop:1dandelion}
Let $n\geq 1$ be an integer. Then, $\Hmu^*(\tD_{1,n-1};\bK) = 0$.
\end{prop}
\proof
Note that $\tD_{1,0}= \tI_1$, and that $\tD_{1,n-1} = \tH_{0,n-1}$. In the former case, the cohomology is trivial by direct computation. In the latter case, the cohomology is trivial by Criterion~\ref{crit:FDM}, as~$\tD_{1,n-1}$ has a coherent tail -- cf.~Remark~\ref{rem:tail}. 
\endproof
\begin{prop}\label{prop:kdandelion}
Let $n > 2$ and $k>0$ be two integers such that $n>k$. Then, we have
\[ \Hmu^i(\tD_{k,n-k};\bK) \cong \begin{cases} \bK^{(k-1)(n-k-1)} & \text{if }i=2\\ 0 & \text{otherwise}. \end{cases}\]
\end{prop}
\proof
We proceed by induction on $n$. If $n=3$,  Proposition~\ref{prop:1dandelion} and Remark~\ref{rem:dandelion reversal} imply  $\Hmu^*(\tD_{2,1};\bK)\cong\Hmu^*(\tD_{1,2};\bK) =0$.
Now, set $\tG = \tD_{k,n-k}$ and $v = v_0$.
By Theorem \ref{finedimondo}, we have
\[ P(\tD_{k,n-k}) \cong \nabla_{P(\tD_{0,n-k})}\underbrace{\left(P(\tD_{1,n-k}),\dots, P(\tD_{1,n-k})\right)}_{\text{$k$ copies}} \overset{(*)}{\cong} \nabla_{P(\tD_{0,n-k})}\left(P(\tD_{1,n-k}),P(\tD_{k-1,n-k})\right)\]
where the isomorphism marked with $(*)$ follows from the associativity of the gluing (and from Theorem \ref{finedimondo}).
By applying Theorem~\ref{thm:MVposets} and the inductive hypothesis, it follows that $\Hmu^i(\tD_{k,n-k};\bK) =0$ for $i\geq 3$, and that the sequence
\[ 0 \rightarrow \Hmu^1(\tD_{0,n-k};\bK) \rightarrow \Hmu^2(\tD_{k,n-k};\bK) \rightarrow \Hmu^2(\tD_{k-1,n-k};\bK) \rightarrow 0\ ,\]
is exact. 
The assertion is now immediate from the fact that the dimension function is additive on short exact sequences.
\endproof
We conclude the section showing that there exist directed graphs with multipath cohomology of arbitrary high rank in arbitrary high degree. 

\begin{figure}[h]
	\centering
	\begin{subfigure}[b]{0.3\textwidth}
		\centering
\begin{tikzpicture}[baseline=(current bounding box.center)]
		\tikzstyle{point}=[circle,thick,draw=black,fill=black,inner sep=0pt,minimum width=2pt,minimum height=2pt]
		\tikzstyle{arc}=[shorten >= 8pt,shorten <= 8pt,->, thick]
        \node[above] at (-.5,1) {$\tt{G'}$};
        \node at (-1,0) {$\tt{G''}$};
		\draw[dashed] (-1,0)  circle (.5);
		\draw[dashed] (-1,0)  circle (2.5 and 1);
		\node[above] (v0) at (0,0) {$v_0$};
		\draw[fill] (0,0)  circle (.05);
		\node[fill, white, below] at (1.5,0) {}  circle (.25);
		\node[below] (v1) at (1.5,0) {$w$};
		\draw[fill] (1.5,0)  circle (.05);
		\node[above] (v2) at (3,1) {$v_{1}$};
		\node[above] at (2.25,.55) {$e_{1}$};
		\node[below] at (2.25,-.55) {$e_{2}$};
		\draw[fill] (3,1)  circle (.05);
		\node[above] (v3) at (3,-1) {$v_{2}$};
		\draw[fill] (3,-1)  circle (.05);
		\draw[thick, bunired] (-0.35,0) -- (-0.15,0);
		\draw[thick, bunired, latex-] (1.35,0) -- (0.15,0);
		\draw[thick, bunired, -latex] (1.65,0.05) -- (2.85,0.95);
		\draw[thick, bunired, -latex] (1.65,-0.05) -- (2.85,-0.95);
	\end{tikzpicture}
		\caption{\phantom{  A   }}
	\end{subfigure}
	\hspace{0.1\textwidth}
	\begin{subfigure}[b]{0.5\textwidth}
	\centering
	\begin{tikzpicture}[baseline=(current bounding box.center)]
		\tikzstyle{point}=[circle,thick,draw=black,fill=black,inner sep=0pt,minimum width=2pt,minimum height=2pt]
		\tikzstyle{arc}=[shorten >= 8pt,shorten <= 8pt,->, thick]
		
        \node[above] at (-.5,1) {$\tt{G'}$};
		\node at (-1,0) {$\tt{G''}$};
		\draw[dashed] (-1,0)  circle (.5);
		\draw[dashed] (-1,0)  circle (2.5 and 1);
		\node[above] (v0) at (0,0) {$v_0$};
		\draw[fill] (0,0)  circle (.05);
		\node[fill, white, below] at (1.5,0) {}  circle (.25);
		\node[below] (v1) at (1.5,0) {$w$};
		\draw[fill] (1.5,0)  circle (.05);
		\node[above] (v2) at (3,1) {$v_{1}$};
		\node[above] at (2.25,.55) {$e_{1}$};
		\node[below] at (2.25,-.55) {$e_{2}$};
		\draw[fill] (3,1)  circle (.05);
		\node[above] (v3) at (3,-1) {$v_{2}$};
		\draw[fill] (3,-1)  circle (.05);

		\draw[thick, bunired] (-0.35,0) --  (-0.15,0);
		\draw[thick, bunired, -latex] (1.35,0) -- (0.15,0);
		\draw[thick, bunired, -latex] (2.85,0.95) -- (1.65,0.05);
		\draw[thick, bunired, -latex] (2.85,-0.95) -- (1.65,-0.05);
	\end{tikzpicture}
				\caption{ \phantom{C }}
\end{subfigure}
	\caption{Configurations of a digraph $\tG$ with subgraphs $\nu$ equivalent away from $v_0$ to  $\tD_{1,2}$ and $\tD_{2,1}$.}
\label{fig:y}
	\begin{subfigure}[b]{0.3\textwidth}
\begin{tikzpicture}[baseline=(current bounding box.center)]
		\tikzstyle{point}=[circle,thick,draw=black,fill=black,inner sep=0pt,minimum width=2pt,minimum height=2pt]
		\tikzstyle{arc}=[shorten >= 8pt,shorten <= 8pt,->, thick]
		
		\node at (-1,-2) {$\tt{G''}$};
		\draw[] (-1,-2)  circle (.5);
		\node[above] (v0) at (0,-2) {$v_0$};
		\draw[fill] (0,-2)  circle (.05);
		\node[above] (v1) at (1.5,-2) {$w$};
		\draw[fill] (1.5,-2)  circle (.05);
		\node[above] (v2) at (3,-1) {$v_{1}$};
		\draw[fill] (3,-1)  circle (.05);
		\node[above] (v3) at (3,-3) {$v_{2}$};
		\draw[fill] (3,-3)  circle (.05);
		
		\draw[thick, bunired] (-0.35,-2) -- (-0.15,-2);
		\draw[thick, bunired, -latex] (0.15,-2) -- (1.35,-2);
	\end{tikzpicture}
		\caption{\phantom{  A   }}
	\end{subfigure}
	\hspace{0.1\textwidth}
\begin{subfigure}[b]{0.3\textwidth}
	\begin{tikzpicture}[baseline=(current bounding box.center)]
		\tikzstyle{point}=[circle,thick,draw=black,fill=black,inner sep=0pt,minimum width=2pt,minimum height=2pt]
		\tikzstyle{arc}=[shorten >= 8pt,shorten <= 8pt,->, thick]
		
		\node at (-1,-2) {$\tt{G''}$};
		\draw[] (-1,-2)  circle (.5);
		\node[above] (v0) at (0,-2) {$v_0$};
		\draw[fill] (0,-2)  circle (.05);
		\node[above] (v1) at (1.5,-2) {$w$};
		\draw[fill] (1.5,-2)  circle (.05);
		\node[above] (v2) at (3,-1) {$v_{1}$};
		\draw[fill] (3,-1)  circle (.05);
		\node[above] (v3) at (3,-3) {$v_{2}$};
		\draw[fill] (3,-3)  circle (.05);

		\draw[thick, bunired] (-0.35,-2) --  (-0.15,-2);
		\draw[thick, bunired, -latex] (1.35,-2) -- (0.15,-2);
	\end{tikzpicture}
				\caption{ \phantom{C }}
\end{subfigure}
	\caption{The graph $\tG'$ as subgraph of $\tG$ in the two cases we consider.}\label{fig:yy}
\end{figure}
\begin{lem}\label{lem:shift homology construction}
Given a digraph $\tG'$ with  a vertex~$w$ of valence~$1$, there exists a digraph~$\tG$ such that $\Hmu^*(\tG';\bK)\cong\Hmu^{*-1}(\tG;\bK)$.
\end{lem}
\begin{proof}
 Let $e \in E(\tG')$ be the only edge incident to $w$. We define $\tG$ as follows;  if $s(e)=w$, glue  a linear sink over $w$ to $\tG'$  -- cf.~Figure \ref{fig:nnstep}, otherwise glue  a linear source.

In the notations of Figures \ref{fig:y} and \ref{fig:yy}, by Theorem \ref{finedimondo} the path poset of $\tG$ is the gluing of the path posets $P(\tG_w^{(1)})$ and $P(\tG_w^{(2)})$ over $P(\tG')$. By Remark~\ref{rem:tail} the subgraphs $ \tG_w^{(1)}$ and $ \tG_w^{(2)}$ have trivial cohomology; hence, by Remark~\ref{cohnabla}, we have $\Hmu^*(\tG;\bK)=\Hmu^{*-1}(\tG';\bK)$.
\end{proof}
{Observe that the digraph $\tG$ constructed in Lemma \ref{lem:shift homology construction} has again (at least) one vertex of valence 1 and consequently the construction can be iterated.}
\begin{prop}\label{prop:comphom}
For all $i, n \in \bN$, there exists a digraph $\tG$ such that $\mathrm{dim }_\bK \left(\Hmu^i(\tG; \bK)\right)=n$.
\end{prop}
\begin{proof}
The multipath cohomology of a  sink graph $\tG$ with $n+1$ edges is concentrated in degree one, where it is $\Hmu^1(\tG;\bK)\cong \bK^n$. By applying iteratively Lemma~\ref{lem:shift homology construction}, we obtain digraphs~{$\tG$ with $\mathrm{dim }_\bK \left(\Hmu^i(\tG; \bK)\right)=n$} for every $i$. 
\end{proof}

\section{Oriented linear graphs}\label{sec:dir_line_graphs}
This section is devoted to the study of the multipath cohomology of oriented linear graphs. 
{Firstly}, we focus on the case of  coefficients in a field~$\bK$. In this case, we achieve a complete description of their cohomology groups. Then, we analyse the general case of  coefficients in a graded algebra~$A$, and we prove some recursive formulae for the graded Euler characteristic.

\subsection{Multipath cohomology of linear graphs}

An \emph{oriented linear graph} \tL~(on $n$ vertices) is a directed graph with vertices $\{ v_0 , ..., v_{n-1}\}$, such that, for all~$i\in\{1,...,n-1\}$, exactly one among $(v_i,v_{i-1})$ and $(v_{i-1},v_i)$ belongs to~$E(\tL)$, and there are no other edges. 
An oriented linear graph~{$\tL$} is called \emph{alternating} if
whenever $(v_{i-1},v_i)\in E(\tL)$ for some~$i<n-1$, we have~$(v_{i+1},v_i)\in E(\tL)$ and, analogously, if $(v_{i},v_{i-1})\in E(\tL)$ then $(v_{i},v_{i+1})\in E(\tL)$. 
We denote by $\tA_n$ an alternating linear graph on $n+1$ vertices.
Observe that the alternating graph~$\tA_n$ is unique up to orientation reversing. 

\begin{defn}
A  vertex  of an oriented linear graph~$\tL$ is called \emph{unstable} if it is both a source and a target, and \emph{stable} otherwise. We denote by $\mathrm{SV}(\tL)$ the set of stable vertices of~$\tL$.
\end{defn}
 
\begin{figure}
    \centering
    \begin{subfigure}[t]{0.25\textwidth}
    \centering
    \begin{tikzpicture}
    \node (a) at (0,0) {};
    \draw[fill] (a) circle (.05);
    
    \node (b) at (1,0) {};
    \draw[fill] (b) circle (.05);

    \node (c) at (2,0) {};
    \draw[fill] (c) circle (.05);

    \draw[dashed] (2.75,0) circle (.75);
    \node at (2.75,0) {$\tG^{\prime \prime}$};

    \node[above] at (.5,0) {$e_1$};
    \node[above] at (1.5,0) {$e_2$};
    
    \draw[thick, bunired] (a) -- (b);
    \draw[thick, bunired] (c) -- (b);
    \end{tikzpicture}
    \subcaption{The graph $\tG$.}
    \label{subfig:G}
    \end{subfigure}
\hspace{.05\textwidth}
   \begin{subfigure}[t]{0.25\textwidth}
   \centering
    \begin{tikzpicture}
    
    \node (b) at (1,0) {};
    \draw[fill] (b) circle (.05);

    \node (c) at (2,0) {};
    \draw[fill] (c) circle (.05);

    \draw[dashed] (2.75,0) circle (.75);
    \node at (2.75,0) {$\tG^{\prime \prime}$};
    
    \node[above] at (1.5,0) {$e_2$};

    \draw[thick, bunired] (c) -- (b);
    \end{tikzpicture}
    \subcaption{The graph $\tG'$.}
    \label{subfig:Gprime}
    \end{subfigure}
\hspace{.05\textwidth}
        \begin{subfigure}[t]{0.25\textwidth}
    \centering
    \begin{tikzpicture}
    \node (a) at (0,0) {};
    \node[above]  at (0,0) {$v$};
    \draw[fill] (a) circle (.05);
    
    \node (b) at (1,0) {};
    \draw[fill] (b) circle (.05);

    \node (c) at (2,0) {};
    \draw[fill] (c) circle (.05);

    \draw[dashed] (2.75,0) circle (.75);
    \node at (2.75,0) {$\tG^{\prime \prime}$};

    \node[above] at (.5,0) {$e_1$};
    \node[above] at (1.5,0) {$e_2$};
    
    \draw[thick, gray, dashed] (a) -- (b);
    \draw[thick, bunired] (c) -- (b);
    \end{tikzpicture}
    \subcaption{The graph $\tG^\prime$ embedded as a spanning subgraph of~$\tG$ ($e_1$ is missing).}
    \label{subfig:GprimeInG}
    \end{subfigure}
    
    \vspace{1.5em}
       \begin{subfigure}[b]{0.4\textwidth}
   \centering
    \begin{tikzpicture}

    \node (a) at (0,0) {};
    \draw[fill] (a) circle (.05);

    \node (b) at (1,0) {};
    \draw[fill] (b) circle (.05);

    \node (c) at (2,0) {};
    \draw[fill] (c) circle (.05);

    \draw[dashed] (2.75,0) circle (.75);
    \node at (2.75,0) {$\tG^{\prime \prime}$};
    
    \node[above] at (.5,0) {$e_1$};
    \node[above] at (1.5,0) {$e_2$};

    \draw[thick, bunired, -latex] (a) -- (b);
    \draw[thick, bunired, latex-] (c) -- (b);
    \end{tikzpicture}
    \subcaption{Coherent orientation of $e_1$ and $e_2$.}
    \label{subfig:coherente1e2}
    \end{subfigure}
\hspace{.05\textwidth}
        \begin{subfigure}[b]{0.4\textwidth}
    \centering
    \begin{tikzpicture}
    \node (a) at (0,0) {};
    \draw[fill] (a) circle (.05);
    
    \node (b) at (1,0) {};
    \draw[fill] (b) circle (.05);

    \node (c) at (2,0) {};
    \draw[fill] (c) circle (.05);

    \draw[dashed] (2.75,0) circle (.75);
    \node at (2.75,0) {$\tG^{\prime \prime}$};

    \node[above] at (.5,0) {$e_1$};
    \node[above] at (1.5,0) {$e_2$};
    
    \draw[thick, bunired, -latex] (a) -- (b);
    \draw[thick, bunired, -latex] (c) -- (b);
    \end{tikzpicture}
    \subcaption{Non-coherent orientation of $e_1$ and $e_2$.}
    \label{subfig:noncohenrete1e2}
    \end{subfigure}
    \caption{The graphs $\tG$, $\tG^{\prime}$, $\tG^{\prime \prime}$, and the relative orientations of $e_1$ and $e_2$.}
    \label{fig:SECLeafRem}
\end{figure}
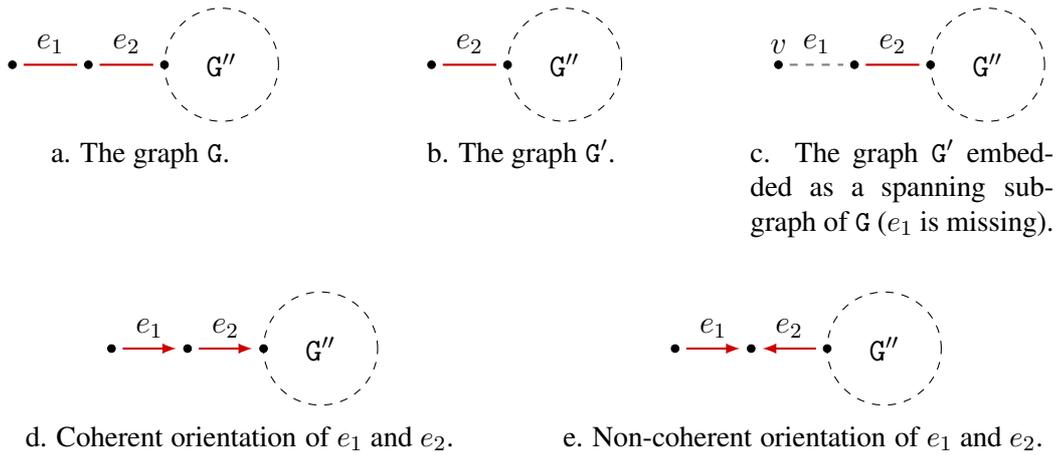

Our aim is to show that the cohomology of an oriented linear graph~$\tL$ is related to the number of stable vertices in $\tL$, and their relative distance.

\begin{defn}
Let $v_i, \, v_j$ be vertices of an oriented linear graph $\tL$. The distance $d(v_i,v_j)$ is the length of the unique simple path, if it exists, between them, and it is set to $-\infty$ otherwise.
\end{defn}
 
For an oriented linear graph~$\tL$ the property $D(k)$ is defined as follows: 
\begin{equation}\label{eq:propDk}
D(k):\quad
\forall \ v,w \in \mathrm{SV}(\tL), \quad d(v,w)\leq k \ .
\end{equation}
A disjoint union of oriented linear graphs satisfies the property $D(k)$ if each component does.
Observe that the set of oriented linear graphs is filtered by the above property; each linear graph satisfies $D(k)$ for some $k$, and if $\tL$ satisfies~$D(k)$, then it also satisfies $D(k+1)$. Furthermore, the alternating graphs satisfy the property $D(1)$, and they are the only connected graphs satisfying it. The graph $\tI_n$ satisfies the property~$D(n)$.
Observe that, if an oriented linear graph~$\tL$ satisfies the property $D(n)$ for $n>2$, then $\tL$ has trivial cohomology. In fact, if there exists a pair of stable vertices at distance grater than~$2$, then there exists a $\nu$-equivalence $f\colon \tI_3 \rightarrow \tL$ away from $v_0, v_3 \in V(\tI_3)$; by Criterion~\ref{crit:istmo}, $\Hmu^*(\tL, \bK)=0$.

\begin{rem}\label{subgraphstablepoints}
If $\tL$ satisfies the property $D(n)$ then each of its subgraphs also satisfies $D(n)$.
\end{rem}

By the above observations, a complete description of the multipath cohomology of oriented linear graphs can be achieved by studying graphs satisfying the property $D(2)$. As a first step, we start with oriented linear graphs satisfying $D(1)$.
Since the cohomology of a disjoint union of linear graphs is the tensor product over its components (cf.~Remark~\ref{DisjointUnion}), we restrict to the case of connected ones, i.e.~the alternating graphs.

\begin{thm}\label{LinAlt}
{Let $\tA_n$ be an alternating graph. Then, we have the following isomorphisms
 \begin{equation}
\Hmu^*(\tA_n, \bK) \cong
\left\{
	\begin{array}{lll}
    \Hmu^*(\tA_{n-1}, \bK) & \mbox{ if } \;  n \equiv 0 \mod 3,\\
    	0 & \mbox{ if } \; n \equiv 1 \mod 3, \\
      \Hmu^{*-1}(\tA_{n-2},\bK) & \mbox{ if } \;  n \equiv 2 \mod 3.\\
	\end{array}
\right.\label{eq:recursion_H_alternating}
\end{equation}
depending on the congruence class of $n$ modulo $3$.}
\end{thm}
\proof
We use the notation illustrated in Figure~\ref{fig:SECLeafRem}, with $\tG = \tA_n$.   By Theorem~\ref{finedimondo}, the path poset~$P(\tA_n)$ is isomorphic to $\nabla_{P(\tG'')} (P(\tG \setminus \{{e_2}\}), P(\tG')) $. Observe that $P(\tG \setminus \{{e_2}\})$ is a cone over $P(\tG'')$ and that we have isomorphisms of graphs $\tG' \cong \tA_{n-1}$ and $\tG'' \cong \tA_{n-2} $. Consequently, there is an induced isomorphism
\[
P(\tA_n)=P(\tG)\cong \nabla_{P(\tA_{n-2})} (\Cone P(\tA_{n-2}), P(\tA_{n-1}))
\]
of posets. 
Since the cohomology groups $\Hmu^*(\Cone P(\tA_{n-2});\bK)$ are all trivial, from Theorem~\ref{thm:MVposets} (applied to the gluing of $\Cone P(\tt{A}_{n-2})$ and  $P(\tA_{n-1})$) we obtain the following exact sequence:
 \[
 \cdots \to \Hmu^i( P(\tA_{n-1});\bK) \rightarrow \Hmu^i( P(\tA_{n-2});\bK) \rightarrow \Hmu^{i+1}(P(\tA_{n}),\bK) \rightarrow \Hmu^{i+1}( P(\tA_{n-1});\bK)   \to 
 \cdots
 \]
A direct computation shows that the cohomology of the graphs $\tA_n$ for $n<5$ agrees with the isomorphisms in Equation~\eqref{eq:recursion_H_alternating}. The assertion now follows  by an induction argument.
\endproof

As a consequence, it is possible to obtain  a precise  description of the ranks of cohomology groups for alternating graphs:

\begin{cor}\label{linalt} 
Let $\tA_n$ be an alternating graph. Then:
 \begin{equation*}
 \dim_{\bK} \mathrm{H}_{\mu}^k(\tA_n; \bK) =
\begin{cases}
    1 & \mbox{if }  n=3(k-1)+2 \mbox{ or } n=3k,\\
    	0 & \mbox{otherwise.} \\
	\end{cases}
\end{equation*}
\end{cor}

 Before proceeding with our analysis of the cohomology of oriented linear graphs, we need the following definition. 

\begin{defn}\label{def:reduced}
Given a oriented linear graph $\tL$,  
its \emph{reduction}~$\mathrm{Red }~\tL$  is the (possibly disconnected) spanning subgraph of $\tL$, obtained {as follows; for each maximal simple path on (the ordered set of) vertices $\{v_h,\dots,v_{h+m}\}$, delete all edges, but the one between $v_{h+m-1}$ and $v_{h+m}$.}
\end{defn}

Note that if a maximal simple path is an edge of $\tL$, then it is still an edge of $\mathrm{Red }~\tL$.

\begin{example}
The reduction of $\tI_n$ is the spanning subgraph of $\tI_n$ with only edge~$(v_{n-1},v_n)$. The reduction $\mathrm{Red }~\tL$ is isomorphic to $\tL$ if, and only if, $\tL$ is alternating.
\end{example}

Observe that, by construction, the digraph $\mathrm{Red }~\tL$ is the disjoint union of $h$ connected components, 
where~$h-1$ is the number of edges deleted during the process of reduction.

We can linearly order the connected components of $\mathrm{Red}~\tL$ according to their minimal-index vertex. Denote by $C_i$ the $i$-th component with respect to this order. 
Notice that $\mathrm{Red}~\tL$ satisfies~$D(1)$, thence for each $i\in \{ 1 ,..., h\}$ there is a $k_{i}\geq 0$ such that $C_i\cong\tA_{k_i}$.

\begin{lem}\label{extr4compo}
Let $\tL$ be an oriented linear graph satisfying the property $D(2)$, 
and let $\tA_{k_1}, \dots, \tA_{k_h}$ be the connected components of $\mathrm{Red}~\tL$. If $k_h \equiv 1 \mod 3$, then $\Hmu^*(\tL; \bK)=0$. 
\end{lem}
\proof 
Note that if $\tL$ satisfies $D(1)$, then it is an alternating graph, and the statement follows by Proposition~\ref{LinAlt}. Suppose $\tL$ satisfies $D(2)$, but not $D(1)$ and denote by $w_0, \dots, w_{k_h}$ the vertices of $\tA_{k_h}$.
First, observe that, by Definition~\ref{def:reduced}, if $k_h=1$ or if~$|E(\tL)|=k_h+1$, then the linear graph~$\tL$ has a coherent tail. Thus, $\Hmu^*(\tL; \bK)=0$ by Remark~\ref{rem:tail}.

In all the other cases, up to orientation reversing, the graph~$\tL$ contains a subgraph as in Figure~\ref{pippo}. 
\begin{figure}[h]
	\begin{tikzpicture}[baseline=(current bounding box.center)]
		\tikzstyle{point}=[circle,thick,draw=black,fill=black,inner sep=0pt,minimum width=2pt,minimum height=2pt]
		\tikzstyle{arc}=[shorten >= 8pt,shorten <= 8pt,->, thick]
		
		\node   at (-.5,0) {$\dots$};
		\node[above] (v0) at (0,0) {$v$};
		\draw[fill] (0,0)  circle (.05);
		\node[above] (v1) at (1.5,0) {$x$};
		\draw[fill] (1.5,0)  circle (.05);
		\node[above] (v2) at (3,0) {$w_{0}$};
		\draw[fill] (3,0)  circle (.05);
		\node[above] (v4) at (4.5,0) {$w_{1}$};
		\draw[fill] (4.5,0)  circle (.05);
		\node[above] (v5) at (6,0) {$w_2$};
		\draw[fill] (6,0)  circle (.05);
		\node[above] (v6) at (7.5,0) {$w_3$};
		\draw[fill] (7.5,0)  circle (.05);
		 \node[above]  at (9,0) {$w_{k_h -1}$};
		 \node (v7) at (9,0) {};
		\draw[fill] (9,0)  circle (.05);
		
		\node   at (8.25,0) {$\dots$};
		 \node (v8) at (10.5,0) {};
		 \node[above]  at (10.5,0) {$w_{k_h}$};
		\draw[fill] (10.5,0)  circle (.05);
		
		\draw[thick, bunired, -latex] (0.15,0) -- (1.35,0);
		\draw[thick, bunired, -latex] (2.75,0) -- (1.65,0);
		\draw[thick, bunired, -latex] (4.35,0) -- (3.15,0);
		\draw[thick, bunired, -latex] (4.65,0) -- (5.85,0);
		\draw[thick, bunired, -latex] (7.35,0) -- (6.15,0);
	    \draw[thick, bunired, ] (v7) -- (v8);
	\end{tikzpicture}
	\caption{The sub-graph $\tA_{k_h}$ inside $\tL$. The edge $(w_{k_h -1}, w_{k_h})$ can be oriented either way depending on the parity of $k_h$. }
	\label{pippo}
\end{figure}
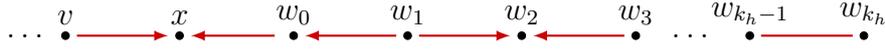
We can apply Theorem~\ref{finedimondo} choosing the vertex~$x$ illustrated in Figure~\ref{pippo}, 
and obtain the isomorphism
\[P(\tL)\cong \nabla_{P(\tL''')}(P(\tL'), P(\tL'')) \ , \]
{where $\tL'=\tL \setminus \{(w_0,x)\}$, $\tL''=\tL \setminus \{(v,x)\}$ and $\tL'''=\tL \setminus \{(w_0,x),(v,x) \}$.}
Now, Remark~\ref{rem:tail} implies $\Hmu^*(\tL''; \bK)=0$.
Furthermore, $\Hmu^*(\tL'; \bK)=\Hmu^*(\tL'''; \bK)=~0$, since $\tL'$ and $\tL'''$ have~$\tA_{k_h}$ as a connected component -- cf.~Remark~\ref{DisjointUnion} and~Proposition~\ref{LinAlt}; in fact, $w_{k_h}$ is univalent, since~$k_h\neq 0$. The statement now follows from Theorem~\ref{thm:MVposets}.
\endproof

Denote by $\lfloor x \rfloor $ the integer part of $x$.

\begin{thm}\label{thm:linearcohomology}
Let $\tL$ be an oriented linear graph satisfying the property $D(2)$, but not $D(1)$. 
Denote by $\tA_{k_1}, \dots, \tA_{k_h}$ the connected components of $\mathrm{Red }~\tL$. Then:
\begin{enumerate}[label={\rm (\arabic*)}]
\item if there exists an index $j \in \{1, \dots h-1\}$ such that $k_j \equiv 0 \mod  3$, then $\Hmu^*(\tL,\bK)=0$;
\item otherwise, the cohomology groups of $\tL$ decompose as
\[\quad \Hmu^{* + h - 1}(\tL;\bK)= \Hmu^*(\mathtt{A}_{3\lfloor k_1/3 \rfloor}; \bK)\otimes \dots \otimes \Hmu^*(\mathtt{A}_{3\lfloor k_{h-1}/3 \rfloor}; \bK)\otimes \Hmu^*(\mathtt{A}_{k_h} ; \bK) \ , \]
where $\otimes$ here denotes the graded tensor product over $\bK$. 
\end{enumerate}
\end{thm}
\proof
We linearly order the edges $e_1, \dots,e_{h-1}$ in $E(\tL)\setminus E(\mathrm{Red}~\tL)$, i.e.~ the edges in the complement of the components $\tA_{k_1}, \dots, \tA_{k_h}$, according to their minimal-index vertex. The addition of $e_i$ to ${\rm Red}~L$ merges the component $\tA_{k_i}$ with the component $\tA_{k_{i+1}}$.
\begin{enumerate}[label={\rm (\arabic*)}]
\item 
To prove the first item, observe that if $k_j \equiv 0  \mod 3$, then either $\tL$ has a coherent tail (when $k_1=0$)  or, up to orientation reversing, it contains a subgraph as in Figure~\ref{fig:CaseA}, where the vertices $w_0, \dots, w_3$ are in $\tA_{k_j}$.
\begin{figure}[h]
	\begin{tikzpicture}[baseline=(current bounding box.center)]
		\tikzstyle{point}=[circle,thick,draw=black,fill=black,inner sep=0pt,minimum width=2pt,minimum height=2pt]
		\tikzstyle{arc}=[shorten >= 8pt,shorten <= 8pt,->, thick]
		
		\node[above] (v0) at (0,0) {$w_0$};
		\draw[fill] (0,0)  circle (.05);
		\node[above] (v1) at (1.5,0) {$w_1$};
		\draw[fill] (1.5,0)  circle (.05);
		\node[above] (v2) at (3,0) {$w_2$};
		\draw[fill] (3,0)  circle (.05);
		\node[above] (v4) at (4.5,0) {$w_{3}$};
		\draw[fill] (4.5,0)  circle (.05);
		\node[above] (v5) at (6,0) {$x$};
		\draw[fill] (6,0)  circle (.05);
		\node[above] (v6) at (7.5,0) {$v$};
		\draw[fill] (7.5,0)  circle (.05);
		
		\draw[thick, bunired, -latex] (0.15,0) -- (1.35,0);
		\draw[thick, bunired, -latex] (2.75,0) -- (1.65,0);
		\draw[thick, bunired, -latex] (3.15,0) -- (4.35,0);
		\draw[thick, bunired, -latex] (5.85,0) -- (4.65,0);
			\draw[thick, bunired, -latex] (7.35,0) -- (6.15,0);
	\end{tikzpicture}
	\caption{A possible sub-graph of $\tL$.}
	\label{fig:CaseA}
\end{figure}
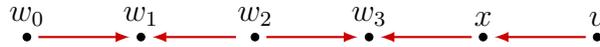
 Now, if $v$ is an univalent vertex of the graph~$\tL$, then $\tL$ has a coherent tail of length two and, again, the cohomology groups~$\Hmu^*(\tL,\bK)$ are trivial. In all remaining cases, $\tL$ contains a subgraph as in Figure~\ref{fig:caseB}.  
  \begin{figure}[h!]
	\begin{tikzpicture}[baseline=(current bounding box.center)]
		\tikzstyle{point}=[circle,thick,draw=black,fill=black,inner sep=0pt,minimum width=2pt,minimum height=2pt]
		\tikzstyle{arc}=[shorten >= 8pt,shorten <= 8pt,->, thick]
		
		\node[above] (v0) at (0,0) {$w_0$};
		\draw[fill] (0,0)  circle (.05);
		\node[above] (v1) at (1.5,0) {$w_1$};
		\draw[fill] (1.5,0)  circle (.05);
		\node[above] (v2) at (3,0) {$w_2$};
		\draw[fill] (3,0)  circle (.05);
		\node[above] (v4) at (4.5,0) {$w_{3}$};
		\draw[fill] (4.5,0)  circle (.05);
		\node[above] (v5) at (6,0) {$x$};
		\draw[fill] (6,0)  circle (.05);
		\node[above] (v6) at (7.5,0) {$v$};
		\draw[fill] (7.5,0)  circle (.05);
        \node[above] (v7) at (9,0) {$y$};
		\draw[fill] (9,0)  circle (.05);
		
		\draw[thick, bunired, -latex] (0.15,0) -- (1.35,0);
		\draw[thick, bunired, -latex] (2.75,0) -- (1.65,0);
		\draw[thick, bunired, -latex] (3.15,0) -- (4.35,0);
		\draw[thick, bunired, -latex] (5.85,0) -- (4.65,0);
        \draw[thick, bunired, -latex] (7.35,0) -- (6.15,0);
        \draw[thick, bunired, -latex] (7.65,0) -- (8.85,0);
			
	\end{tikzpicture}
	\caption{Another possible sub-graph of $\tL$.}
	\label{fig:caseB}
\end{figure}
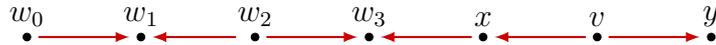
We are in the hypothesis of Theorem~\ref{finedimondo}, choosing the vertex~$v$, to decompose~$P(\tL)$ as 
\[P(\tL)\cong \nabla_{P(\tL''')}(P(\tL'), P(\tL'')) \ , \]
where $\tL'=\tL \setminus \{(v,x)\}$, $\tL''=\tL \setminus \{(v,y)\}$, and $\tL'''=\tL \setminus \{(v,x), (v,y)\}$.

The first assertion follows from Criterion~\ref{crit:FDM}: $\Hmu^*(\tL''; \bK)=0$ because $\tL''$ contains a coherent tail, and the cohomologies $\Hmu^*(\tL';\bK)$ and $\Hmu^*(\tL'''; \bK)$ are both trivial in virtue of Lemma~\ref{extr4compo} and Remark~\ref{DisjointUnion}.

\item The proof of second statement proceeds by induction. We are in the case $k_j\not\equiv 0 \mod 3$ for $j\in \{ 1 ,\dots, h-1\}$.

The graph obtained from~$\tL$ by deleting $e_1$ has two connected components $\tL_1=\tA_{k_1}$ and $\tL_2$. Observe that $\tL_1$ and $\tL_2$ are oriented linear graphs satisfying the property $D(2)$ -- cf.~Remark~\ref{subgraphstablepoints}. 
Since $k_1\neq 0$, we have that $\tL_1$ has at least one edge.
Furthermore, by definition of the reduction, we can assume that $e_1$ is contained in a linear subgraph of the form \raisebox{-.2em}{\begin{tikzpicture}
\node (a) at (0,0) {}  ;
\node[above] at (a) {};
\draw[fill] (a)  circle (.05);

\node (b) at (1,0) {} ;
\node[above] at (b) {};

\draw[fill] (b)  circle (.05);

\node (c) at (2,0) {};
\node[above] at (c) {};
\draw[fill] (c)  circle (.05);
\node[above] at (.5,0) {$e$};
\node[above] at (1.5,0) {$e_1$};

\draw[thick, -latex, bunired] (a) -- (b);

\draw[thick, -latex, bunired] (c) -- (b);

\end{tikzpicture}} or \raisebox{-.2em}{\begin{tikzpicture}
\node (a) at (0,0) {}  ;
\node[above] at (a) {};
\draw[fill] (a)  circle (.05);

\node (b) at (1,0) {} ;
\node[above] at (b) {};

\draw[fill] (b)  circle (.05);

\node (c) at (2,0) {};
\node[above] at (c) {};
\draw[fill] (c)  circle (.05);
\node[above] at (.5,0) {$e$};
\node[above] at (1.5,0) {$e_1$};

\draw[thick, -latex, bunired] (b) -- (a);

\draw[thick, -latex, bunired] (b) -- (c);

\end{tikzpicture}} with $e\in E(\tL_1)$. 
By Theorem~\ref{finedimondo}, we have the decomposition
\[P(\tL)\cong\nabla_{P((\tL_1 \sqcup \tL_2) \setminus \{e\} )}(P(\tL_1 \sqcup \tL_2), P(\tL \setminus \{e\})) \ . \]
Observe now that $ P(\tL \setminus \{e\}) \cong \Cone P\left((\tL_1 \sqcup \tL_2) \setminus \{e\}\right)$, hence 
\[P(\tL)\cong \nabla_{P(\tA_{k_1-1} \sqcup \tL_2)}\left(
P(\tA_{k_1} \sqcup \tL_2), \Cone  P(\tA_{k_1-1} \sqcup \tL_2) \right) \ . \] 
Using  Theorem~\ref{thm:MVposets}, we get the following exact sequence 
\begin{equation}\label{MValt}\Scale[0.90]{\qquad
    \cdots \to \Hmu^i({\tA}_{k_1-1} \sqcup \tL_2; \bK) \rightarrow \Hmu^{i+1}( \tL; \bK) \rightarrow \Hmu^{i+1}({\tA}_{k_1} \sqcup \tL_2; \bK) \rightarrow \Hmu^{i+1}( {\tA}_{k_1-1} \sqcup \tL_2; \bK)  \to \cdots }
\end{equation}  

In the following, set first $k_1=3j+1$. 
By Corollary~\ref{linalt} and Remark~\ref{DisjointUnion}, the cohomology group $\Hmu^i(\tA_{k_1} \sqcup \tL_2; \bK)$ is trivial. 
Analogously, the  group $\Hmu^i(\tA_{k_1-1} \sqcup \tL_2; \bK)$ is isomorphic to the product $\Hmu^{j}(\tA_{k_1-1}; \bK)\otimes \Hmu^{i-j}(\tL_2; \bK)$. 
{In the case $k_1=3j+2$ instead,} the group $\Hmu^i(\tA_{k_1} \sqcup \tL_2; \bK)$ is isomorphic to the product $\Hmu^{j+1}(\tA_{k_1}; \bK)\otimes \Hmu^{i-j-1}(\tL_2; \bK)$, and $\Hmu^i(\tA_{k_1-1} \sqcup \tL_2; \bK)$ is trivial.

As a consequence, using the Mayer Vietoris sequence in Equation~\eqref{MValt}, we have 
 \begin{equation*}
\Hmu^{i+1}\left(\tL; \bK\right) \cong
	\begin{cases}
     \Hmu^{j}({\tA}_{k_1-1};\bK)\otimes \Hmu^{i-j}(\tL_2; \bK) & \mbox{ if }  k_1=3j+1,\\
      \Hmu^{j+1}({\tA}_{k_1};\bK)\otimes \Hmu^{i-j}(\tL_2; \bK) & \mbox{ if } k_1=3j+2\\
	\end{cases}
\end{equation*}
If $k_1=3j+1$, we can rewrite the first isomorphism, using that the multipath cohomology of ${\tA}_{k_1-1} = \tA_{3j}$ is concentrated in cohomological degree $j$, 
as follows:
\[
\quad \Hmu^{i+1}\left(\tL; \bK\right) \cong \left(\Hmu^{j}({\tA}_{k_1-1}, \bK)\otimes \Hmu^{i-j}(\tL_2, \bK)\right) \cong \bigoplus_{r+s = i} \left(\Hmu^{r}({\tA}_{3j}, \bK)\otimes \Hmu^{s}(\tL_2, \bK)\right)
\]
On the other hand, if $k_1=3j+2$, by Theorem~\ref{LinAlt} and Corollary~\ref{linalt}, we obtain the isomorphism
\begin{equation*}
\Hmu^{i}\left(\tL; \bK\right) \cong  \Hmu^{j+1}({\tA}_{k_1}, \bK)\otimes \Hmu^{i-j -1}(\tL_2, \bK) \cong \Hmu^{j}({\tA}_{3j}, \bK)\otimes \Hmu^{i-j-1}(\tL_2, \bK).
\end{equation*}
Finally, we have also that
\[ \Hmu^{j}({\tA}_{3j}, \bK)\otimes \Hmu^{i-j-1}(\tL_2, \bK) \cong \bigoplus_{r+s = i -1} \left(\Hmu^{r}({\tA}_{3j}, \bK)\otimes \Hmu^{s}(\tL_2, \bK)\right)
\]
since the multipath cohomology of ${\tA}_{k_1-2} = \tA_{3j}$ is concentrated in degree $j$.
We have shown that, in either case, we have
\[\quad \Hmu^{*+1}\left(\tL; \bK\right) \cong \bigoplus_{r+s = *} \left(\Hmu^{r}({\tA}_{3j}, \bK)\otimes \Hmu^{s}(\tL_2, \bK)\right) = \left(\Hmu^{*}({\tA}_{3j}, \bK)\otimes \Hmu^{*}(\tL_2, \bK)\right).\]
The statement now follows by induction.
\end{enumerate}
\endproof

\subsection{Graded characteristic of linear graphs}

We now analyse the cohomology of linear graphs from a different perspective;
instead of considering multipath cohomology with coefficient in a field, we fix a principal ideal domain~$R$ as a base ring,  and take coefficients in a unital $R$-algebra~$A$.
We are interested in analysing the (graded) Euler characteristic in the case where~$A$ is graded.
Firstly, let us prove a general result.

\begin{thm}\label{thm:secremoveleaves}
Let $\tG$, $\tG^{\prime}$, and $\tG^{\prime \prime}$ be three digraphs as illustrated in Figure~\ref{fig:SECLeafRem}, and  let $A$ be a unital $R$-algebra.
Then, one of the following holds:
\begin{enumerate}[label ={\rm (\arabic*)}]
\item if the edges $e_1$ and $e_2$ are coherently oriented, as in Figure \ref{subfig:coherente1e2}, then the sequence
\[ 0 \to C_{\mu}^{*-1}(\tG^\prime;A) \longrightarrow C_{\mu}^*(\tG ; A ) \longrightarrow C_{\mu}^*(\tG^\prime;A) \otimes A \to 0 \]
is exact;
\item if the edges $e_1$ and $e_2$ are not coherently oriented, see
Figure~\ref{subfig:noncohenrete1e2}, then the sequence
\[ 0 \to C_{\mu}^{*-1}(\tG^{\prime\prime};A)\otimes A \longrightarrow C_{\mu}^*(\tG ; A ) \longrightarrow C_{\mu}^*(\tG^\prime;A)\otimes A \to 0 \]
is exact;
\end{enumerate}
where all tensor products are over $R$.
\end{thm}

\begin{proof}
Before dwelling into the details of each case, we first discuss the general picture.
Denote for simplicity by $P$, $P'$, and $P''$ the path posets of $\tG$, $\tG^{\prime}$, and $\tG^{\prime \prime}$, respectively.
We have an embedding of $\tG^\prime$ into $\tG$ that induces (as a spanning subgraph, cf.~Figure~\ref{subfig:GprimeInG}) an injective morphism of poset $\imath \colon P' \to P$. We remark that $\imath(P')$ is a downward closed faithful subposet of~$P$, and that the minimal length of an element in $P\setminus \imath(P')$ is $1$. By \cite[Proposition~5.12]{primo},  we have the following sequence of cochain complexes
\[ 0 \to C^{*-1}_{\mathcal{F}_{A,A}}(P\setminus \imath(P')) \longrightarrow  C^*_{\mathcal{F}_{A,A}}(P)  \longrightarrow C^*_{\mathcal{F}_{A,A}}(P')  \otimes A \to 0 \]
which is exact.
By definition, we have that $C^*_{\mathcal{F}_{A,A}}(P) = C_{\mu}^*(\tG ; A )$ and  $C^*_{\mathcal{F}_{A,A}}(P') = C_{\mu}^*(\tG' ; A )$.

To conclude it is enough to identify the complex $C^*_{\mathcal{F}_{A,A}}(P\setminus {\imath(P')})$. 
Observe that the elements of $P\setminus \imath(P')$ are precisely the multipaths in $\tG$  containing $e_1$. The proof splits now in two cases;
\begin{enumerate}[label ={\rm (\arabic*)}]
    \item\label{case:seci} if $e_1$ and $e_2$ are coherently oriented, then $\tH \cup \{ e_1 \}\in P\setminus \imath(P')$ for each $\tH\in \imath(P')$. Thus, we have an order-preserving bijection between $P(\tG^\prime)$ and $P\setminus \imath(P')$, which also preserves the number of connected components. Therefore, we obtain the isomorphism
    \[C^*_{\mathcal{F}_{A,A}}(P\setminus \imath(P')) \cong C^*_{\mathcal{F}_{A,A}}(P') = C_{\mu}^*(\tG^\prime;A)\]
   of cochain complexes;
    \item if $e_1$ and $e_2$ are  not coherently oriented, then a multipath which contains $e_1$ cannot contain $e_2$. Thus, we have an identification between multpaths in $P\setminus \imath(P)$ and multipath in $\tG^{\prime\prime} \cup \{ e_1\}$, which gives  the isomorphisms
    \[C^*_{\mathcal{F}_{A,A}}(P\setminus \imath(P')) \cong C^*_{\mathcal{F}_{A,A}}(P'') \otimes A = C_{\mu}^*(\tG^{\prime\prime};A)\otimes A \ . \]
    The tensor factor $\otimes A$ arises from an extra connected component in each element of the poset $P\setminus \imath(P')$ with respect to the corresponding element in $P''$ (namely the edge~$e_1$).
\end{enumerate}
The statement is now immediate from the above identifications.
\end{proof}

The short exact sequence in Theorem \ref{thm:secremoveleaves} \ref{case:seci} holds also in a slightly different case;

\begin{rem}\label{rem:seci}
Assume that the edges $e_1$, $e_2$, and $e_3$  in $\tG$ form a path, and that $e_2$ is not contained in any coherently oriented cycle in~$\tG$, then the sequence in Theorem \ref{thm:secremoveleaves} \ref{case:seci} holds.
In this case, the role of $\tG'$ is played by the graph obtained from $\tG$ by contracting the ``middle edge'' $e_2$.
Then, we have that
\[P(\tG)\setminus P(\tG \setminus \{ e_2\}) \cong  P(\tG') \quad\text{and} \quad P(\tG)\setminus P(\tG \setminus \{ e_2\}) \cong  P(\tG \setminus \{ e_2\}) ,\]
where the first identification is given by contracting $e_2$, while the second is given by deleting it.
Note that, in the former case, the number of connected components of multipaths is preserved, while in the second case a multipath $\tH\in P(\tG)\setminus P(\tG \setminus \{ e_2\})$ is sent to a multipath with one more connected component.
At this point, the same reasoning as in the proof of Theorem~\ref{thm:secremoveleaves} provides the desired exact sequence.
\end{rem}

We can use Theorem~\ref{thm:secremoveleaves} to reprove a result of Przytycki~\cite{Prz} which computes the cohomology of $\tI_n$ (cf.~\cite[Corollary~7.5]{primo}). We observe that Przytycki obtains this result using the chromatic polynomial as intermediate step, while we prove it directly by induction. Moreover, the following corollary, if $A$ is commutative, can also be  proved as an application of the deletion-contraction exact sequence for the chromatic homology \cite[Theorem~3.2]{HGRong}.

\begin{cor}\label{cor:homology linear graphs}
Let~$\tI_n$ be the coherently oriented linear graph (cf.~Figure~\ref{fig:nstep}). For each (unital) $R$-algebra~$A$, we have 
\[ \Hmu^*(\tI_n;A) = \Hmu^0(\tI_n;A),\]
and \[{\rm rank}_R(\Hmu^*(\tI_n;A)) = {\rm rank}_R(\Hmu^0(\tI_n;A)) =  \begin{cases}{\rm rank}_R (A) ( {\rm rank}_R (A) -1 )^n & n\geq 1 \\ {\rm rank}_R (A)  & n =0 \end{cases}\ . \]
\end{cor}
\begin{proof}
The statement is true for $\tI_0$ and it is easily proved for $\tI_1$; in fact, we have
\[ 0\to C^0_{\mu}(\tI_1;A) = A\otimes A \overset{d^0}{\longrightarrow} C^1_{\mu}(\tI_1;A) = A \to 0,\]
where $d^0$ is the map $a\otimes b\mapsto ab$, 
which is surjective since $A$ is unital.

We proceed by induction. Assume  that the statement is true for $n=k-1$. Then, we can apply Theorem~\ref{thm:secremoveleaves} to $\tG = \tI_k$, $\tG^\prime = \tI_{k-1}$, and $e = (v_0,v_1)$ -- cf.~Figure~\ref{fig:nstep}.
From the inductive hypothesis, we obtain the exact sequences
\[ 0 = {\rm H}_{\mu}^{i-1}(\tI_{k-1};A) \to{\rm H}_{\mu}^i(\tI_{k} ; A ) \to {\rm H}_{\mu}^{i}(\tI_{k-1};A) = 0,\quad \text{for }i>1,\]
and
\begin{equation}
\label{eq:lespartial}
0 \to {\rm H}_{\mu}^0(\tI_{k}; A ) \to{\rm H}_{\mu}^0(\tI_{k-1};A) \otimes A \overset{\delta^*}{\to} {\rm H}_{\mu}^0(\tI_{k-1};A)  \to {\rm H}_{\mu}^1(\tI_{k} ; A ) \to 0,
\end{equation}
where $\delta^*(x\otimes 1) = x$.
As a consequence we have
\begin{enumerate}
\item ${\rm H}_{\mu}^i(\tI_{k} ; A ) = 0$, for all $i>1$;
\item the exact sequence in Equation~\eqref{eq:lespartial}, since $\delta^*$ is surjective, splits into the exact sequences
\[ 0 \to {\rm H}_{\mu}^0(\tI_{k}; A ) \to{\rm H}_{\mu}^0(\tI_{k-1};A) \otimes A \overset{\delta^*}{\to} {\rm H}_{\mu}^0(\tI_{k-1};A)  \to 0 \]
and
\[ 0 \to {\rm H}_{\mu}^1(\tI_{k} ; A ) \to 0.  \]
\end{enumerate}
Since the rank is additive on short exact sequences\footnote{To see this one can tensor for the quotient field, or localise -- cf.~\cite[Definition~1.4.2 and Proposition~1.4.5]{HB}.}, this concludes the proof.
\end{proof}

Let $R$ be a PID, and $A^* = \bigoplus_{i\in \bZ} A^i$ be a finitely generated $\bZ$-graded $R$-algebra. 
The \emph{graded dimension} of $A$ is the Laurent polynomial 
\[\qdim(A^*) \coloneqq \sum_{i\in\bZ} \mathrm{rank}_{R} (A^i)q^i \in \bZ\left[ q,q^{-1} \right] \ , \]
where $\mathrm{rank}_{R}(M)$ indicates the maximal number of non-torsion, linearly independent, elements. 
A graded algebra has, by definition, an homogeneous multiplication. As a consequence, the multipath cochain complex inherits from $A^*$ a second $\bZ$-grading, which is  preserved by the differential. This gives the multipath cohomology the structure of bi-graded cohomology theory. 
Define the \emph{graded Euler characteristic} of a graph $\tG$ (with respect to $A^*$) as
\[\chi_{\rm gr} (\tG;A^*) = \sum_{i,j\in\bZ} (-1)^i \mathrm{rank}_{R} ({\rm H}_{\mu}^{i,j}(\tG;A^*))q^j \in \bZ\left[ q,q^{-1} \right] \ . \]
Note that if we evaluate $\qdim(A^*)$ (resp.~$\chi_{\rm gr}$) in $q =1$, we obtain the rank of $A^*$ as $R$-module (resp. the usual Euler characteristic $\chi$).
It is well-known that the Euler characteristic is additive under exact sequences, and the same holds for the graded Euler characteristic \footnote{For each fixed value $j$ of the second grading we have a short exact sequence of chain complexes. Notice that $\Cmu^{i,j}(\tG;A)\neq 0$ for finitely many values of $i$ and $j$. Thence, we can re-arrange the sum and write 
\[\chi_{\rm gr} (\tG;A^*) = \sum_{j\in\bZ} \chi (\Cmu^{*,j}(\tG;A))q^j.\]
Now, additivity follows from the additivity of the (usual) Euler characteristic.}.
With the above notation in place, the following corollary is an immediate consequence of Theorem~\ref{thm:secremoveleaves}.

\begin{cor}\label{cor:ChiLinear}
Let $\tG$, $\tG^{\prime}$, and $\tG^{\prime \prime}$ be three digraphs as illustrated in Figure~\ref{fig:SECLeafRem}. Let $A^*$ be a finitely generated free $\bZ$-graded $R$-algebra with graded dimension $\alpha  = \qdim(A^*)$.
\begin{enumerate}[label ={\rm (\arabic*)}]
\item\label{case:chigr1} If the edges $e_1$ and $e_2$ are coherently oriented (cf.~Figure~\ref{subfig:coherente1e2}), then
\[ \chi_{\rm gr} (\tG;A^*) = (\alpha -1)  \chi_{\rm gr} (\tG^\prime ;A^*). \]
\item If the edges$e_1$ and $e_2$ are { not} coherently oriented (cf.~Figure \ref{subfig:noncohenrete1e2}), then
\[ \chi_{\rm gr} (\tG;A^*) = \alpha \left(  \chi_{\rm gr} (\tG^\prime;A^* ) - \chi_{\rm gr} (\tG^{\prime\prime};A^*) \right).\]
\end{enumerate}
\end{cor}

Corollary~\ref{cor:ChiLinear} allows us to compute the (graded) characteristic of any oriented linear graph, recursively. We provide, as an example of this process, the graded characteristic of some alternating graphs.

\begin{example}
Let $\tA_n$ be the alternating graph on $n$ vertices, with $n\geq 3$; by Corollary~\ref{cor:ChiLinear} the graded characteristic can be expressed as:
\[ \chi_{\rm gr} (\tA_n;A^*) = \alpha \left(  \chi_{\rm gr} (\tA_{n-1};A^* ) - \chi_{\rm gr} (\tA_{n-2};A^*) \right).\]
For $n=0$ and $n=1$, the graded characteristic of $\tA_n$ are $\qdim(A^*) =: \alpha$ and $\alpha(\alpha -1)$, respectively. Some further examples are listed in Table~\ref{tab:AltChar}.
\begin{table}[h]
    \centering
    \begin{tabular}{cl}
       $n$ & $\chi_{\rm gr}(\tA_n;A^*)$ \\
       \hline\hline
       0 & $\alpha$\\
       1 & $\alpha(\alpha -1)$ \\
       2 & $\alpha^{2}(\alpha - 2)$\\
3 & $\alpha^{2}(\alpha^2 - 3\alpha + 1)$\\
4 & $\alpha^{3}(\alpha - 1)(\alpha - 3)$\\
5 & $\alpha^{3}(\alpha^3 - 5\alpha^2 + 6\alpha - 1)$\\
6 & $\alpha^{4}(\alpha - 2)(\alpha^2 - 4\alpha + 2)$\\
7 & $\alpha^{4}(\alpha - 1)(\alpha^3 - 6\alpha^2 + 9\alpha - 1)$\\
8 & $\alpha^{5}(\alpha^2 - 5\alpha + 5)(\alpha^2 - 3\alpha + 1)$\\
9 & $\alpha^{5}(\alpha^5 - 9\alpha^4 + 28\alpha^3 - 35\alpha^2 + 15\alpha - 1)$\\
10 & $\alpha^{6}(\alpha - 1)(\alpha - 2)(\alpha - 3)(\alpha^2 - 4\alpha + 1)$\\
11 & $\alpha^{6}(\alpha^6 - 11\alpha^5 + 45\alpha^4 - 84\alpha^3 + 70\alpha^2 - 21\alpha + 1)$
    \end{tabular}
    \caption{The graded Euler characteristic of some alternating linear graphs.}
    \label{tab:AltChar}
\end{table}

The usual Euler characteristic over $R = \bK$ can be obtained by evaluating the graded Euler characteristic in~$\alpha =1$. 
Thus, by Corollary~{\ref{linalt}}, $(\alpha -1)$ divides $\chi_{\rm gr} (\tA_n;A^*)$ if, and only if,~$n\equiv 1$ modulo $3$. Observe that our computations are in perfect accordance with this fact.
\end{example}

Despite not having a closed formula for the graded Euler characteristics of alternating graphs, we can compute its associated  generating function 
$A(t)=\sum \chi_{\rm gr}(\tA_n) t^n {\in (\bZ[\alpha])[[t]]}$ (cf. \cite{genfun}). In fact, if we denote by $S_h$ the classical shift for power series
\[ S_{h}\left(\sum_{n=0}^{\infty} c_n t^{n}\right) = \sum_{n=0}^{\infty} c_{n+h} t^{n}\ ,\]
from Theorem \ref{thm:secremoveleaves} one obtains the relation
$S_2(A(t))=\alpha(S_1(A(t))-A(t))$.
It is also immediate to see that 
\[S_1(A(t))=\frac{A(t)-\alpha}{t} \ . \]
Iterating $S_1$ and using the fact that $A(1)=A(0)(\alpha -1)$ one easily obtains that the ordinary generating function of the graded Euler characteristic of alternating graphs is 
\[A(t)=\frac{\alpha (1-t)}{1 - \alpha t  ( 1-t)} \ . \]
Furthermore, in some special cases, we can actually obtain a closed formula.

\begin{cor}
Let $A^*$ be a free $\bZ$-graded $R$-algebra with graded dimension $\alpha\in\bZ\left[ q,q^{-1} \right]$.
If~$\tI_n$ is the coherently oriented linear graph of length $n$, then 
\[ \chi_{\rm gr} (\tI_n;A^*) = \alpha (\alpha - 1)^n\text{, for }n>0\text{,}\]
and  $\chi_{\rm gr} (\tI_0;A^*) = \alpha$.
\end{cor}

From Remark~\ref{rem:seci} and Corollary~\ref{cor:ChiLinear} it follows that $(\alpha - 1)^{x}$ divides $\chi_{\rm gr}(\tL)$, where $x$ is the number of edges incident only to vertices which are either univalent or unstable.
This, similar divisibility properties, and the decomposition shown in Theorem~\ref{LinAlt}, seem to hint to the fact that~$\chi_{\rm gr}(\tL)$ is sensible  to ``dynamical properties'' of the graph.

\begin{q} 
Does it exist a closed formula for $\chi_{\rm gr}(\tL)$, with $\tL$ a linear graph, which features only dynamical data (e.g.~stable and unstable vertices, change of stability \emph{etc.}) and the polynomial of an alternating~$\chi_{\rm gr}(\tA_n)$ ($n$ also depending on dynamical data)?\end{q}

\section{Relations with simplicial homology}\label{sec:rel with simpl hom}

In this section we give a topological description of multipath cohomology. More specifically, we  see that~$\Hmu^*(\tG;R)$ is the ordinary  cohomology of a certain simplicial complex~$X(\tG)$ associated to $\tG$ -- cf.~Theorem~\ref{thm:multipath is simplicial}. For instance, this approach leads to a reinterpretation of the Mayer-Vietoris exact sequence for multipath cohomology, in topological terms. Furthermore, we also discuss which simplicial complexes can be realised as $X(\tG)$ for some $\tG$.

\subsection{Background material}
Recall that 
a \emph{regular CW-complex} is a CW-complex for which all the characteristic maps are homeomorphisms 
-- cf.~\cite[Section~3]{BJOR}. 
Recall also that the \emph{face poset}~$\mathcal{F}(X)$ of a CW-complex~$X$ is the poset on the set of cells of $X$, ordered by containment and augmented with a minimum element $\hat 0$ corresponding to the empty cell.
%
A poset $P$ with at least two elements is said to be a \emph{CW-poset} if it has a minimum~$\hat 0$,  and, for all $x\in P\setminus \hat 0$, the (geometric realisation\footnote{By geometric realisation of a poset we mean the geometric realisation of its order complex (i.e.~the abstract simplicial
complex whose faces are the totally ordered sub-posets of our poset), see~\cite[Chapter 9]{Kozlov}. } of the) interval $(\hat 0,x)\coloneqq \{z\in P\mid \hat 0<z<x\}$ is  a sphere.

\begin{example}\label{kjaf}
    A Boolean poset is a CW-poset. More generally, by \cite[Proposition~2.6 (b)]{BJOR}, every downward closed subposet of a Boolean poset is a CW-poset. By Remark~\ref{rem: PG squared and faithful}, the path poset $P(\tG)$  is a downward closed subposet of a Boolean poset, hence it is a CW-poset.
\end{example}


 A poset~$P$ is a CW-poset if, and only if, it is isomorphic to the face poset of a regular CW-complex -- see \cite[Proposition~3.1]{BJOR}.
As a consequence, for a digraph~$\tG$ there exists a regular CW-complex $X(\tG)$ whose face poset is isomorphic to $P(\tG)$. We can actually be more specific. Recall that 
an \emph{(augmented abstract) simplicial complex} $K$ on a vertex set $V$ is a simplicial complex augmented with a unique $(-1)$-simplex given by the empty set $\emptyset$. 

\begin{defn}\label{def:pathcomplx}
Given a digraph $\tG$, its \emph{multipath complex} is the augmented abstract simplicial complex~$X(\tG)$ on~$E(\tG)$ whose $k$-simplices are given by the multipaths in $\tG$ of length~$k-1$.
\end{defn}
Since each spanning sub-graph of a multipath is a multipath, it is clear that $X(\tG)$ is indeed an augmented abstract simplicial complex. Furthermore, we have the following observation.

\begin{rem}\label{rem:functor}
A  morphism of digraphs induces a morphism between the corresponding multipath complexes, which sends multipaths of length~$1$ to multipaths of length~$1$. 
It follows that for each  morphism of digraphs $\phi \colon\tG \to \tG'$ there is an associated  simplicial map $X(\phi)\colon X(\tG) \to X(\tG')$.
Clearly, we have that~$X(\mathrm{id}_{\tG})= \mathrm{id}_{X(\tG)}$, and that $X(\phi \circ \psi) =X(\phi)\circ X(\psi)$.
Hence, taking the multipath complex defines a functor~$X\colon \mathbf{ Digraph} \to \mathbf{ SimpComp}$ from the category of digraphs, and  morphisms of digraphs, to the category of augmented abstract simplicial complexes, and simplicial maps. 
\end{rem}

We can give a more explicit description of $X(\tG)$.

\begin{construction}
We construct the CW-complex which has~$X(\tG)$ as face poset by explicitly describing its cells and their gluing maps. 
We start by associating to the empty multipath of $\tG$, i.e.~the set of vertices of $\tG$, the empty simplex, i.e.~the $(-1)$-skeleton of $X(\tG)$. We build the complex $X(\tG)$ by attaching $n$-cells to the discrete set $X(\tG)^0\coloneqq E(\tG)$, i.e.~the $0$-skeleton of $X(\tG)$ is given by the set of edges of $\tG$. 

Suppose to have iteratively constructed the $(n-1)$-skeleton $X(\tG)^{(n-1)}$ of $X(\tG)$.
Each multipath~$\tH$ of length~$n+1$ -- i.e.~with $n+1$ edges -- is identified by edges $e_{i_0},\dots,e_{i_n}$ of $\tG$; hence, by $n+1$ points of  $X(\tG)^0$. We associate to $\tH$ the (abstract) $n$-dimensional simplex $\Delta_{\tH}^n \coloneqq [e_{i_0} , \cdots, e_{i_n}]$, which is an $n$-cell of $X(\tG)^n$. In this way, each multipath~$\tH$ of length~$n$ gives an $n$-cell of $X(\tG)$. Note that the multipath~$\tH$ can be obtained from exactly $n+1$ (sub-)multipaths of length $n$, say $\tH_0 ,\dots, \tH_n$, by adding an edge; more precisely, $\tH_{j}$ is the multipath of $\tG$ with edges $e_{i_1}, \dots ,\widehat{e_{i_j}}, ... ,e_{i_n}$, where $\widehat{e_{i_j}}$ indicates that the edge~$e_{i_j}$ is not counted. The multipaths $\tH_0 ,\dots, \tH_n$ correspond to $(n-1)$-cells of $X(\tG)$, and, moreover, the boundary $\partial(\Delta_\tH^n)$ of $\Delta_\tH^n$ can be identified with the union of its faces $\Delta_{\tH_j}^{n-1}$. The characteristic map of $\Delta_{\tH}^n$ is then defined by gluing each face $[e_{i_0}, \dots ,\widehat{e_{i_j}}, ... ,e_{i_n}]$ in $\partial(\Delta_\tH^n)$ with the corresponding $(n-1)$-cell in  $X(\tG)^{(n-1)}$.
More generally, we glue the simplices $\{ \Delta_{\tH}^n \}_{\lgt(\tH) =n+1}$ to $X(\tG)^{(n-1)}$ by identifying the facets of each $\Delta_{\tH}$ with the simplices $\Delta_{\tH_0} ,\dots, \Delta_{\tH_n}$, concluding the construction of the $n$-skeleton, hence of $X(\tG)$.
\end{construction}

Observe that maximal simplices in $X(\tG)$ correspond to maximal multipaths in $P(\tG)$. Therefore, to construct the geometric realisation of $X(\tG)$ we can proceed by finding the  maximal multipaths in $P(\tG)$, look at the intersection of each pair of maximal multipath, then glue the simplices associated to maximal multipaths along the faces determined by their intersections.

\begin{example}
In this example we explicitly describe the geometric realisation of the simplicial complex~$X(\tG)$, for \tG the H-shaped digraph illustrated in Figure~\ref{fig:CW-Hgraph}. Observe that we have five $0$-cells $E^{0}_{01}$, $E^{0}_{14}$, $E^{0}_{21}$, $E^{0}_{34}$, and  $E^{0}_{54}$, where the cell $E^{0}_{ij}$ corresponds to the edge $(v_i,v_j)$.
Each $1$-cell in $X(\tG)$ corresponds to a multipath of length two. Thus, there are precisely six $1$-cells $E^{1}_{01,14}$,  $E^{1}_{01,34}$,  $E^{1}_{01,54}$,  $E^{1}_{21,14}$, $E^{1}_{21,34}$, and $E^{1}_{21,54}$. The $1$-cell $E^{1}_{x,y}$ bounds the $0$-cells $E^0_x$ and $E^0_y$. 
    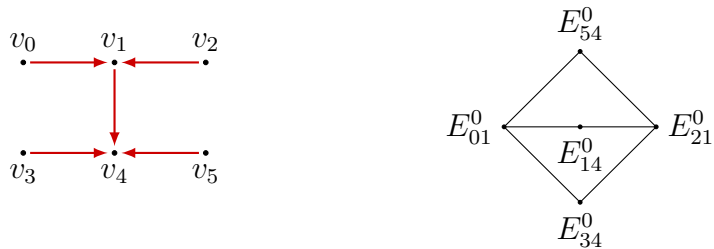
\begin{figure}[h]
    \begin{subfigure}[]{0.4\textwidth}
	\centering
		\begin{tikzpicture}[baseline=(current bounding box.center), scale =.6]
			\tikzstyle{point}=[circle,thick,draw=black,fill=black,inner sep=0pt,minimum width=2pt,minimum height=2pt]
			\tikzstyle{arc}=[shorten >= 8pt,shorten <= 8pt,->, thick]
			
			\node[above] (v0) at (-1,1) {$v_0$};
			\draw[fill] (-1,1)  circle (.05);
			\node[above] (v1) at (1,1) {$v_1$};
			\draw[fill] (1,1)  circle (.05);
			\node[above] (v2) at (3,1) {$v_{2}$};
			\draw[fill] (3,1)  circle (.05);
			\node[below] (v3) at (-1,-1) {$v_{3}$};
			\draw[fill] (-1,-1)  circle (.05);
			\node[below] (v4) at (1,-1) {$v_{4}$};
			\draw[fill] (1,-1)  circle (.05);
			\node[below] (v5) at (3,-1) {$v_{5}$};
			\draw[fill] (3,-1)  circle (.05);
			
			\draw[thick, bunired, -latex] (-0.85,1) -- (0.85,1);
			\draw[thick, bunired, -latex] (-0.85,-1) -- (0.85,-1);
			\draw[thick, bunired, -latex] (2.85,1) -- (1.15,1);
			\draw[thick, bunired, -latex] (2.85,-1) -- (1.15,-1);
			\draw[thick, bunired, -latex] (1,0.85) -- (1,-0.85);
			\node at (0,-2.5) {};
		\end{tikzpicture}
		\end{subfigure}
		\hspace{.05\textwidth}
		\begin{subfigure}[]{0.4\textwidth}
\begin{tikzpicture}[scale=0.5][thick]
		\draw[fill] (0,0)  circle (.05);
		\node[left] at (0,0)  {$E^{0}_{01}$};
		
		\draw[fill]  (2,0) circle (.05);
		\node[below] at (2,0)  {$E^{0}_{14}$};
		 
		\draw[fill] (4,0) circle (.05);
		\node[right] at (4,0)  {$E^{0}_{21}$};

		\draw[fill] (2,2) circle (.05);
		\node[above] at (2,2)  {$E^{0}_{54}$};

		\draw[fill] (2,-2) circle (.05);
		\node[below] at (2,-2)  {$E^{0}_{34}$};

		\draw (0,0) --  (2,0);
		\draw (4,0) --  (2,0);
		\draw (0,0) -- (2,2);
		\draw (0,0) -- (2,-2);
		\draw (4,0) -- (2,-2);
		\draw (4,0) -- (2,2);
		\end{tikzpicture}
		\end{subfigure}
			\caption{The $H$ digraph and its geometric realisation. }\label{fig:CW-Hgraph}
	\end{figure}
\end{example}

Let $K$ be an augmented abstract simplicial complex. The \emph{$n$-th reduced simplicial chain group} $\widetilde{C}_n(K;R)$  (with coefficients in $R$)  is the free $R$-module generated by all the $n$-simplices in $K$. 
We  assume, from now on, a linear ordering $ v_1 < v_2 < ... < v_k$ of the $0$-simplices of $K$ to be fixed.
Given an~$n$-simplex in $K$, say $\sigma = \{ v_{i_0}, ... , v_{i_n}\}$, we  denote $\sigma$ as $[v_{i_0},...,v_{i_n}]$, where $i_0 < i_2 < ... < i_n$, and define
\[ [v_{i_{s(0)}},...,v_{i_{s(n)}}] : = (-1)^{\sign (s)} [v_{i_0},...,v_{i_n}],\]
for each $s$ in the symmetric group over $\{ 0,...,n\}$.
The \emph{$n$-th simplicial differential} is defined as
\[ \delta_n\colon \widetilde{C}_n(K;R) \longrightarrow \widetilde{C}_{n-1}(K;R) : [v_{i_0},...,v_{i_n}] \longmapsto \sum_{j=0}^{n} (-1)^j  [v_{i_0},...,\widehat{v_{i_j}},...,v_{i_n}]\ ,   \]
where the hat $\widehat{x}$ indicates that $x$ is missing.
The \emph{$n$-th reduced simplicial co-chain group} is
\[ \widetilde{C}^n(K;R) = {\rm Hom}_R(\widetilde{C}_n(K;R); R) \]
and, given $f\in C^{n+1}(K;R)$, we define the \emph{co-boundary map} $\delta$ by setting 
$ \delta^n (f) = f(\delta_n(\sigma))$.

\subsection{Multipath cohomology is simplicial}

Let $X(\tG)$ be the multipath complex associated to the digraph~$\tG$, and let $\bK$ be a field.

\begin{thm}\label{thm:multipath is simplicial}
The multipath cohomology $\mathrm{H}_\mu^n(\tG;\bK)$ of $\tG$ is isomorphic to~$\widetilde{\rm H}^{n-1}(X(\tG);\bK)$, that is, the reduced (simplicial) cohomology of $X(\tG)$.
\end{thm}

Note that in the isomorphism between the multipath cohomology of $\tG$ and the simplicial cohomology of $X(\tG)$ there is a shift of degree one.

\begin{proof}
Let $X$ be a finite simplicial complex. 
The reduced simplicial cohomology cochain complex $\widetilde{C}^n (X;\bK)$ can be seen as the vector space over $\bK$ generated by the duals of all simplices (including the empty simplex) of $X$, with co-boundary map $\delta$ given by
\[ \delta (\sigma^*) = \sum_{\tau} (-1)^{\epsilon'(\tau,\sigma)} \tau^* \ ,\]
where $\tau$ ranges among all  simplices admitting $\sigma$ as a face, and $\epsilon'(\tau,\sigma)$ is $0$ or $1$ depending on whether or not the orientation of $\sigma$ matches with the orientation induced by $\tau$ -- for a more detailed construction, the reader can consult~\cite[Section 3.4.3]{Kozlov}.

By construction, the simplices of $X(\tG)$ correspond to multipaths in $\tG$; more precisely, the points of $X(\tG)^0$ are the edges of $\tG$, and  the multipath $\tH$ identified by the edges $e_{i_1}, ..., e_{i_n}$ corresponds to the simplex $[e_{i_1}, ..., e_{i_n}]$. The vector space $C^{n}_{\mu}(\tG;\bK)$ has one generator $b_{\tH}$ for each multipath $\tH$ of length $n$, and the differential is given by
\[ d(b_{\tH}) = \sum_{\tH' \supset \tH} (-1)^{\epsilon(\tH' , \tH)} b_{\tH'},\]
for a certain sign assignment $\epsilon$. 
It follows that the correspondence $b_{\tH} \mapsto [e_{i_1}, ..., e_{i_n}]$, where $\tH$ is identified by the edges $e_{i_1}, ..., e_{i_n}$, extends to an isomorphism of graded vector spaces which commutes with the differentials up to a sign. Note that a multipath of length $n$ corresponds to an $(n-1)$-dimensional simplex, which gives the shift in cohomological degree.

We conclude by observing that  $\epsilon'$ is a sign assignment on the face poset of $X(\tG)$, which is isomorphic to $P(\tG)$; in fact, the statement now follows from the uniqueness (up to isomorphism) of the sign assignment on the path posets \cite[Corollary~3.17]{primo}.
\end{proof}

An alternative way to associate to~$P(\tG)$ a simplicial complex $X'(\tG)$, having the same cohomology groups as $X(\tG)$, is to use the \emph{order complex}~$\Delta(P(\tG))$ -- see~\cite{Wachs}. 
It is a standard fact that, for a simplicial complex~$X$, the realisation of $\Delta(\mathcal{F}(X))$, that is the order complex of the face poset of $X$, is the barycentric subdivision of~$X$. Consequently, $X'(\tG)$ is the barycentric subdivision of $X(\tG)$. 
It is well-known that a simplicial complex and its barycentric subdivision have the same simplicial (co)homology.
As an immediate consequence we have the following corollary.

\begin{cor}
The  multipath cohomology $\mathrm{H}_\mu^*(\tG;\mathbb{K})$ is the reduced (simplicial) cohomology of the order complex~$\Delta(P(\tG))$ associated to $P(\tG)$.
\end{cor}

Observe that the isomorphism of Theorem~\ref{thm:multipath is simplicial} is well-behaved in a functorial sense:

\begin{rem}\label{rem:naturality}
By Remark~\ref{def:pathcomplx}, associating the simplicial complex $X(\tG)$ to a graph $\tG$ is functorial with respect to morphisms of digraphs.  Therefore, since taking cohomology groups is functorial with respect to simplicial maps, the isomorphism in cohomology provided by Theorem~\ref{thm:multipath is simplicial} induces a natural isomorphism  of functors $\eta\colon\Hmu^*(-;\bK) \Rightarrow \widetilde{\rm H}^{*-1}(-;\bK)\circ X$. More concretely, given a  morphism of digraphs $\phi\colon \tG\to\tG'$, we obtain the following square
     \[\begin{tikzcd}
\Hmu^*(\tG';\bK)\arrow[r, "\phi^*"]\arrow[d, "\eta_{\tG^\prime}"']& \Hmu^*(\tG;\bK)\arrow[d, "\eta_\tG"] \\
\widetilde{\rm H}^{*-1}(X(\tG');\bK)\arrow[r,  "\phi^*"' ]& \widetilde{\rm H}^{*-1}(X(\tG);\bK)
	\end{tikzcd}\]
that is commutative, with vertical arrows which are isomorphisms. This, in particular, extends the functoriality result \cite[Theorem~1.3]{primo} also to non-regular morphisms of digraphs.
\end{rem}

\subsection{Mayer-Vietoris from the topological viewpoint}\label{sec:MVtop}

In this subsection we reinterpret the Mayer-Vietoris-type theorem for multipath cohomology, that is Theorem~\ref{thm:MVposets}, using the simplicial description given in Theorem~\ref{thm:multipath is simplicial}.

Let $X$ be a simplicial complex and assume that there are  sub-complexes\footnote{A \emph{sub-complex} $Y$ of a simplicial complex $X$ is a subset of $X$ which is itself a simplicial complex.} $Y_1$ and $Y_2$ such that: 
\begin{enumerate}
\item $X = Y_1 \cup Y_2$;
\item their intersection of $Y_1$ and $Y_2$ is a sub-complex of $X$, $Y_1$, and $Y_2$;
\end{enumerate}
then, we have a long exact sequence in (reduced) simplicial cohomology
\[ \cdots \to \widetilde{\rm H}^i(X;\bK) \longrightarrow  \widetilde{\rm H}^i(Y_1;\bK) \oplus  \widetilde{\rm H}^i(Y_2; \bK) \longrightarrow  \widetilde{\rm H}^i(Y_1 \cap Y_2;\bK) \longrightarrow  \widetilde{\rm H}^{i+1} (Y_1 \cap Y_2;\bK) \to \cdots\]
called \emph{Mayer-Vietoris sequence} \cite[Chapter~2.2]{hatcher}. Intuitively, this sequence ``describes'' the cohomology of the space $X$ in terms of the cohomology of $Y_1$, $Y_2$, and $Y_1 \cap Y_2$.

In Theorem~\ref{thm:MVposets}, we obtained a similar sequence for path posets associated to regular morphisms of digraphs. We can actually use the correspondence between multipath and simplicial cohomologies to re-prove this result; we can interpret the path poset of a graph \tG as the face poset of a simplicial complex $X(\tG)$. Furthermore, a downward closed sub-poset $S$ of $P(\tG)$ corresponds to a (face poset of a) sub-complex of $X(\tG)$ -- the correspondence being described by taking the simplices of $X(\tG)$ given by the multipaths belonging to $S$.
By Theorem~\ref{thm:multipath is simplicial}, since we have $\Hmu^{*}(\tG;\bK) \cong \widetilde{\rm H}^{*-1}(X(\tG);\bK)$, we can replace the cohomology of spaces with the multipath cohomology of the corresponding graph with the corresponding multipath cohomology, obtaining the sequence in Theorem~\ref{thm:MVposets}.
To a more intimate level, and in a more abstract language, this is due to the fact that the gluing construction in Definition~\ref{nabla} represents the pushout in $\mathbf{Digraph}$ -- cf.~Remark~\ref{rem:nabla is pushout}.

\begin{rem}\label{rem:mvtop}
The Mayer-Vietoris long exact sequence in homology is classically obtained using the long exact sequence of the pair applied to homotopy pushouts \cite[Theorem~6.3]{Rotman1979AnIT}. Since the gluing construction in Definition~\ref{nabla} gives a pushout diagram of graphs, it is reasonable to think that a Mayer-Vietoris long exact sequence can be obtained also for multipath cohomology. In fact,
let $\tG, \tG_1, \tG_2$ as in Definition~\ref{nabla}; by Remark~\ref{rem:nabla is pushout}, the square
\begin{equation}\tag{$\spadesuit$}\label{eq:square path poset}
\begin{tikzcd}
 P(\tG) \arrow[r, "i_1"] \arrow[d, "i_2"'] & P(\tG_1) \arrow[d, "j_1"]\\
P(\tG_2)\arrow[r, "j_2"] & P(\tG_1) \nabla_{P(\tG)} P(\tG_2)
	\end{tikzcd}
		\end{equation}
	is a pushout square. Observe that also the square 
	\begin{equation}\tag{$\diamondsuit$}\label{eq:square simpl}
	    \begin{tikzcd}
 X(\tG) \arrow[r, "\iota_1"] \arrow[d, "\iota_2"'] & X(\tG_1) \arrow[d, "\jmath_1"]\\
X(\tG_2)\arrow[r, "\jmath_2"] &  X(\tG_1)\coprod_{X(\tG)} X(\tG_2)
	\end{tikzcd}
	\end{equation}
	is a pushout square of simplicial complexes, where $X\coprod_{Z} Y$ indicates the gluing of $X$ and~$Y$ along $Z\hookrightarrow X,Y$. 
	Furthermore, taking the face poset takes the square~\eqref{eq:square simpl} to the pushout square~\eqref{eq:square path poset}; i.e.~$\mathcal{F}(X(\tG_1) \coprod_{X(\tG)} X(\tG_2)) = P(\tG_1) \nabla_{P(\tG)} P(\tG_2)$. 
	As inclusions of simplicial complexes are cofibrations, the square~\eqref{eq:square simpl} is (equivalent to) a homotopy pushout square. Hence, Diagram~\eqref{eq:square simpl} gives rise to the Mayer-Vietoris long exact sequence in reduced cohomology. Naturality of Theorem~\ref{thm:multipath is simplicial}, as in  Remark~\ref{rem:naturality}, now gives the Mayer-Vietoris long exact sequence in multipath cohomology. In the light of the above discussion, we can think informally that ``inclusions as downward closed sub-posets are cofibrations in the category of posets, and the corresponding nablas are actually homotopy pushouts'' -- for a more precise treatment of the homotopy theory of posets, see for example~\cite{hha/1296223882}. 
\end{rem}

\subsection{Realisability of path posets and examples}

It is interesting to understand which simplicial complexes can be realised as $X(\tG)$ for some digraph~$\tG$. In this subsection we investigate this problem and some of its consequences.
First, we observe that all the spheres can be realised.

\begin{example}\label{ex:spheres}
Let $n\geq 1$ be a natural number.
The boundary of the $n$-dimensional standard simplex $\Delta^{n}$ is homeomorphic to the $(n-1)$-dimensional sphere $\bS^{n-1}$; we argue that $\partial \Delta^{n} $ is the simplicial complex~$ X(\tP_{n})$, where $\tP_{n}$ is the coherently oriented polygonal graph on $n$ vertices as illustrated in Figure~\ref{fig:poly}. By \cite[Section~2.3]{primo}, the path poset~$P(\tP_{n})$ is isomorphic (as posets) to the Boolean poset~$\mathbb{B}(n+1) $ with its maximum removed. Note that the $(n+1)$-Boolean poset~$\mathbb{B}(n+1) $ is the face poset~$\mathcal{F}(\Delta^n)$ of $\Delta^{n}$. Hence, it follows that $P(\tP_{n}) $ is the face poset of $\Delta^{n}$ minus its unique $(n+1)$-dimensional cell, which turn to be precisely the sphere~$\partial \Delta^n\cong \bS^{n-1}$. As a consequence, by Theorem~\ref{thm:multipath is simplicial}, we get
\[ 
\Hmu^i (\tP_n;\bK) \cong \widetilde{\rm H}^{i-1}(\bS^{n-1};\bK)
\cong \begin{cases} \bK & \text{ if } i = n, \\ 0 & \text{ otherwise}\end{cases} \]
the computations of the  cohomology of $\tP_n$, with coefficients in $\bK$, for every $n\in\bN$. Alternatively, this could have been obtained as a consequence of the isomorphism with Hochschild homology of $\bK$ -- cf.~\cite[Proposition~1.4]{primo}.
\end{example}

Another example is given by the dandelion graphs; their multipath complexes (or, better, their geometric realisations) are homotopy equivalent to wedges of $1$-dimensional spheres.

\begin{example}\label{ex:K-nm}
Consider the dandelion graph~$\tD_{n,m}$ 
in~Figure~\ref{fig:nmgraph}. We assume $nm>0$.
The multipath complex of $\tD_{n,m}$ is easily described as follows;
we have two sets $e_{1},...,e_{n}$ and $e'_{1},...,e'_{m}$ of  $0$-cells  in $X(\tD_{n,m})$, and each $e_i$ is joined with each $e'_j$ by a $1$-cell, and there are no other $1$- or higher cells.
Thus,  $X(\tD_{n,m})$ is the complete bipartite graph $K_{n,m}$. It follows that
\[ \Hmu^i (\tD_{n,m};\bK) \cong \widetilde{\H}^{i-1}(K_{n,m};\bK) \cong \begin{cases} \bK^{(n-1)(m-1)} & i =2, \\ 0 & \text{otherwise}.\end{cases} \]
In fact, since $K_{n,m}$ is a connected $1$-dimensional complex, we have that: 
$\widetilde{\H}^{0}(K_{n,m};\bK) = 0$, and ${\rm rank}(\widetilde{\H}^{1}(K_{n,m};\bK) )$ equals the number of edges minus the number of vertices plus one, which yields exactly $(n-1)(m-1)$.
\end{example}

It is yet not clear whether multipath cohomology can be supported in different degrees. In the next example we show that this is  also possible.

\begin{example}\label{ex:squarediag}
Consider the graph ${\tt Q}$ to be the coherently oriented polygon $\tP_3$ with a diagonal, as illustrated in Figure~\ref{fig:diag square}. 
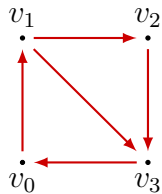
\begin{figure}[h]
	\begin{tikzpicture}[scale=0.55][baseline=(current bounding box.center)]
		\tikzstyle{point}=[circle,thick,draw=black,fill=black,inner sep=0pt,minimum width=2pt,minimum height=2pt]
		\tikzstyle{arc}=[shorten >= 8pt,shorten <= 8pt,->, thick]

		\node  (v0) at (0,0) {};
		\node[below] at (0,0) {$v_0$};
		\draw[fill] (0,0)  circle (.05);
		\node  (v1) at (0,3) { };
		\node[above]  at (0,3) {$v_1$};
		\draw[fill] (0,3)  circle (.05);
		\node  (v2) at (3,3) { };
		\node[above] at (3,3) {$v_{2}$};
		\draw[fill] (3,3)  circle (.05);
		\node  (v3) at (3,0) { };
		\node[below]  at (3,0) {$v_{3}$};
		\draw[fill] (3,0)  circle (.05);
		
		\draw[thick, bunired, -latex] (v0) -- (v1);
		\draw[thick, bunired, -latex] (v1) -- (v2);
		\draw[thick, bunired, -latex] (v2) -- (v3);
		\draw[thick, bunired, -latex] (v3) -- (v0);
		\draw[thick, bunired, -latex] (v1) -- (v3);
	\end{tikzpicture}
	\caption{The coherently oriented polygon $\tP_3$ with a diagonal.}
	\label{fig:diag square}
\end{figure}
The path poset associated to ${\tt Q}$ has multipaths of length at most~$3$, corresponding to having at most $2$-simplices in the realisation of $X({\tt Q})$. The realisation of the multipath complex associated to $\tP_3$ is isomorphic to the $2$-sphere $\bS^2$, and adding the diagonal results in adding two $1$-cells to $\bS^2$. A depiction of the geometric realisation of $X({\tt Q})$ is given in Figure~\ref{fig:CW-square-graph}; we have decorated the vertices in the realisation corresponding to the edges~$(v_1,v_2)$ and $(v_1,v_3)$ of ${\tt Q}$, all the others being interchangeable in the realisation. 
 \begin{figure}[h]
	\centering
		\begin{tikzpicture}[thick]
		\draw[fill] (0,0)  circle (.05);
		\node[below] at (0,0)  {$[v_1,v_2]$};

		\draw[fill] (2.7,0) circle (.05);

		\draw[fill] (1.2,2.5) circle (.05);

		\draw (0,0) --  (2.7,0);
		\draw (0,0) -- (1.2,2.5);
		\draw (2.7,0) -- (1.2,2.5);
		
	    \draw[fill=green,opacity=0.15] (0,0) -- (2.7,0) -- (1.2,2.5) -- (0,0) --cycle; 	
		
		\draw[fill] (2.8,0.8) circle (.05);

		\draw[dotted] (2.8,0.8) --  (0,0);

		\draw (2.7,0) -- (2.8,0.8);
		\draw (2.8,0.8) -- (1.2,2.5);
		
	    \draw[fill=green,opacity=0.25] (2.8,0.8) -- (2.7,0) -- (1.2,2.5) -- (2.8,0.8) --cycle;
	    
		\draw[fill] (4.5,1.7) circle (.05);
		\node[right] at (4.5,1.7)  {$[v_1,v_3]$};

		\draw (4.5,1.7) -- (2.8,0.8);
		\draw (1.2,2.5) -- (4.5,1.7);
		\end{tikzpicture}
		\caption{The geometric realisation of $X(\tG)$, where $\tG$ is the graph in Figure~\ref{fig:diag square}. In green is illustrated the boundary of a tetrahedron.}\label{fig:CW-square-graph}
	\end{figure}
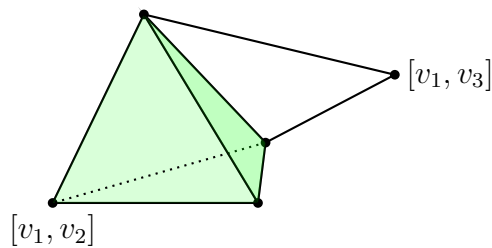
The simplicial complex~$X(\tG)$ is {homotopically} equivalent to the wedge of $\bS^2$ with a copy of $\bS^1$, showing that the multipath cohomology
\[ \Hmu^*({\tt Q};\bK) \cong \begin{cases} \bK & \text{ for } n = 2,\\ \bK & \text{ for } n = 3,\\ 0 & \text{ otherwise.}\\  \end{cases}\]
 of ${\tt Q}$  is non-zero in degrees $2$ and $3$ by Theorem~\ref{thm:multipath is simplicial}.
\end{example}

Let $\tG'$ be a digraph with a univalent vertex $w$. In Lemma~\ref{lem:shift homology construction}, we have proved that there exists a graph $\tG$, obtained by gluing a linear sink or source to $w$ (cf.~Figure~\ref{fig:y}), such that $\Hmu^*(\tG;\bK) \cong \Hmu^{*+1}(\tG';\bK)$.
Classically, this is the property of the suspension\footnote{the topological space $\Sigma X$ obtained from $X$ by taking the cylinder $X\times [0,1]$, and collapsing each of the faces $X\times \{0\}$ and $X\times \{ 1\}$ to a point. Alternatively, the suspension can be seen as two copies of the cone $(X\times [0,1])/(x,1)\sim (y,1)$ glued along $X\times \{ 0 \}$.}  $\Sigma X$. 
In the proof of the next proposition we will see that these two constructions are related. Let $\bigvee^{k} \mathbb{S}^{n}$ be the wedge of $k$ $n$-dimensional spheres.
\begin{prop}\label{prop:wedgsofspheres}
For all $k, n \in \mathbb{N}$, there exists a graph $\tG_{k,n}$ such that
\[ \bigvee^{k} \mathbb{S}^{n} \simeq | X(\tG_{k,n}) | \]
where $| X(\tG_{k,n}) |$ is the geometric realisation of the (abstract) simplicial complex $X(\tG_{k,n})$.
\end{prop}

\begin{proof}
Let $\tD_{n,m}$ be a dandelion graph, cf.~Definition~\ref{dandelion} and Figure~\ref{fig:nmgraph}.
The geometric realisation of $\tG_{k,0} \coloneqq \tD_{k+1,0}$ consists of $k+1$ points, which can be seen as a wedge of $k$ $0$-spheres. 
In Example~\ref{ex:K-nm}, we have shown that $X(\tD_{k+1,2})$ for $k>0$ is homotopy equivalent to a wedge of $k$ $1$-spheres. In order to get a wedge $\bigvee^{k} \mathbb{S}^{n}$  of higher dimensional spheres, it is enough to iteratively glue sinks or sources to $\tG_{k,1} \coloneqq \tD_{k+1,2}$. 

More precisely, we fix $k$ and proceed by induction on $n$. 
We assume to have already constructed a graph $\tG_{k,n}$, whose associated multipath complex is homotopic to a wedge of $k$ $n$-dimensional spheres, and that it has a univalent vertex, say $v$.
Now, we glue a source or a sink, depending on whether $v  = s(e)$ or $v = t(e)$, for some edge $e\in E(\tG_{k,n})$ -- cf.~Lemma~\ref{lem:shift homology construction} and Figure~\ref{fig:y}. Define $\tG_{k,n+1}$ to be the graph obtained this way. Note that $\tG_{k,n+1}$ has at least two univalent vertices.
Now, by Theorem~\ref{finedimondo} applied to~$v$, the path poset of $\tG_{k,n+1}$ is the gluing of two copies of $\Cone P(\tG_{k,n})$ (which are $P((\tG_{k,n})_v^{(1)})$ and $P((\tG_{k,n})_v^{(2)})$, in the notation of Theorem~\ref{finedimondo}) glued along $P(\tG_{k,n})$. This can be illustrated using the following pushout diagram
\[\begin{tikzcd}
P(\tG_{k,n})\arrow[r, "\jmath_1"]\arrow[d, "\jmath_2"']& \Cone P(\tG_{k,n}) \arrow[d, "\imath_1"] \\
\Cone P(\tG_{k,n})\arrow[r,  "\imath_2"' ]& \tG_{k,n+1} 
	\end{tikzcd} \]
By passing to the geometric realisations of  multipath complexes (see also Remark~\ref{rem:mvtop}), we obtain the pushout diagram 
	\newcommand{\cone}{\rm Cone}
	\[
 \begin{tikzcd}
\vert X(\tG_{k,n}) \vert \simeq \bigvee^{k} \mathbb{S}^{n} \arrow[r, "\iota_1"] \arrow[d, "\iota_2"'] & \cone (\bigvee^{k} \mathbb{S}^{n})\simeq  * \arrow[d, "\jmath_1"]\\ *\simeq\cone (\bigvee^{k} \mathbb{S}^{n})\arrow[r, "\jmath_2"] &  \Sigma(\bigvee^{k} \mathbb{S}^{n})\simeq |X(\tG_{k,n+1})|
	\end{tikzcd}
	\]
of topological spaces, where $*$ denotes the one point space. As the suspension of a wedge on~$k$ $n$-spheres is homotopy equivalent to a wedge of $k$ $(n+1)$-spheres, the statement follows. 
\end{proof}

The proposition says that  wedges of spheres of the same dimension can be realised, up to homotopy, as multipath complexes of certain digraphs. It is therefore natural to ask whether wedges of spheres of different dimensions can also be realised as (geometric realisations of) multipath complexes. We have seen, in Example~\ref{ex:squarediag}, that the wedge $\mathbb{S}^1\vee \mathbb{S}^2$ can be realised, but it is not clear whether all wedges of spheres can be:
\begin{q}\label{q:realisability1}
Can we realise, up to homotopy, all wedges of spheres? For example, can we realise a space homotopic to $\bS^1 \vee ... \vee \bS^1 \vee \bS^0 \vee ...\vee \bS^0$?
\end{q}
We observe here that, if we allow the graphs to be disconnected, then it is possible to obtain also the topological join $X(\tG) \ast X(\tG')$ of $X(\tG)$ and $X(\tG')$; concretely, this can be realised as $X(\tG \sqcup \tG')$ -- compare~\cite[Definition 2.16]{Kozlov}, and Remark~\ref{DisjointUnion}.
Note that the join of simplicial complexes commutes with the geometric realisation~\cite[Equation (2.9)]{Kozlov}.
In particular, we can realise the suspension $\Sigma X(\tG)$ of $X(\tG)$ also as the multipath complex $X(\tG \sqcup \tA_{2})$~\cite[Example 2.32.(1)]{Kozlov}. Realising  joins allows us to obtain more complicated wedges of spheres giving a partial answer to the question above. 
\begin{example}
The join commutes, up to homotopy, with the wedge operation, and it is well-known that $\mathbb{S}^{n}\ast\mathbb{S}^{m} \simeq \mathbb{S}^{n+m+1}$ -- cf.~\cite{JoinandWedge}.
An immediate consequence of these facts is that
\[ \bigvee^{k}_{i=1} \mathbb{S}^{n_i} \ast \bigvee^{h}_{j=1} \mathbb{S}^{m_j} \simeq  \bigvee^{k}_{i=1}\bigvee^{h}_{j=1} \mathbb{S}^{n_i+m_j + 1} \ . \]
For instance, we have shown in Example~\ref{ex:squarediag} that we can realise $\bS^2 \vee \bS^1$ as $|X({\tt Q})|$, where ${\tt Q}$ is the square $\tP_3$ with a diagonal. Therefore, we have that
\[ \bS^{5} \vee \bS^{4} \vee \bS^{4} \vee \bS^{3} \simeq (\bS^2 \vee \bS^1)\ast (\bS^2 \vee \bS^1)\simeq \vert X({\tt Q} \sqcup {\tt Q}) \vert \ . \]
\end{example}

{Since all spaces we are able to realise are, up to homotopy, wedges of spheres, the following question arises naturally;}

\begin{q}\label{q:realisability2}
Can we realise spaces which are not homotopic to wedges of spheres, e.g.~the real projective spaces?
\end{q}

Before further discussing the realisability of simplicial complexes via path posets, we want to extend the category of digraphs.
Recall that a \emph{multigraph}, or \emph{quiver}, is $4$-uple $(V,E,s,t)$ where $V$ and~$E$ are finite sets, whose elements are called \emph{vertices} and \emph{edges} respectively,
while $ s,t \colon E \longrightarrow V $
are the \emph{source} and \emph{target} functions. To avoid self-loops we require $t(e)\neq s(e)$, for all $e\in E$.
Clearly, all digraphs can be seen as multigraphs.
The definition of path poset and multipath cohomology extends verbatim to multigraphs. Similarly, we have an isomorphism between  multipath cohomology and simplicial homology of $X(\tG)$.

\begin{figure}
    \centering
    \begin{tikzpicture}[scale=0.6][thick]
    \draw[fill] (-2,1.5) circle (.075) -- (2,-1.5) circle (.075);
    \draw[fill] (-2,-1.5) circle (.075) -- (2,1.5) circle (.075);
    \draw (-2,1.5) -- (-2,-1.5);
    \draw (2,1.5) -- (2,-1.5);
    \draw[fill] (0,0) circle (.075);
    
    \node[above] at (-2,1.5) {$e_1$};
    \node[above] at (2,1.5) {$e_4$};
    \node[below] at (-2,-1.5) {$e_2$};
    \node[below] at (2,-1.5) {$e_3$};
    \node[above] at (0,0) {$e_5$};
    \end{tikzpicture}
    \caption{A $1$-dimensional complex $X$ whose face poset is not a path poset.}
    \label{fig:face no path}
\end{figure}
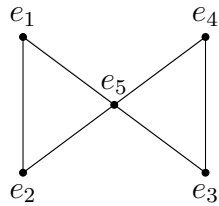

\begin{prop}\label{prop:facenopathposet}
There are simplicial complexes whose face poset cannot be realised as the path poset of any multigraph. In other words, the functor $X\colon \mathbf{Digraph} \to \mathbf{SimpComp}$  (cf.~Remark~\ref{rem:functor}), or better its extesion to multigraphs, is not (essentially\footnote{That is up to (simplicial, in our case) isomorphism.}) surjective on objects.
\end{prop}
\begin{proof}
Consider Figure~\ref{fig:face no path}, it contains the picture of a geometric realisation of the simplicial complex, say $X$.
The face poset of $X$ consists of: the empty simplex $\emptyset$, five $0$-simplices, and six $1$-simplices.

Assume, by the means of contradiction, that there exists a multigraph $\tG$ such that $X = X(\tG)$.
It follows that $\tG$ has five edges, six multipaths of length two, and no multipaths of higher length.
Note that we cannot read the number of vertices form the path poset (and thence from $X(\tG)$). Moreover, adding a vertex with no edges does not change the path poset.
Therefore, without loss of generality, we might assume that $\tG$ has no connected components which are vertices.

We claim that the edges $e_1$, $e_2$, and $e_5$, form a coherently oriented triangle in $\tG$. We start by noticing that the sub-graph $\tT$ of $\tG$ given by the edges $e_1$, $e_2$, and $e_5$ is connected.
Either all pairs of edges among $e_1$, $e_2$, and $e_5$ share a vertex, in which case $\tT$ is connected, or there are $e_i$ and $e_j$, with $i\neq j$ and $i,j\in \{ 1,2,5\}$, do not share any vertex. In the latter case, both must share at least one vertex with $e_k$,  $k\neq i,j$ and $k\in \{ 1,2,5\}$. If not, we would have a multipath of length $3$ given by $e_1,e_2,e_5$. It follows that $\tT$ is connected.
The path poset of $\tT$ is (isomorphic to) the sub-poset  of $P(\tG)$ given by multipaths with edges $e_1$, $e_2$, or $e_5$.
Pasting all the pieces together, we have that  $\tT$ is a connected graph  with $3$ edges whose path poset is a Boolean poset minus its maximum; hence, $\tT$ is a coherently oriented triangle by~\cite[Proposition~2.37]{primo}.

The same reasoning as above proves that the edges $e_3$, $e_4$, and $e_5$, also form a coherently oriented triangle in $\tG$. It follows that $\tG$ must consist of two coherently oriented triangles glued at least along one edge (namely $e_5$). 
Now we can enumerate all possible $\tG$'s; these are illustrated, up to orientation reversal and exchanging in the roles of $e_1$ and $e_2$, and of $e_3$ and $e_4$, in Figure~\ref{fig:possiblegraphs}.
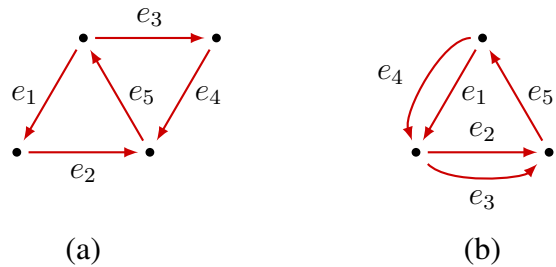
\begin{figure}[H]
    \centering
    
    \begin{tikzpicture}[thick, scale = 1.75]
     \node (a) at (0,0) {};
     \node (b) at (1,0) {};
     \node (c) at (.5,0.866) {};
     \node (d) at (1.5,.866) {};
     
     \draw[fill] (a) circle (0.025);
     \draw[fill] (b) circle (0.025);
     \draw[fill] (c) circle (0.025);
     \draw[fill] (d) circle (0.025);
     
     \node[left] at (0.25,.433) {$e_1$};
     \node[right] at (.75,.433) {$e_5$};
     \node[below] at (0.5,0) {$e_2$};
     \node[right] at (1.25,.433) {$e_4$};
     \node[above] at (1,.866) {$e_3$};
     
     \node at (.5,-.75) {(a)};
     
     \draw[-latex, bunired] (a) -- (b);
     \draw[-latex, bunired] (b) -- (c);
     \draw[-latex, bunired] (c) -- (a);
     \draw[-latex, bunired] (d) -- (b);
     \draw[-latex, bunired] (c) -- (d);
     
     \begin{scope}[shift = {+(3,0)}]
     \node (a) at (0,0) {};
     \node (b) at (1,0) {};
     \node (c) at (.5,0.866) {};
     \node (d) at (1.5,.866) {};
     
     \draw[fill] (a) circle (0.025);
     \draw[fill] (b) circle (0.025);
     \draw[fill] (c) circle (0.025);
     
     \node[right] at (0.25,.433) {$e_1$};
     \node[right] at (.75,.433) {$e_5$};
     \node[above] at (0.5,0) {$e_2$};
     \node[above left] at (0,.433) {$e_4$};
     \node[below] at (.5,-.2) {$e_3$};
     
     \node at (.5,-.75) {(b)};
     
     \draw[-latex, bunired] (a) -- (b);
     \draw[-latex, bunired] (b) -- (c);
     \draw[-latex, bunired] (c) -- (a);
     \draw[-latex, bunired] (a) .. controls +(.25,-.25) and +(-.2,-.2) .. (b);
     \draw[-latex, bunired] (c) .. controls +(-.3,0) and +(-.1,.3) ..  (a);
     \end{scope}
     \end{tikzpicture}
    \caption{Possible configurations for the graph $\tG$.}
    \label{fig:possiblegraphs}
\end{figure}
In case (a) we have two multipaths of length $3$, while in case (b) the $0$-cells corresponding to~$e_4$ and $e_3$ must be in the boundary of three $1$-cells in $X(\tG) = X$. In either case we get a contradiction, and the statement follows.
\end{proof}

Note that, while the face poset of the simplicial complex in Figure~\ref{fig:face no path} cannot be realised as a path poset, we can realise the face poset of an homotopy equivalent space as a path poset (e.g.~$P(\tD_{2,3})$, cf.~Example~\ref{ex:K-nm}).

\bibliographystyle{alpha}
\bibliography{bibliography}

\end{document}